\documentclass[a4paper, 12pt]{amsart}
\usepackage{amsfonts, amsthm, amssymb, amsmath, stmaryrd}
\usepackage{mathrsfs,array}
\usepackage{xy}
\usepackage{verbatim}
\usepackage{amscd} 
\usepackage{tikz}
\usepackage{tikz-cd}
\usetikzlibrary{matrix,arrows,decorations.pathmorphing}
\input xy
\xyoption{all}
\usepackage[unicode]{hyperref}

\newtheorem{thm}{Theorem}[subsection]

\newtheorem{cor}[thm]{Corollary}

\newtheorem{lem}[thm]{Lemma} 
\newtheorem{prop}[thm]{Proposition}
 \theoremstyle{definition}
 \theoremstyle{definition}
\newtheorem{defn}[thm]{Definition} \theoremstyle{remark}
\newtheorem{rem}[thm]{\bf Remark}
\newtheorem{remark}[thm]{\bf Remark}

\newtheorem{assumption}[thm]{\bf Assumption}

\DeclareMathOperator{\Hom}{Hom}
\DeclareMathOperator{\End}{End}
\DeclareMathOperator{\Ext}{Ext}

\DeclareMathOperator{\Ker}{ker}

\DeclareMathOperator{\modd}{mod}

\DeclareMathOperator{\Spec}{Spec}

\DeclareMathOperator{\gr}{gr}
\DeclareMathOperator{\Gal}{Gal}

\DeclareMathOperator{\Frob}{Frob}

\DeclareMathOperator{\Sym}{Sym}

\DeclareMathOperator{\Aut}{Aut}
\DeclareMathOperator{\ad}{ad}

\def\cont{\mathrm{cont}}

\def\dR{\mathrm{dR}}

\def\Ind{\mathrm{Ind}}

\def\ps{\mathrm{ps}}
\def\loc{\mathrm{loc}}
\def\ord{\mathrm{ord}}

\def\Irr{\mathrm{Irr}}
\def\sm{\mathrm{sm}}
\def\Ban{\mathrm{Ban}}
\def\adm{\mathrm{adm}}
\def\Mod{\mathrm{Mod}}

\newcommand{\Z}{\mathbb{Z}}
\newcommand{\F}{\mathbb{F}}
\newcommand{\Q}{\mathbb{Q}}
\newcommand{\R}{\mathbb{R}}
\newcommand{\T}{\mathbb{T}}
\newcommand{\A}{\mathbb{A}}
\newcommand{\bC}{\mathbb{C}}

\newcommand{\tr}{\mathrm{tr}}
\newcommand{\GL}{\mathrm{GL}}

\newcommand{\Sp}{\mathrm{Sp}}

\newcommand{\cO}{\mathcal{O}}

\newcommand{\kq}{\mathfrak{q}}
\newcommand{\kp}{\mathfrak{p}}

\newcommand{\kQ}{\mathfrak{Q}}
\newcommand{\kP}{\mathfrak{P}}
\newcommand{\kC}{\mathfrak{C}}
\newcommand{\kB}{\mathfrak{B}}

\newcommand{\ka}{\mathfrak{a}}
\newcommand{\kb}{\mathfrak{b}}
\newcommand{\km}{\mathfrak{m}}
\newcommand{\mm}{\mathrm{m}}
\newcommand{\DAi}{(D\otimes_F \A_F^\infty)^\times}
\newcommand{\AFi}{(\A_F^\infty)^\times}

\newcommand{\RNum}[1]{\uppercase\expandafter{\romannumeral #1\relax}}

\newcommand{\overbar}[1]{\mkern 1.5mu\overline{\mkern-1.5mu#1\mkern-1.5mu}\mkern 1.5mu}

\begin{document}
	\title{On the Fontaine-Mazur conjecture for $p=3$}
	\author{Xinyao Zhang}
	\address{Graduate School of Mathematical Sciences, University of Tokyo, Komaba, Meguro, Tokyo 153-8914, Japan}
	\email{zhangxy96@g.ecc.u-tokyo.ac.jp}
	\begin{abstract}
		%In this article, we prove the exceptional case at (p=3) of the regular two-dimensional Fontaine–Mazur conjecture over (\mathbb Q). Together with the work of Skinner–Wiles, Kisin, Hu–Tan, Khare–Wintenberger and Pan, this completes the regular case of the conjecture for all odd primes. The key new input is a potential pro-modularity theorem for residually reducible pseudo-deformation rings: after a suitable totally real abelian base change, the large-dimensional components are pro-modular.
		
		%In this article, we prove the remaining open cases of the Fontaine-Mazur conjecture for two-dimensional regular Galois representations of $\Gal(\overline{\Q}/\Q)$ when $p=3$, thereby completing the regular case of the conjecture for all odd primes. The main new input is a potential big $R=\T$ theorem for global pseudo-deformation spaces, based on the recent progress in $p$-adic Langlands correspondence, Galois deformation theory and a componentwise construction of pro-modular primes.
		
		We prove the remaining \(p=3\) cases of the regular two-dimensional
		Fontaine-Mazur conjecture over \(\mathbb{Q}\), thereby completing the regular
		case for all odd primes. Our main new input is a potential big
		\(R=\mathbb{T}\) theorem for large-dimensional components of global
		pseudo-deformation spaces, valid for all odd primes. After a single abelian
		base change, propagation of pro-modularity and induction on partial
		ordinariness reduce the componentwise argument to the case of globally irreducible ordinary points. A characteristic-zero Greenberg-Wiles dimension estimate supplies the additional ordinary deformation-theoretic input needed in the exceptional case. We combine these arguments with small-prime \(p\)-adic Langlands correspondence, local deformation theory and patching.
		
		%Our result is a sequel to Pan's work, based on some recent progress on $p$-adic Langlands correspondence, Galois deformation theory and a potential pro-modularity result. %%Also, we prove some finiteness results between the local pseudo-deformation ring at odd prime $p$ and the global pseudo-deformation ring, which are analogous to a theorem of Calegari-Allen.  
	\end{abstract}
	
	\maketitle
	
	\tableofcontents
	
	\section{Introduction}
	 
	In a recent breakthrough \cite{pan2022fontaine}, Pan has proved the following theorem:
	
	\begin{thm}\cite[Theorem 1.0.4]{pan2022fontaine}\label{1.0.1}
		Let $p$ be an odd prime number and 
		\[\rho:\Gal(\bar{\Q}/\Q)\to\GL_2(\overbar{\Q_p})\]
		be a continuous, irreducible representation such that
		\begin{itemize}
			\item $\rho$ is only ramified at finitely many places,
			\item $\rho|_{G_{\Q_p}}$ is de Rham of distinct Hodge-Tate weights,
			\item $\rho$ is odd, i.e. $\det\rho(c)=-1$ for any complex conjugation $c\in \Gal(\bar{\Q}/\Q)$.
			\item If $p=3$, the semi-simplification of $\bar{\rho}|_{G_{\Q_p}}$ is not of the form $\eta\oplus\eta\omega$ for some character $\eta$, where $\omega$ denotes the $\mod p$ cyclotomic character.
		\end{itemize}
		Then $\rho$ arises from a cuspidal eigenform up to twist.
	\end{thm}
	In this paper we remove the last assumption in Pan's theorem. Combined with the work of \cite{skinnerwiles01}, \cite{hutan15}, \cite{kisin2009fontaine}, \cite{pan2022fontaine} and \cite{Khare_2009a}, this completes the regular two-dimensional Fontaine-Mazur conjecture over $\Q$ for all odd primes.
	
	\begin{thm}[Theorem \ref{mainthm}, Theorem \ref{mainthm2}]\label{1.0.2}
		Let $p$ be an odd prime number and 
		\[\rho:\Gal(\bar{\Q}/\Q)\to\GL_2(\overbar{\Q_p})\]
		be a continuous, irreducible representation such that
		\begin{itemize}
			\item $\rho$ is only ramified at finitely many places,
			\item $\rho|_{G_{\Q_p}}$ is de Rham of distinct Hodge-Tate weights,
			\item $\rho$ is odd.
		\end{itemize}
		Then $\rho$ arises from a cuspidal eigenform up to twist.
	\end{thm}
	
	The main difficulty at $p=3$ lies in the residually reducible case, where the exceptional local type behaves differently both in the $p$-adic local Langlands correspondence and in Galois deformation theory. Recent results of Pašk\=unas-Tung \cite{Pa_k_nas_2021}, B\"ockle-Iyengar-Pašk\=unas \cite{B_ckle_2023} and B\"ockle-Juschka \cite{Bockle_2023} provide the necessary local inputs, but do not by themselves make Pan's global strategy applicable in this case. We develop a different componentwise strategy, combining propagation of pro-modularity, induction on partial ordinariness and dimension estimates at characteristic-zero ordinary points. This yields a potential big $R=\T$ theorem for large-dimensional components, valid for all odd primes.
	
	We now introduce the setting of our main deformation-theoretic result. Let $p$ be an odd prime and let $F$ be an abelian totally real field in which $p$ splits completely. Let $\Sigma_p$ be the set of places of $F$ above $p$, and let $\Sigma$ be a finite set of finite places of $F$ containing $\Sigma_p$. Let $\bar{\chi}: G_{F, \Sigma} \to \F^\times$ be a continuous, totally odd, almost unramified character which can be extended to a character of $G_\Q$. Then $\bar T=1+\bar{\chi}$ defines a pseudo-representation. Further assume that for any $v|p$, $\bar{\chi}|_{G_{F_v}} \ne \mathbf{1}$. Let $\chi$ be a de Rham lifting of $\bar{\chi}$. 
	
	\begin{thm}[Theorem \ref{potential R=T}, Proposition \ref{equi of T}]\label{1.0.3}
		There exists a finite extension \(F_1/F\), with \(F_1/\mathbb Q\) abelian and totally real of even degree, such that \(p\) splits completely in \(F_1\) and the following holds. Let \(R^{\mathrm{ps}}_{F_1}\) be the universal pseudo-deformation ring parametrizing pseudo-deformations of \(\bar T|_{G_{F_1}}\), unramified outside the places above \(\Sigma\), with fixed determinant \(\chi|_{G_{F_1}}\). Then every irreducible component of \(R^{\mathrm{ps}}_{F_1}\) of dimension at least \(1+2[F_1:\mathbb Q]\) is pro-modular and has dimension \(1+2[F_1:\mathbb Q]\). Moreover, the corresponding big Hecke algebra over \(F_1\), localized at the residual system \(\bar T|_{G_{F_1}}\), is equidimensional of dimension \(1+2[F_1:\mathbb Q]\) whenever this localization is non-zero.
		
		%There exists an abelian totally real extension $F_1/F$ of even degree, in which $p$ splits completely, such that any irreducible component of the universal pseudo-deformation ring  $R_{F_1}^\ps$ (parametrizing all pseudo-deformations of $T|_{G_{F_1}}$ unramified outside places lying above $\Sigma$ with fixed determinant $\chi|_{G_{F_1}}$) of dimension at least $1+ 2[F_1: \mathbb{Q}] $, it is pro-modular of dimension $1+2[F_1:\mathbb{Q}]$. The corresponding big Hecke algebra over $F_1$ is equidimensional of dimension $1+2[F_1:\mathbb{Q}]$.
	\end{thm}
	
	\begin{rem}
		Actually, we can show that if an irreducible component of the global pseudo-deformation ring contains a one-dimensional point corresponding to an irreducible Galois representation, then it is of dimension at least $1+2d$, where $d$ is the degree of the number field. See \cite[Proposition 2.4.2]{ZhaZariskiDensity}.
		
		The case $\bar{\chi}|_{G_{F_v}} = \mathbf{1}$ for any $v|p$ is treated in \cite[Section 4.2]{ZhaZariskiDensity}. In this paper, we also prove a potential big $R=\T$ theorem in the residually irreducible case. See Proposition \ref{big R=T in irr}.
	\end{rem}
	
	As Pan explains in \cite[page 1032]{pan2022fontaine}, the strategy for the Fontaine-Mazur conjecture combines a big $R=\T$ statement for completed cohomology with a classicality theorem. For this purpose, it suffices to identify the relevant irreducible components of the global pseudo-deformation space with components of the big Hecke algebra. Theorem \ref{1.0.3} gives a uniform potential form of this first step: one abelian base change is chosen before the componentwise argument, and every component of the expected or larger dimension over that field is pro-modular. No further base change depending on a chosen potentially nice prime is needed. This uniformity is also useful for the applications discussed below.
	
	%The main difficulty of the Fontaine-Mazur conjecture at $p=3$ is the residually reducible case. To outline the proof of Theorem \ref{1.0.2} in this case, we suppose the semisimplification of $\bar{\rho}$ is (up to twist) isomorphic to $1 \oplus \bar{\chi}$ for some continuous odd almost unramified character $\bar{\chi} : G_\Q \to \F^\times$, where $\F$ is a finite field of characteristic $p$.
	
	%Our strategy is a refinement of Pan's. Precisely, we establish a potential pro-modularity result, rather than just identify some irreducible component of the universal deformation ring as in the big Hecke algebra, so that our strategy works for all odd primes.

	We now outline the proof and explain where it differs from Pan's argument. Following \cite{Pa_k_nas_2021}, we call the local type $p=3$ and $\bar\chi|_{G_{F_v}}=\omega^{-1}$ \textit{exceptional}. Pan's proof combines local-global compatibility, patching at ``nice'' primes, and ordinary modularity and finiteness. The remaining geometric argument connects the relevant component to ordinary pro-modular points through potentially nice primes, which become nice after a further base change; see \cite[Sections 7.2--7.4]{pan2022fontaine}.
	
	The recent local results allow us to retain the local-global compatibility, patching, and ordinary-modularity framework. Pašk\=unas-Tung \cite{Pa_k_nas_2021} treat the exceptional block excluded from \cite{Pa_k_nas_2013}, giving the input for local-global compatibility in Section \ref{section lgc}.\footnote{Using these small-prime results, Tung studied the Fontaine-Mazur conjecture for $p=2$ in \cite{Tung_2020} and for $p=3$ in \cite{Tung_2021} under Taylor-Wiles hypotheses.} B\"ockle-Iyengar-Pašk\=unas \cite{B_ckle_2023} supply the structural properties of local framed deformation rings needed for patching, while the results of B\"ockle-Juschka \cite{Bockle_2023} enter the deformation-theoretic comparisons in Section \ref{sec 2}. These inputs give the exceptional-case local-global compatibility and patching theorems, but the global componentwise argument still encounters three difficulties. Write $d=[F:\Q]$ and let $C$ be an irreducible component of the fixed-determinant global pseudo-deformation space of dimension at least $1+2d$.
	
	\begin{itemize}
		\item \textbf{The conditions on nice primes.} Besides pro-modularity and irreducibility, a nice prime must satisfy a Taylor-Wiles type condition on dihedral representations and a triviality lifting condition at prime-to-$p$ places (Definition \ref{nice}). When $p=3$, the character $\omega$ is quadratic, so the dihedral obstruction is no longer excluded by the non-quadratic residual-character argument available in the cyclotomic case for $p\ge5$. 
		
		\item \textbf{The dimension of ordinary intersections.} For $p\ge5$ and local residual ratio $\omega$, the local reducibility ideal is generated by two elements \cite[Lemma 7.4.4]{pan2022fontaine}. At $p=3$, both opposite extension spaces are two-dimensional and the corresponding ideal requires four generators (Proposition \ref{four elements}). Thus imposing ordinarity simultaneously at all $d$ places above $p$ costs up to $4d$ equations. The same dimension count no longer shows that $C$ contains a one-dimensional ordinary prime, unlike \cite[Corollary 7.4.5]{pan2022fontaine}.
		
		\item \textbf{The dimension of ordinary components.} Even if an irreducible ordinary prime has been found, the residual dimension estimate used in \cite[Lemma 7.4.7]{pan2022fontaine} may be insufficient. When $\bar\chi=\omega=\omega^{-1}$, the group $H^0(G_{F,\Sigma},\ad^0\bar\rho_b(1))$ for a non-split extension $\bar\rho_b$ of $\bar\chi$ by $\mathbf{1}$ can be non-zero. The Greenberg-Wiles lower bound for the ordinary deformation ring is then only $d$, rather than $d+1$. Ordinary finiteness \cite[Theorem 5.1.1]{pan2022fontaine} alone therefore does not give the required pro-modularity by the same argument.
	\end{itemize}
	
	Our strategy addresses these difficulties by the following three ingredients. The first two give the componentwise construction and propagation mechanism; the third establishes the ordinary special case at which the induction terminates.
	
	1) \textbf{Propagation of pro-modularity.} Following the author's previous work \cite{Zhang2024}, we choose a single abelian base change at the outset, independently of the component or primes considered later. Over this field, we construct primes which are already nice: the tame local conditions are built into the construction, and Leopoldt's conjecture for abelian CM fields bounds the dihedral locus. Combined with a connectedness-dimension argument, this gives Lemma \ref{promodular to nice}: every irreducible component through a one-dimensional globally irreducible pro-modular prime contains a nice prime, and hence is pro-modular. Consequently, subsequent steps need only produce globally irreducible pro-modular primes; the additional nice-prime conditions, including the dihedral condition in Definition \ref{nice}, do not have to be imposed during the induction. Once pro-modularity is obtained at the ordinary endpoint, Lemma  \ref{promodular to nice} propagates it backwards through these points.
	
	2) \textbf{Induction on partial ordinariness.} After choosing the auxiliary field of sufficiently large degree, we impose ordinarity one place at a time. We cut by the four local reducibility equations at a new place, choose a globally irreducible one-dimensional prime, and prescribe an ordinary stable line there. The associated deformation problem is ordinary, with prescribed lines, at the places already treated and unrestricted at the remaining places above $p$. A Greenberg-Wiles dimension estimate supplies a partial ordinary component through this prime of dimension at least $2d-r$, where $r$ is the number of ordinary places (Lemma \ref{partial ordinary dimension}). Passing to this component recovers enough dimension to continue while retaining the previously prescribed ordinary lines. These components need not be nested: propagation eventually transfers pro-modularity back to $C$. Thus we do not need a direct lower bound for its fully ordinary locus.
	
	3) \textbf{Pro-modularity of irreducible ordinary points.} In the auxiliary setting of Section \ref{section 5}, every one-dimensional globally irreducible ordinary prime is pro-modular, regardless of its residue characteristic and without local genericity (Corollary \ref{special case ordinary}). We follow the ordinary deformation-theoretic argument of \cite[Section 7.4.2, especially Lemma 7.4.7]{pan2022fontaine}, but overcome the possible failure of residual adjoint $H^0$-vanishing by passing to suitable characteristic-zero irreducible ordinary points. There $H^0(G_{F,\Sigma},\ad^0\rho(1))=0$, so a characteristic-zero Greenberg-Wiles dimension estimate gives the required dimension bound (Proposition \ref{Greenberg-Wiles} and Corollary \ref{crucial}). Together with ordinary finiteness, Subsection \ref{sec special generic} and propagation, this proves the special case in Subsection \ref{sec special ordinary}.
	
	Once the induction reaches a prime which is ordinary at every place above $p$, the third ingredient gives its pro-modularity. The first then propagates pro-modularity backwards through the chosen intersection points, using components of the ambient unrestricted pseudo-deformation space, until it reaches $C$. This proves Proposition \ref{induction pro-modularity}; combined with the previously treated local types, it gives Theorem \ref{1.0.3} for all odd primes. Together with the classicality criterion of Corollary \ref{classicality}, based on \cite{Colmez_2014}, and the known ordinary cases, it also gives the residually reducible case of Theorem \ref{1.0.2}. The residually irreducible case is treated in Section \ref{section 6}.

	\begin{rem}[Further applications of the potential big \(R=\T\) theorem]
		
		First, as explained in \cite[Remark 6.4.10]{Pan_2022}, the main remaining
		input for extending the irregular Fontaine-Mazur theorem to more general
		totally real fields is a suitable classicality theorem for completed
		cohomology in the Hilbert modular setting.  If Pan's classicality result can be generalized to this setting, then the potential pro-modularity
		theorems (Theorem \ref{1.0.3} and Proposition \ref{big R=T in irr}) proved in this paper would give corresponding weight one Fontaine-Mazur results in the situations to which they apply.
		
		Second, the potential big \(R=T\) theorem can be combined with local-global
		compatibility for completed cohomology (see Corollary \ref{Hecke module}) and the method of Khare-Wintenberger \cite{Khare_2009} to obtain finiteness results from local pseudo-deformation rings to global pseudo-deformation rings in both
		the residually reducible and irreducible cases, which can be viewed as analogues of \cite[Theorem 1]{Allen_2014}.  These finiteness results can then be used to study Mazur's dimension conjecture and the equidimensionality of big Hecke
		algebras. These applications are developed in the author’s subsequent work
		\cite{ZhaZariskiDensity}.
	\end{rem}

	This paper is organized as follows. Section \ref{sec 2} prepares the local reducibility ideals, the comparison between pseudo-deformations and deformations, and the dimension and connectedness estimates. Section \ref{sec 3} establishes local-global compatibility and classicality in the exceptional case, using the results of Pašk\=unas-Tung recalled in Appendix \ref{A}. Section \ref{sec 4} provides the patching theorems and ordinary modularity inputs. The main componentwise argument is carried out in Section \ref{section 5}: propagation and two ordinary special cases lead to the induction, from which we deduce potential pro-modularity and the residually reducible case of the main theorem. Section \ref{section 6} treats the residually irreducible case.
	
	\subsection{Notations}
	Let $p$ be a prime. We always denote a $p$-adic local field by $K$. We also denote its uniformizer by $\varpi$, its ring of integers by $\cO$ and the residue field by $\F$. We fix an embedding of $K$ into $\overbar{\Q_p}$, some algebraic closure of $\Q_p$ and an isomorphism $\iota_p:\overbar{\Q_p}\simeq \bC$.
	
	We use $\mathfrak{U}_{\mathcal{O}}$ to denote the category of Artinian local $\cO$-algebras with residue field $\F$. For an $\cO$-module $M$, we use $M^\vee$ to denote its Pontryagin dual $\Hom_{\cO}^{\cont}(M, K/\cO)$, and we use $M^d$ to denote its Schikhof dual $\Hom_{\cO}(M, \cO)$. We use $M_{\operatorname{tf}}$ to denote its maximal $\cO$-torsion free quotient.
	
	For a prime ideal $\kp$ of a commutative ring $R$, we denote its residue field by $k(\kp)$. Let $R_\kp$ be the localization at $\kp$. We write $\widehat{R_\kp}$ as its $\kp$-adic completion. We say a commutative ring $R$ is a CNL ring if it is a complete noetherian local ring, and we usually denote its maximal ideal by $\mathfrak{m}_R$.
	
	Suppose $F$ is a number field with maximal order $O_F$. We write $F_+$ for the set of totally positive elements in $F$. For any finite place $v$, we write $F_v$ (resp. $O_{F_v}$) for the completion of $F$ (resp. $O_F$) at $v$, $\varpi_v$ for a uniformizer of $F_v$, $k(v)$ for the residue field, $\operatorname{Nm}(v)$ for the norm of $v$ (in $\Q$),  $G_{F_v}$ for the decomposition group above $v$, $I_{F_v}$ (or just $I_v$ if there is no confusion for the number field $F$) for its inertia group and $\Frob_v$ for a geometric Frobenius element in $G_{k(v)}:=G_{F_v}/I_{F_v}$. If $l$ is a rational prime, then we denote $O_F\otimes_\Z \Z_l$ by $\cO_{F,l}$. The adele ring of $F$ will be denoted by $\A_F$ . Suppose $\Sigma$ is a finite set of places of $F$. We use $G_{F, \Sigma}$ for the Galois group of the maximal extension of $F$ unramified outside $\Sigma$ and all infinite places. The absolute Galois group of $F$ is denoted by $G_F=\Gal(\overbar{F}/F)$. 
	
	We use $\varepsilon$ to denote the $p$-adic cyclotomic character and $\omega$ to denote the mod $p$ cyclotomic character. Our convention for the Hodge-Tate weight of $\varepsilon$ is $-1$. We use $\mu_n$ to denote a primitive $n$-th root of unity.
	
	\subsection{Acknowledgements}
    The author would like to thank his supervisor Professor Takeshi Saito for his constant help and encouragement. The author would like to thank Kojiro Matsumoto for many helpful discussions. The author would like to thank GPT 6 Astra for pointing out a serious error in an earlier version of Proposition \ref{four elements} and helping him to modify the strategy of the proof in Section \ref{section 5}.

	\section{Galois deformation theory}\label{sec 2}
	This section prepares the deformation-theoretic inputs for Section \ref{section 5}: local framed rings, comparison with pseudo-deformations, and dimension and connectedness estimates. In particular, the exceptional local reducibility ideal and the mixed deformation problems will allow us to impose ordinary conditions successively while retaining sufficiently large components.
	
	\subsection{Local framed deformation rings} We record the unrestricted local rings used in patching and the ordinary local rings used in the mixed deformation problem of Lemma \ref{partial ordinary dimension}. The unrestricted results follow from \cite{B_ckle_2023}; for the exceptional ordinary case, we use B\"ockle's explicit description \cite{bockle2010deformation}.
	
	Let $p$ be a prime. Let $F_v$ be a finite extension of $\mathbb{Q}_p$. Let $K$ be another finite extension of $\mathbb{Q}_p$ with its ring of integers $\mathcal{O}$, a uniformizer $\varpi$ and its residue field $\mathbb{F}$.
	
	Let $\chi: G_{F_v} \to \mathcal{O}^\times$ be a continuous character. Let $\bar{\rho}: G_{F_v} \to \operatorname{GL}_2(\mathbb{F})$ be a continuous representation with determinant $\bar{\chi}$. Denote by $D_{\bar{\rho}}^{\square, \chi} : \mathfrak{U}_{\mathcal{O}} \to \operatorname{Sets}$ the functor such that for $ (A, \mathfrak{m}_A) \in \mathfrak{U}_{\mathcal{O}}$, $ D_{\bar{\rho}}^{\square, \chi}(A)$ is the set of continuous representations $\rho_A : G_{F_v} \to \operatorname{GL}_2(A)$ satisfying $ \rho_A \mod \mathfrak{m}_A = \bar{\rho}$ and $\operatorname{det} \rho_A= \chi$. The functor $D_{\bar{\rho}}^{\square, \chi}$ of liftings (framed deformations) of $\bar{\rho}$ is pro-represented by a CNL $\mathcal{O}$-algebra $R_v^{\square}$.
	
	\begin{prop}\label{dim of loc framed der ring}
		1) The ring $R_v^{\square}$ is a local complete intersection, flat over $\mathcal{O}$ and of relative dimension $3([F_v: \mathbb{Q}_p]+1)$. 
		
		2) The ring  $R_v^{\square}$ is a normal integral domain.
	\end{prop}
	
	\begin{proof}
		The first result is \cite[Corollary 5.4]{B_ckle_2023}. The second one is \cite[Theorem 5.6]{B_ckle_2023}.
	\end{proof}
	
	For the rest of this subsection, we assume $p=3$, $F_v=\mathbb{Q}_p$ and $ \bar{\rho}(\sigma)=\begin{pmatrix}
		\omega(\sigma) & *\\
		0 & 1
	\end{pmatrix}$, where $*$ may be zero. To describe the framed deformation ring $R_v^{\square}$, we recall the main work in \cite{bockle2010deformation}.
	
	Let $E=\mathbb{Q}_3(\mu_3)$, and $L$ be the splitting field of $\bar\rho$. Then $U:=\operatorname{Gal}(L/E)$ has order $1$, $3$ or $9$, depending on $*$.
	
	Let $P_E$ be the maximal pro-$3$ quotient of $G_E$, and let $E(3)$ be the fixed field of its kernel. Put $G:=\operatorname{Gal}(E/\mathbb{Q}_3)=\{1,\sigma\}$ and choose a splitting of $\operatorname{Gal}(E(3)/\mathbb{Q}_3)\twoheadrightarrow G$. Choose a residual basis in which $\bar\rho(\sigma)=\operatorname{diag}(-1,1)$. For a CNL $\mathcal O$-algebra $R$, let $\tilde\Gamma(R)$ be the inverse image of the upper triangular unipotent subgroup $U_2(\F)$ under $\GL_2(R)\to\GL_2(\F)$. The group $G$ acts on $P_E$ by conjugation and on $\tilde\Gamma(R)$ through $\operatorname{diag}(-1,1)$.
	
	When $U$ is non-trivial, choose $u\in P_E$ with non-trivial image in $U$, as specified below, and normalize the residual basis so that $\bar\rho(u)=\left(\begin{smallmatrix}1&1\\0&1\end{smallmatrix}\right)$. We use the normalization of \cite[Section 4]{bockle2010deformation} and define
	$$\operatorname{EH}(R):=\{\alpha\in\operatorname{Hom}_{G}^{\cont}(P_E,\tilde\Gamma(R)): \alpha\bmod\mathfrak m_R=\bar\rho|_{G_E},\ (\alpha(u))_{12}=1\}.$$
	Here only the upper right entry of $\alpha(u)$ is fixed, not the whole matrix; this normalization is obtained by strict diagonal conjugation. When $U$ is trivial, we omit the condition on $\alpha(u)$.
	
	\begin{prop}
		The functor $\operatorname{EH}$ is representable. Let $(\tilde{R}, \tilde{\alpha})$ denote the universal pair. Then $ \tilde{R}$ is a versal hull of $\bar{\rho}$. If $\bar{\rho}$ is non-split, then $ \tilde{R}$ is universal.
	\end{prop}
	
	\begin{proof}
		This follows from \cite[Proposition 2.2 and Section 4]{bockle2010deformation}, with the above normalization in the non-split case; see also \cite[page 531]{bockle2010deformation}.
	\end{proof}
	
	The versal deformation $\tilde\rho:G_{\mathbb Q_3}\to\GL_2(\tilde R)$ factors through $G_{\mathbb Q_3}\twoheadrightarrow\operatorname{Gal}(E(3)/\mathbb Q_3)=P_E\rtimes G$, with $\tilde\rho(\sigma)=\operatorname{diag}(-1,1)$.
	
	Let $F_4$ be the free pro-$3$ group with topological generators $x_1, x_2, x_3, x_4$ with the following $G$-action:
	$$ \sigma(x_1)=x_1^{-1}, \sigma (x_2)=x_2, \sigma(x_3)=x_3^{-1}, \sigma(x_4)=x_4.$$
	Define $r_0:= x_1^3[x_1, x_2][x_3, x_4]$, where $[g, h]= ghg^{-1}h^{-1}$, and define 
	$$r:= r_0\sigma(r_0)^{-1}= x_1^3[x_1, x_2][x_3, x_4][x_4, x_3^{-1}][x_2, x_1^{-1}]x_1^3.$$
	Let $N_4 \subset F_4$ be the closed normal subgroup generated by $r$. By \cite[Lemma 3.2]{bockle2010deformation}, we know that the pro-$3$ group \(F_4/N_4\) is isomorphic to \(P_E\) as a group with \(G\)-action.
	
	Denote by $A_i \in \tilde{\Gamma}(R) \subset \operatorname{GL}_2(R)$ the image of $x_i$.
	Then we can deduce that $$
	A_1=\begin{pmatrix}
		\sqrt{1+bc}& b\\
		c & \sqrt{1+bc}
	\end{pmatrix},
	A_2=\sqrt{1+a}\begin{pmatrix}
		\sqrt{1+d}& 0\\
		0 & \sqrt{1+d}^{-1}
	\end{pmatrix},$$	
	$$ A_3=\begin{pmatrix}
		\sqrt{1+b'c'}& b'\\
		c' & \sqrt{1+b'c'}
	\end{pmatrix},
	A_4=\sqrt{1+a'}\begin{pmatrix}
		\sqrt{1+d'}& 0\\
		0 & \sqrt{1+d'}^{-1}
	\end{pmatrix},$$	
where $a,a',c,c',d,d' \in \mathfrak{m}_R$ and $b,b' \in R$ (whether $b$ or $b'$ lie in $\mathfrak{m}_R$ depending on $\bar{\rho}$). 

   Let $B_1:=[A_2, A_1]A_1^{-6}[A_1^{-1}, A_2]$, $B_2:= [A_3, A_4][A_4, A_3^{-1}]$. Then $B=1$ is equivalent to $B_1=B_2$, where $B=A_1^3[A_1, A_2][A_3, A_4][A_4, A_3^{-1}][A_2, A_1^{-1}]A_1^3$. Let $B_k(i,j)$ be the $(i,j)$-entry of the matrix $B_k$. Then $B_1=B_2$ is equivalent to the equalities $B_1(1,2)=B_2(1,2)$ and $B_1(2,1)=B_2(2,1)$.
   
   Consider the $1$-cocycle $\beta$ defined by $\bar{\rho}$. We distinguish the following cases:
   
   1) $\beta(x_1) \ne 0$. We choose $u=x_1$, so that $b=1$, and we write $b'= \tilde{\bar{b'}}+\delta_b'$, where $\tilde{\bar{b'}}$ is the Teichm\"uller lifting and $\delta_b' \in \mathfrak{m}_R$.
   
   2) $\beta(x_3) \ne 0 = \beta(x_1)$. We choose $u=x_3$, so that $b'=1$ and $b \in \mathfrak{m}_R$.
   
   3) $\beta=0$. Then $U$ is trivial and $b,b' \in \mathfrak{m}_R$.
   
   \begin{prop} \label{loc def ring at 3}
   	According to the above three cases, we have
   	
   	1) $\tilde{R}= \mathcal{O}[[a, a', \delta_b', c, c', d, d']]/(B_1(1,2)-B_2(1,2), B_1(2,1)-B_2(2,1))$, where we replace $b$ by $1$ and $b'$ by $\tilde{\bar{b'}}+\delta_b'$ in the expressions of the entries.
   	
   	2) $\tilde{R}= \mathcal{O}[[a, a', b, c, c', d, d']]/(B_1(1,2)-B_2(1,2), B_1(2,1)-B_2(2,1))$, where we replace $b'$ by $1$ in the expressions of the entries.
   	
   	3) $\tilde{R}= \mathcal{O}[[a, a',b, b', c, c', d, d']]/(B_1(1,2)-B_2(1,2), B_1(2,1)-B_2(2,1))$. 
   \end{prop}
	
	\begin{proof}
		See \cite[Theorem 4.1]{bockle2010deformation}.
	\end{proof}
	
	Using this, we can study the ordinary deformation ring. Suppose $R_v^{\triangle}$ pro-represents the functor from $\mathfrak{U}_{\mathcal{O}} $ to the category of sets sending $R$ to the set of pairs $(\rho_R, \psi_R)$ such that 
	\begin{itemize} 
		\item $\rho_R:G_{F_v}\to\GL_2(R)$ is a lifting of $\bar{\rho}$ with determinant $\chi$.
		\item $\psi_R:G_{F_v}\to R^\times$ is a lifting of $\omega$ such that $\rho_R$ has a unique $G_{F_v}$-stable rank one $R$-submodule of $R^2$ with $G_{F_v}$-action via $\psi_R$.
	\end{itemize}
	
	\begin{prop}\label{framed loc ord}
		According to the above three cases, we have
		
		1) $R_v^{\triangle} \cong \mathcal{O}[[\delta_b',d',x, y, z]]$ if $\beta(x_1) \ne 0$.
		
		2) $R_v^{\triangle} \cong \mathcal{O}[[b, d, x, y, z]]$ if $\beta(x_3) \ne 0 = \beta(x_1)$.
		
		3) $R_v^{\triangle} \cong \mathcal{O}[[b, b', d, d', x, y]]/(3b-bd-b'd')$ if $\beta=0$. 
		
		In each case, the ring $R_v^{\triangle}$ is a normal domain of dimension $6$.
	\end{prop}
	
	\begin{proof}
		We only prove the first case, and the other two cases are similar.
		
		Note that $R_v^{\square}$ is the universal framed deformation (lifting) ring. By \cite[Proposition 2.1]{Khare_2009}, we have $R_v^\square \cong \mathcal{O}[[\delta_b', c, c', d, d', x, y,z]]/(B_1(1,2)-B_2(1,2), B_1(2,1)-B_2(2,1))$ as $a, a'$ are uniquely determined by the determinant $\chi$. Therefore, we have $R_v^{\triangle} \cong R_v^{\square}/(c, c')$. Then our result follows from an easy calculation.
	\end{proof}
	
	\begin{rem}
		This dimension bound will be used for the partial ordinary deformation rings in Subsection \ref{sec partial ordinary}. See also Remark \ref{ordinary R=T}.
	\end{rem}
	
	\subsection{Pseudo-representations}
	In this subsection, we will give some results on 2-dimensional pseudo-representations. Our main references are \cite{Skinner_1999} and \cite{pan2022fontaine}.
	
	\begin{defn} \label{pseudo}
		For a profinite group $G$ and a topological commutative ring $R$ in which 2 is invertible, a 2-dimensional \textit{pseudo-representation} is a continuous function $T: G \to R$ such that 
		
		1) $T(1)=2$,
		
		2) $T(\sigma \tau)=T(\tau \sigma)$ for all $ \sigma, \tau \in G$,
		
		3) $T(\gamma \delta \eta)+T(\gamma \eta \delta)-T(\gamma \eta)T(\delta)-T(\eta \delta)T(\gamma)-T(\delta \gamma)T(\eta)+T(\gamma)T(\delta)T(\eta)=0,$
		for any $\delta, \gamma,\eta\in G$.
		
		The \textit{determinant} $\operatorname{det}(T)$ of a pseudo-representation $T$ is a character (using 3)) defined by $$ \operatorname{det}(T): G \to R^\times, ~
		\operatorname{det}(T)(\sigma)=\frac{1}{2}(T(\sigma)^2-T(\sigma^2)), \sigma \in G.$$
		
		If there exists an order $2$ element $c \in G$ such that $T(c)=0$, then for a pseudo-representation $T$ and $\sigma, \tau \in G$,  we define:
		
		1) $a(\sigma)=\frac{1}{2}(T(c\sigma)+T(\sigma))$,
		
		2) $d(\sigma)=\frac{1}{2}(-T(c\sigma)+T(\sigma))$,
		
		3) $y(\sigma, \tau)=a(\sigma \tau)-a(\sigma)a(\tau).$
	\end{defn}

	If $\rho: G \to \operatorname{GL}_2(R)$ is a continuous representation with an element $c \in G$ satisfying $\rho(c)=\begin{pmatrix}
		1 & ~\\
		~& -1
	\end{pmatrix}
	$, then $\operatorname{tr}(\rho)$ is a pseudo-representation. More explicitly, if $\rho(\sigma)=\begin{pmatrix}
		a_\sigma & b_\sigma\\
		c_\sigma& d_\sigma
	\end{pmatrix}$, then $a(\sigma)=a_\sigma$, $d(\sigma)=d_\sigma$ are as defined above and $y(\sigma, \tau)=b_\sigma c_\tau$.
	
	Pan \cite[2.1.4]{pan2022fontaine} verifies that the continuous functions $\{a,d,y\}$ satisfy the following equations.
	
	(1) $y(\sigma, \tau)=d(\tau\sigma)-d(\tau)d(\sigma)$.
	
	(2) $y(\sigma\tau, \delta)=a(\sigma)y(\tau, \delta)+y(\sigma, \delta)d(\tau)$.
	
	(3) $y(\sigma, \tau\delta)=a(\delta)y(\sigma, \tau)+y(\sigma, \delta)d(\tau)$.
	
	(4) $y(\alpha, \beta)y(\sigma, \tau)=y(\alpha, \tau)y(\sigma, \beta)$.\\
	
	The following result is clear from the last equation.
	
	\begin{prop}\label{u}\cite[Proposition 2.1.2]{Zhang2024}
		Let $R$ be an integral domain, and $T: G \to R$ be a pseudo-representation for a profinite group $G$. Let $\alpha, \beta $ be two elements in $ G$.
		
		1) If $y(\alpha, \beta) =0$ in $R$, then either $y(\alpha, \tau)=0$ or $y(\tau, \beta)=0$ for any $\tau \in G$.
		
		2) Assume that the set $J_{\beta} =\{\tau \in G : y(\beta, \tau)\ne 0\}$ (resp. $J'_{\beta} =\{\tau \in G : y(\tau, \beta) \ne 0\}$) is non-empty. For $\tau \in J_{\beta}$ (resp. $J'_{\beta}$), let $u(\alpha, \beta)(\tau)= \frac{y(\alpha, \tau)}{y(\beta, \tau)}$ (resp. $u'(\alpha, \beta)(\tau)=\frac{y(\tau, \alpha)}{y(\tau, \beta)}$) be an element in the fraction field of $R$. Then we have $u(\alpha, \beta)(\tau)=u(\alpha, \beta)(\tau')$ (resp. $u'(\alpha, \beta)(\tau)=u'(\alpha, \beta)(\tau')$)  for any $\tau, \tau' \in J_\beta$ (resp. $J'_{\beta}$). (Afterwards, we will use the notation $u(\alpha, \beta)$ and $u'(\alpha, \beta)$ for simplicity if they are well-defined.)
	\end{prop}
	
	\begin{defn}\label{construction}\cite[2.1.5]{pan2022fontaine}
		Assume that $R$  is either a field or a DVR. For a pseudo-representation $T: G\to R$, a representation $\rho$ associated to $T$ is in the following form.
		
		1) $ \rho(\sigma)=
		\begin{pmatrix}
			a(\sigma) & ~\\
			~& d(\sigma)
		\end{pmatrix}$, if all $y(\sigma, \tau)=0$. We call this case \textit{reducible}.
		
		2) If $y(\sigma, \tau) \ne 0 $ for some $ \sigma, \tau \in G$, choose $\sigma_0, \tau_0$ such that $\frac{y(\sigma, \tau)}{y(\sigma_0, \tau_0)} \in R$ for any $ \sigma, \tau \in G$, . Define $ \rho(\sigma)=
		\begin{pmatrix}
			a(\sigma) & \frac{y(\sigma, \tau_0)}{y(\sigma_0, \tau_0)}\\
			y(\sigma_0, \sigma)& d(\sigma)
		\end{pmatrix}$.  We call this case \textit{irreducible}.
	\end{defn}
	
	The next result is helpful to study the representation constructed in the previous definition.
	
	\begin{prop}\label{innovation'}
		Let $R$ be a CNL domain with maximal ideal $\mathfrak{m}_R$ and residue field $\F$ satisfying $\operatorname{char} \F=p >0$. Let $G$ be a profinite group and $T: G \to R$ be a pseudo-representation of $G$. Assume that $T~\operatorname{mod}~\mathfrak{m}_R$ is reducible. Let $S$ be a finite subset (not necessarily a group) of $G$. Then there exist a partition of $S = S_1 \amalg S_2$, a positive integer $n>1$ satisfying $(n, p)=1$ and a CNL domain $R'$ satisfying the following conditions.
		
		a) $R'$ is a quotient of $R$.
		
		b) For any $\theta \in S_1$, we have $y(\theta, \alpha)=0$ for any $\alpha \in G$ in $R'$. For any $\theta', \theta'' \in S_2$, we have $y(\theta', \alpha)^n=y(\theta'', \alpha)^n$ for any $\alpha \in G$ in $R'$.
		
		c) For any $\alpha \in G,~\theta \in S_2$, either $ y(\theta, \beta)=0$ for any $\beta \in G$ in $R'$ or $u(\alpha, \theta)$ is well-defined and integral over $R'$.
		
		d) We have $\operatorname{dim} R' \ge \operatorname{dim} R - |S|$. If $S_2$ is not empty, then further $\operatorname{dim} R' \ge \operatorname{dim} R - |S|+1$.
		
	\end{prop}
	
	\begin{proof}
		This is \cite[Corollary 2.1.7]{Zhang2024}.
	\end{proof}
	
	Now suppose that $p>2$ and $G$ satisfies Mazur's finiteness condition $\Phi_p$. Let $\cO$ be the ring of integers of a finite extension of $\Q_p$, with residue field $\F$. Let $\bar\chi:G\to\F^\times$ be a character such that $\bar\chi(c)=-1$ for some order $2$ element $c\in G$. Write $R^{\ps}$ for the universal pseudo-deformation ring of $1+\bar\chi$, and $R_b$ for the universal deformation ring of a non-split extension $\bar\rho_b$ of $\bar\chi$ by $\mathbf{1}$. Taking traces gives a natural map $R^{\ps}\to R_b$. We recall the following comparison, which applies to one-dimensional primes in both characteristics.
	
	\begin{prop}\label{ps-def comparison}
		1) Let $\kp_b\in\Spec R_b$ be a one-dimensional prime such that the specialization of the universal representation at $\kp_b$ is irreducible, and let $\kp=\kp_b\cap R^{\ps}$. Then the natural map $\widehat{(R^{\ps})_\kp}\to\widehat{(R_b)_{\kp_b}}$ is an isomorphism, and $\Spec(R_b)_{\kp_b}\to\Spec(R^{\ps})_\kp$ is surjective.
		
		2) Let $\mathfrak{Q}\in\Spec R_b$ be a prime at which the universal representation is irreducible. Then $\dim R_b/\mathfrak{Q}\le\dim R^{\ps}/(\mathfrak{Q}\cap R^{\ps})$. If $\mathfrak{Q}$ is minimal, then $\mathfrak{Q}\cap R^{\ps}$ is also minimal.
		
		The same statements hold with the same fixed determinant imposed on both rings.
	\end{prop}
	
	\begin{proof}
		This is \cite[Corollary 2.3.6]{pan2022fontaine}. The fixed-determinant version follows by imposing the same determinant on both sides of the comparison of completed local rings and using the same proof.
	\end{proof}
	 
	\subsection{Local pseudo-deformation rings}\label{2.3} We study local reducibility ideals. Proposition \ref{four elements} gives the four-generator bound for the induction in Subsection \ref{sec partial ordinary}; Proposition \ref{principal} gives the sharper bound at generic ordinary points used in Subsection \ref{sec special generic}.
	
	Let $p$ be a prime. Let $F_v$ be a finite extension of $\mathbb{Q}_p$. Let $K$ be another finite extension of $\mathbb{Q}_p$ containing all the roots of unity of $p$-power order in $F_v$, with its ring of integers $\mathcal{O}$, a uniformizer $\varpi$ and its residue field $\mathbb{F}$. Let $\chi :G_{F_v} \to \mathcal{O}^\times$ be a continuous character.
	
	Let $\bar{D}= 1+\bar{\chi}$ be a pseudo-representation. Let $D^{\operatorname{ps}}: \mathfrak{U}_{\mathcal{O}} \to \operatorname{Sets}$ be the functor such that for $ (A, \mathfrak{m}_A) \in \mathfrak{U}_{\mathcal{O}}$, $ D^{\operatorname{ps}}(A)$ is the set of continuous $A$-valued $2$-dimensional pseudo-representations $D : G_{F_v} \to A$ satisfying $ (D \mod \mathfrak{m}_A) = \bar{D}$. The functor is pro-represented by a CNL $\mathcal{O}$-algebra $(R_v^{\operatorname{ps}}, \mathfrak{m}_{R_v^{\operatorname{ps}}})$. Note that in our definition of $ R_v^{\operatorname{ps}}$, we do not fix the determinant of our pseudo-representation.
	
	\begin{prop} \label{dim of ps}
		Assume that $\bar{D}$ is multiplicity free, i.e., $\bar{\chi}$ is non-trivial.
		
		1) The ring $R_v^{\operatorname{ps}}$ is reduced and $\mathcal{O}$-torsion free.
		
		2) The ring $R_v^{\operatorname{ps}}/(\varpi)$ is equidimensional of dimension $4[F_v: \mathbb{Q}_p]+1$.
		
		3) Let $R_v^{\ps, \chi}$ be the universal pseudo-deformation ring parametrizing all pseudo-deformations of $\bar{D}$ with determinant $\chi$. Then $R_v^{\ps, \chi}$ is an $\mathcal{O}$-torsion free integral domain.
	\end{prop}
	
	\begin{proof}
		For 1), see \cite[Corollary 4.28]{B_ckle_2023}. For 2), see \cite[Theorem 5.5.1]{Bockle_2023}. For 3), see \cite[Corollary 5.11]{B_ckle_2023}.
	\end{proof}
	
	For the rest of this subsection, we assume $p=3$, $F_v=\mathbb{Q}_p$ and $\bar{\chi}=\omega=\omega^{-1}$. Let $R_v^{\operatorname{ps, ord}}$ be the quotient of $R_v^{\operatorname{ps}}$ parametrizing all reducible liftings of $1+\bar{\chi}$.
	
	\begin{prop}\label{dim of ps ord}
		1) The ring $R_v^{\operatorname{ps}}$ is of dimension $6$.
		
		2) The ring $R_v^{\operatorname{ps, ord}}$ is a power series ring over $\mathcal{O}$ of relative dimension $4$.
	\end{prop}
	
	\begin{proof}
		1) By Krull's principal ideal theorem and 2) of Proposition \ref{dim of ps}, $R_v^{\operatorname{ps}}$ is of dimension $5$ or $6$. If $\operatorname{dim} R_v^{\operatorname{ps}} =5$, then $\varpi$ is a zero divisor in $R_v^{\operatorname{ps}}$, which is contrary to 1) of Proposition \ref{dim of ps}.
		
		2) Let $R_1$ and $ R_{\bar{\chi}}$ be the universal deformation rings of the one-dimensional representation $1$ and $\bar{\chi}$, respectively. By \cite[page 6]{B_ckle_2023}, each of them is a power series ring over $\mathcal{O}$ of relative dimension $2$. Then $R= R_1 \hat{\otimes}_{\mathcal{O}} R_{\bar{\chi}}$ is a power series ring over $\mathcal{O}$ of relative dimension $4$. By the universal property, we get a natural surjection $R \twoheadrightarrow R_v^{\operatorname{ps, ord}}$.
		
		Recall the versal deformation ring $\tilde{R}$ in the third case of Proposition \ref{loc def ring at 3}, i.e., the cocycle $\beta=0$. Then we get natural maps $R_v^{\operatorname{ps}} \to \tilde{R} \twoheadrightarrow  \tilde{R}/(b, b', c, c') \cong \mathcal{O}[[a, a', d, d']]$. From the description in the previous subsection,  the last ring defines a natural reducible pseudo-representation. Hence, we have natural maps $R_v^{\operatorname{ps}} \twoheadrightarrow R_v^{\operatorname{ps, ord}} \to \mathcal{O}[[a, a', d, d']]$.
		
		By the universal property, we get a natural surjection $R \to \mathcal{O}[[a, a', d, d']]$. Since each of them is a power series ring over $\mathcal{O}$ of relative dimension $4$, the surjection is actually an isomorphism. Now we consider the natural maps $ R \twoheadrightarrow R_v^{\operatorname{ps, ord}} \to  \mathcal{O}[[a, a', d, d']]$. The composite map is an isomorphism, so the natural surjection $R \twoheadrightarrow R_v^{\operatorname{ps, ord}}$ is an isomorphism.
	\end{proof}
	
	The reducibility ideal in the exceptional case differs from the one in \cite[Lemma 7.4.4]{pan2022fontaine}.
	
	\begin{prop}\label{four elements}
		The kernel $J_v$ of the natural surjection $R_v^{\operatorname{ps}}\twoheadrightarrow R_v^{\operatorname{ps, ord}}$ can be generated by four elements, and cannot be generated by fewer than four elements.
	\end{prop}
	
	\begin{proof}
		For a CNL $\cO$-algebra $R$ with residue field $\F$, write $\mathfrak{t}_R=\Hom_\F(\mathfrak{m}_R/(\mathfrak{m}_R^2+\varpi R),\F)$ for its relative tangent space. Put $G=G_{\Q_3}$. Local Tate duality and the Euler characteristic formula give $H^2(G,\F)=0$, $\dim_\F H^1(G,\F)=2$ and $\dim_\F H^1(G,\F(\omega))=\dim_\F H^1(G,\F(\omega^{-1}))=2$. Thus \cite[Theorem 2]{bellaiche2012pseudodeformations} gives an exact sequence
		\[0\longrightarrow H^1(G,\F)^{\oplus2}\longrightarrow\mathfrak{t}_{R_v^{\ps}}\longrightarrow H^1(G,\F(\omega))\otimes_\F H^1(G,\F(\omega^{-1}))\longrightarrow0.\]
		In particular, $\dim_\F\mathfrak{t}_{R_v^{\ps}}=8$, whereas $\dim_\F\mathfrak{t}_{R_v^{\ps,\ord}}=4$ by Proposition \ref{dim of ps ord}(2).
		
		Choose a surjection $S:=\cO[[x_1,\ldots,x_8]]\twoheadrightarrow R_v^{\ps}$. Both $S$ and $R_v^{\ps,\ord}$ are regular local rings, so the kernel of the composite surjection $S\twoheadrightarrow R_v^{\ps,\ord}$ is generated by $\dim S-\dim R_v^{\ps,\ord}=4$ elements by \cite[Lemma 10.106.4]{stacks-project}. Their images generate $J_v$. Conversely, the image of $J_v$ in $\mathfrak{m}_{R_v^{\ps}}/(\mathfrak{m}_{R_v^{\ps}}^2+\varpi R_v^{\ps})$ has dimension $8-4=4$, which implies our claim.
	\end{proof}
	
	\begin{rem}\label{exceptional local presentation}
		1) For comparison, \cite[Lemma 7.4.4]{pan2022fontaine} concerns a fixed-determinant ring when $p>3$, of the form $\cO[[x_0,x_1,y_0,y_1]]/(x_0y_1-x_1y_0)$, with reducibility ideal $(x_0,x_1)$. At $p=3$, the identity $\omega=\omega^{-1}$ makes both opposite extension spaces two-dimensional. The relative tangent dimensions of $R_v^{\ps}$ and $R_v^{\ps,\ord}$ are consequently $8$ and $4$; after fixing the determinant they are $6$ and $2$. In either convention the difference is four. Thus the two-generator argument in Pan's lemma does not extend to this case.
		
		2) B\"ockle's description also gives a presentation of $R_v^{\ps}$ without expanding the lengthy coefficients. Use the variables and matrices in Proposition \ref{loc def ring at 3}(3), and put $X=bc$, $Y=bc'$, $Z=b'c$ and $W=b'c'$. Write $B_1(1,2)-B_2(1,2)=bf-b'h$ and $B_1(2,1)-B_2(2,1)=cg-c'j$, where $f,g\in\cO[[d,X]]$ and $h,j\in\cO[[d',W]]$. The matrix formulas give $f(0)=g(0)=-6$ and $h(0)=j(0)=0$. We obtain
		\[\begin{split}
		R_v^{\ps}&\cong\cO[[a,a',d,d',X,Y,Z,W]]/\mathcal I,\\
		\mathcal I&=(XW-YZ,\ Xf-Zh,\ Yf-Wh,\ Xg-Yj,\ Zg-Wj).
		\end{split}\]
		Indeed, the diagonal torus has weights $1$ on $b,b'$, $-1$ on $c,c'$ and $0$ on $a,a',d,d'$. The multiplicity-free GMA description identifies $R_v^{\ps}$ with the completed ring of invariants of the split adapted representation space; see \cite[Corollary 2.25 and Theorem 3.8(4)]{WangErickson2018}. Besides $a,a',d,d'$, the invariant coordinates are the four products above, subject to $XW-YZ=0$. Since a torus is linearly reductive, the remaining invariant relations are obtained by multiplying $bf-b'h$ by $c,c'$ and $cg-c'j$ by $b,b'$, giving the four remaining generators of $\mathcal I$.
		
		Under this presentation, $J_v=(X,Y,Z,W)$ and $R_v^{\ps,\ord}\cong\cO[[a,a',d,d']]$. Fixing the determinant specializes the two scalar variables $a,a'$. All five relations lie in the square of the maximal ideal of the ambient power series ring, which also recovers the tangent dimensions above. This is a presentation deduced from B\"ockle's split model, rather than the hypersurface description available for $p>3$.
	\end{rem}
	
	The ring $R_v^{\operatorname{ps}}$ need not be regular. The following result describes its local rings at generic reducible points.
	
	\begin{prop} \label{singular locus}
		Let $\mathfrak{q}$ be a one-dimensional prime of $R_v^{\ps}$ with pseudo-representation $\psi_1+\psi_2$. If $\psi_1\ne\psi_2(m)$ for $m\in\{0,\pm1\}$ (here $(*)$ means the Tate twist), then $(R_v^{\ps})_{\mathfrak{q}}$ is regular of dimension $5$.
	\end{prop}
	
	\begin{proof}
		If $\varpi\notin\mathfrak{q}$, the completed local ring $\widehat{(R_v^{\ps})_{\mathfrak{q}}}$ is the completed local ring at the corresponding point of the rigid analytic character variety, as in \cite[Section 1]{chenevier2009variete} and \cite[Proposition 2.2.1]{pan2022fontaine}. It is regular of dimension $5$ by \cite[Th\'eor\`emes 2.1 and 2.3]{chenevier2009variete}, and the claim follows. If $\varpi\in\mathfrak{q}$, then $(R_v^{\ps}/(\varpi))_{\mathfrak{q}}$ is regular by \cite[Theorem 5.5.5(c)]{Bockle_2023}, and has dimension $4$ by Proposition \ref{dim of ps}(2) and catenarity. By Proposition \ref{dim of ps}(1), $\varpi$ is a non-zero divisor in $(R_v^{\ps})_{\mathfrak{q}}$. Hence $(R_v^{\ps})_{\mathfrak{q}}$ is regular of dimension $5$.
	\end{proof}
	
	The following result is an analogue to \cite[Lemma 7.4.8]{pan2022fontaine}.
	
	\begin{prop}\label{principal}
		Let $\mathfrak{q}$ be a prime of $ R_v^{\operatorname{ps, ord}}$ of dimension $1$, and let $\mathfrak{q}'$ be the inverse image of $\mathfrak{q}$ in $R_v^{\operatorname{ps}}$. Suppose that the pseudo-deformation $D_{\mathfrak{q}}= \psi_1 + \psi_2$ associated to $\mathfrak{q}$ satisfies $\psi_1 \ne \psi_2 (\pm 1)$. Then the kernel of the natural surjection $(R_v^{\operatorname{ps}})_{\mathfrak{q}'} \twoheadrightarrow (R_v^{\operatorname{ps, ord}})_{\mathfrak{q}}$ is principal.
	\end{prop}
	
	\begin{proof}
		The residual pseudo-representation is multiplicity-free, so $\psi_1\ne\psi_2$. Proposition \ref{singular locus} shows that $(R_v^{\ps})_{\mathfrak{q}'}$ is regular of dimension $5$. By Proposition \ref{dim of ps ord}(2), $R_v^{\ps,\ord}$ is a power series ring over $\mathcal{O}$ of relative dimension $4$, so $(R_v^{\ps,\ord})_{\mathfrak{q}}$ is regular of dimension $4$. The kernel is therefore generated by one element by \cite[Lemma 10.106.4]{stacks-project}.
		
		%Thus, we know that the kernel is a prime of $(R_v^{\operatorname{ps}})_{\mathfrak{q}'}$ of height $1$ as $(R_v^{\operatorname{ps}})_{\mathfrak{q}'}$ is catenary. By Auslander-Buchsbaum Theorem (\cite[Theorem 48]{matsumura80}), $(R_v^{\operatorname{ps}})_{\mathfrak{q}'}$ is also a UFD, and hence every prime ideal of it of height $1$ is principal by \cite[Theorem 47]{matsumura80}.
	\end{proof}
	
	\subsection{Galois cohomology} We establish the local and global Euler characteristic formulae and the Greenberg-Wiles formula for representations over a $p$-adic local field. These will be used to study deformation rings at characteristic-zero points, where the adjoint $H^0$ term can behave differently from its residual counterpart.
	
	Let $p$ be an odd prime. Let $K$ be a $p$-adic local field with its ring of integers $\mathcal{O}$, a uniformizer $\varpi$ and its residue field $\mathbb{F}$. Let  $G$ be a profinite group, and $V$ a finite-dimensional $K$-vector space with continuous $G$-action. We first state the Euler-Poincar\'e characteristic formulae in the version for $V$. When $V$ is a finite $G$-module, one can see \cite[Theorem 7.3.1]{Neukirch_2008} and \cite[8.7.4]{Neukirch_2008}.
	
	\begin{thm}[Local Characteristic Formula]\label{local char for}
	  Let $l$ be a prime. Let $E$ be an $l$-adic field, and suppose $G=G_{E}$.
		
		1) For any $i \ge 0$, we have $ \operatorname{dim}_K H^i(G_E, V) < \infty$. If further $i \ge 3$, then we have $\operatorname{dim}_K H^i(G_E, V) =0$.
		
		2) If $l=p$, we have $\sum_{i=0}^{2} (-1)^i \operatorname{dim}_K H^i(G_E, V) = -[E:\mathbb{Q}_p]\cdot \operatorname{dim}_K V$. If $l \ne p$, we have $\sum_{i=0}^{2} (-1)^i \operatorname{dim}_K H^i(G_E, V)=0.$
	\end{thm}
	
	\begin{proof}
	 See \cite[page 78, Theorem 3.9.1]{Berger_2013}. For the case $l=p$ and $K$ a local field of characteristic $p$, this result is also true and one can see \cite[Theorem 3.4.1]{Bockle_2023}. 
	\end{proof}
	
	\begin{thm}[Global Characteristic Formula] \label{global char for}
		Let $E$ be a number field, and $E_S$ the maximal extension of $E$ unramified outside a finite set of places $S$, containing the set of archimedean places $S_\infty$, and the finite places over $p$. Suppose that $K$ is a finite extension of $\mathbb{Q}_p$ and $G= G_{E,S}= \operatorname{Gal}(E_S/E)$.
		
		1) For any $i \ge 0$, we have $ \operatorname{dim}_K H^i(G_{E, S}, V) < \infty$. If further $i \ge 3$, then we have $\operatorname{dim}_K H^i(G_{E, S}, V) =0$.
		
		2) We have $\sum_{i=0}^{2} (-1)^i \operatorname{dim}_K H^i(G_{E, S}, V) = \sum_{v \in S_\infty} \operatorname{dim}_K V^{G_{E_v}} -[E: \mathbb{Q}]\cdot \operatorname{dim}_K V$.
	\end{thm}
	
	\begin{proof}
		See \cite[Lemma 9.7]{kisin2003overconvergent}.
	\end{proof}
	
	Now let $F$ be a number field, and $T$ a finite set of places of $F$ containing the finite places over $p$. Suppose $G=G_{F,T}$. Let $M$ be a $G$-stable $\mathcal{O}$-lattice of $V$, and $\rho: G \to \operatorname{GL}_n(\mathcal{O})$ the corresponding Galois representation, where $n= \operatorname{dim}_K V$. For simplicity, we write $\operatorname{ad}^0_m := \operatorname{ad}^0 \rho \otimes \mathcal{O}/\varpi^m$ and $\operatorname{ad}_m := \operatorname{ad} \rho \otimes \mathcal{O}/\varpi^m$.
	
	 We define the following complex as an analogue to \cite[Section 3.23]{Gee_2022} (or \cite[page 21]{Clozel_2008}). Let $\partial:C^i(G, \operatorname{ad}^0_m)\to C^{i+1}(G,\operatorname{ad}^0_m)$ be the usual coboundary map. Define $C^i_{T,\loc}(G,\ad^0_m)$
	 by $$ C^0_{T, \loc}(G,\ad^0_m)=\oplus_{v\in
	 	T}C^0(G_{F_v},\ad_m), ~~ C^1_{T,\loc}(G, \ad^0_m)=\oplus_{v \in T}C^1(G_{F_v},\ad^0_m),$$
	 	$$C^i_{T,\loc}(G,\ad^0_m)=\oplus_{v\in
	 	T}C^i(G_{F_v},\ad^0_m), i \ge 2.$$
	
	Let $C^0_0(G,\ad^0_m):=C^0(G,\ad_m)$, and let
	$C^i_0(G,\ad^0_m)=C^i(G,\ad^0_m)$ for $i>0$. Then we let $H^i_{T}(G,\ad^0_m)$ denote the cohomology of the complex $$ C^i_{T}(G, \ad^0_m):=C_0^i(G, \ad^0_m)\oplus C^{i-1}_{T,\loc}(G, \ad^0_m),$$ 
	where the coboundary map is given by $ (\phi,(\psi_v))\mapsto(\partial
	\phi,(\phi|_{G_{F_v}}-\partial\psi_v))$. Now we get an exact sequence of complexes $$ 0\to C^{i-1}_{T,\loc}(G,\ad^0_m) \to C^i_{T}(G,\ad^0_m)\to C^i_0(G,\ad^0_m) \to 0,$$
	and the corresponding long exact sequence in cohomology is
	
	$$ 0 \to H^0_{T}(G,\ad^0_m) \to H^0(G, \ad_m) \to \bigoplus_{v \in T}H^0(G_{F_v} , \ad_m) \to $$
	$$ H^1_{T}(G, \ad^0_m) \to H^1(G, \ad^0_m) \to \bigoplus_{v\in T}H^1(G_{F_v}, \ad^0_m) \to $$
    $$ H^2_{T}(G, \ad^0_m) \to H^2(G, \ad^0_m) \to \bigoplus_{v \in T}H^2(G_{F_v},\ad^0_m) \to H^3_{T}(G, \ad^0_m) \to \cdots .$$
	We can also replace $ \ad^0_m$ by $ \ad^0 \rho \otimes K$ and obtain exact sequences similarly.
	
	From Theorem \ref{local char for} and Theorem \ref{global char for}, we have $$H^i(G, \ad^0 \rho \otimes K)=0,~~~ H^i(G_{F_v}, \ad^0 \rho \otimes K)  =0,~~  i \ge 3,  v \in T.$$
	Hence, we get an exact sequence
	
	$$ H^1(G, \ad^0 \rho \otimes K) \to \bigoplus_{v\in T}H^1(G_{F_v}, \ad^0 \rho \otimes K) \to H^2_{T}(G, \ad^0 \rho \otimes K) \to $$
	$$ H^2(G, \ad^0 \rho \otimes K) \to \bigoplus_{v \in T}H^2(G_{F_v},\ad^0 \rho \otimes K) \to H^3_{T}(G, \ad^0 \rho \otimes K) \to 0 .$$
	
	Using the Poitou-Tate duality (\cite[Theorem 8.6.14]{Neukirch_2008}) and the Five Lemma, we have $ H^2_{T}(G, \ad^0 \rho) \cong H^1(G, (\ad^0 \rho)')^\vee$, where $(\ad^0 \rho)'= \operatorname{Hom}(\ad^0 \rho, \mu_{p^\infty}) $. Hence, we have 
	$$ \operatorname{dim}_K H^2_{T}(G, \ad^0 \rho \otimes K)  = \operatorname{rank}_\mathcal{O} H^1(G, (\ad^0 \rho)')^\vee = \operatorname{dim}_K H^1(G, \Hom(\ad^0 \rho \otimes K, K)(1) ).$$
	The last equality is obtained by taking the long exact sequence of the following sequence
	$$ 0 \to \Hom(\ad^0 \rho, \cO)(1) \to \Hom(\ad^0 \rho \otimes K, K)(1) \to \Hom(\ad^0 \rho, K/\cO)(1) \to 0.$$
	
	Using the Poitou-Tate duality(\cite[Theorem 8.6.10]{Neukirch_2008}), we have $H^3_{T}(G, \ad^0 \rho \otimes K) \cong (\varprojlim_m H^0(G, \ad^0_m(1))^\vee) \otimes K$. Here we use \cite[Appendix B, Proposition 2.3, 2.4 \& 2.7]{rubin2000euler} and the Mittag-Leffler condition (see \cite[Lemma 10.86.4]{stacks-project}).
	
	Now we can deduce a generalized Greenberg-Wiles formula. For the usual version, one can see \cite[Lemma 2.3.4]{Clozel_2008} or \cite[Theorem 8.7.9]{Neukirch_2008}.
	
	\begin{prop}\label{Greenberg-Wiles}
		Write $W= \Hom(\ad^0 \rho \otimes K, K)(1)$. Suppose that $\rho \otimes K$ is irreducible, and $ H^0(G, W)=0$.  Then we have 
		$$ \operatorname{dim}_K H^1_{T}(G, \ad^0 \rho \otimes K)- \operatorname{dim}_K H^1(G, W) = -1+|T| - \sum_{v \mid \infty} \operatorname{dim}_K H^0(G_{F_v}, \ad^0 \rho \otimes K).$$
	\end{prop}
	
	\begin{proof}
		As $T$ is not empty, we have $H^0_{T}(G,\ad^0 \rho \otimes K)=0$. Note that we have 
		$$\dim_K H^0(G, \ad \rho \otimes K) - \dim_K H^0(G, \ad^0 \rho \otimes K)=1, $$ 
		 $$\dim_K H^0(G_{F_v} , \ad \rho \otimes K) -  \dim_K H^0(G_{F_v}, \ad^0 \rho \otimes K)=1, ~~~v \in T. $$
		 Then our result follows from Theorem \ref{local char for}, Theorem \ref{global char for} and the previous discussions.
	\end{proof}
	
	\subsection{Global Galois deformation theory} We obtain dimension and connectedness bounds for completed local rings at points of characteristic zero. The ordinary case will be used in Subsection \ref{sec special ordinary} to overcome the failure of residual adjoint $H^0$-vanishing. For the classical deformation-theoretic framework, see B\"ockle's article in \cite{Berger_2013}.
	
	Let $p$ be an odd prime. Let $F$ be a totally real field such that $p$ splits completely, and $T$ a finite set of places of $F$ containing the finite places over $p$. Let $K$ be a $p$-adic local field with its ring of integers $\mathcal{O}$, a uniformizer $\varpi$ and its residue field $\mathbb{F}$. Let $\chi :G_{F,T} \to \mathcal{O}^\times$ be a continuous character, and we suppose that $\bar{\chi}$ is totally odd.
	
	Let $R^\ps$ be the universal pseudo-deformation ring parametrizing all liftings of the pseudo-representation $1+ \bar{\chi}$ with determinant $\chi$. Let $\mathfrak{p} \in \operatorname{Spec} R^\ps$ satisfy $\dim R^\ps/\kp=1$ (we do not require $\varpi \notin \kp$ here) and suppose that it defines an irreducible pseudo-representation. Then by Definition \ref{construction}, we can construct an absolutely irreducible Galois representation $\rho(\mathfrak{p})^\circ: G_{F, T} \to \operatorname{GL}_2(A)$, where $A$ is the normalization of $R^\ps/\mathfrak{p}$, and hence a finite discrete valuation ring. By \cite[Corollary 2.14]{Skinner_1999}, we may write $\bar{\rho}_b = \rho(\mathfrak{p})^\circ \mod \mathfrak{m}_A$, where $ 0 \ne b \in H^1(G_{F, T}, \mathbb{F}(\bar{\chi}^{-1}))$. Moreover, $b$ is unique up to multiplication by a non-zero scalar.
	
	%Write $\bar{\rho}_b = \rho(\mathfrak{p})^\circ \mod \mathfrak{m}_A$, where $ 0 \ne b \in H^1(G_{F, T}, \mathbb{F}(\bar{\chi}^{-1}))$. Then by \cite[Lemma 2.13]{Skinner_1999}, up to some non-zero scalars of $\mathbb{F}$, $b$ is independent of the choice of $\rho(\mathfrak{p})^\circ$ given in Definition \ref{construction}.
	
	To study the local ring $R^\ps_\mathfrak{p}$, we compare it with the deformation ring of $\bar\rho_b$ using Proposition \ref{ps-def comparison}. This comparison of completed local rings applies in both characteristics.
	 
	\begin{defn} \cite[Definition 5.4.3]{pan2022fontaine}
		Let $R$ be a CNL ring. The connectedness dimension of $R$ is defined to be 
		$$c(R):=\min_{Z_1,Z_2}\{\dim (Z_1\cap Z_2)\}, $$
		where $Z_1,Z_2$ are non-empty unions of irreducible components of $\Spec R$ such that $Z_1\cup Z_2=\Spec R$.
	\end{defn}
	
	\begin{prop}\label{connect dim}
		1) $\operatorname{dim}  R^\ps_\mathfrak{p} \ge 2[F: \mathbb{Q}]$.
		
		2) The connectedness dimension $c(R^\ps_\mathfrak{p})$ is at least $2[F: \mathbb{Q}]-1$.
	\end{prop}
	
	\begin{proof}
		We follow \cite[Lemma 7.4.6]{pan2022fontaine}. Let $R_b$ be the universal deformation ring of $\bar\rho_b$ with determinant $\chi$, and let $\mathfrak{p}_b$ be the prime defined by $\rho(\mathfrak{p})^\circ$. By Proposition \ref{ps-def comparison}(1), $\widehat{(R^\ps)_\mathfrak{p}}\cong\widehat{(R_b)_{\mathfrak{p}_b}}$. Write $h_i=\dim_\F H^i(G_{F,T},\ad^0\bar\rho_b)$. The usual presentation gives $R_b\cong\mathcal{O}[[x_1,\ldots,x_{h_1}]]/(f_1,\ldots,f_r)$ with $r\le h_2$, and the global Euler characteristic formula gives $h_1-h_2=2[F:\Q]$ since $\bar\rho_b$ is non-split and totally odd. Completing the power series ring at the pull-back of $\mathfrak{p}_b$ gives a complete regular local ring of dimension $h_1$. Thus $\dim\widehat{(R_b)_{\mathfrak{p}_b}}\ge2[F:\Q]$, and \cite[Proposition 5.4.4]{pan2022fontaine} gives $c(\widehat{(R_b)_{\mathfrak{p}_b}})\ge2[F:\Q]-1$. The claims follow from invariance of dimension under completion and $c((R^\ps)_\mathfrak{p})\ge c(\widehat{(R^\ps)_\mathfrak{p}})$, as in the proof of \cite[Lemma 7.4.6]{pan2022fontaine}.
	\end{proof}
	
	In the rest of this subsection, we study the universal deformation ring of the Galois representation $\rho(\mathfrak{p}) := \rho(\mathfrak{p})^\circ \otimes k(\mathfrak{p})$. Suppose that $\varpi \notin \kp$, $\bar\chi|_{G_{F_v}}\neq \mathbf{1}$ for every $v|p$, and $\rho(\mathfrak{p})$ is nearly ordinary at each $v|p$ in the sense of \cite[Definition 2.2]{Calegari_2008}. For every $v|p$, fix a $G_{F_v}$-stable line $L_v\subset k(\mathfrak{p})^2$, which we regard as the special line in the nearly ordinary structure. Throughout this subsection, by ``ordinary'' we mean ``nearly ordinary''; in particular, we impose no unramifiedness condition on the character acting on $L_v$ or on the quotient.
	
	For any finite place $v \in T$, denote by $D_{v}^{\square, \chi} : \mathfrak{U}_{k(\mathfrak{p})} \to \operatorname{Sets}$ the functor such that for $ (B, \mathfrak{m}_B) \in \mathfrak{U}_{k(\mathfrak{p})}$, $ D_{v}^{\square, \chi}(B)$ is the set of continuous representations $\rho_B : G_{F_v} \to \operatorname{GL}_2(B)$ satisfying $ \rho_B \mod \mathfrak{m}_B = \rho(\mathfrak{p})|_{G_{F_v}}$ and $\operatorname{det} \rho_B= \chi |_{G_{F_v}}$. Using the same proof as \cite[page 27, Proposition 1.3.1]{Berger_2013}, we can show that this functor is pro-representable by some $R_{v, k(\mathfrak{p})}^{\square, \chi}$, which is a CNL $ k(\mathfrak{p})$-algebra. 
	
	For any $v|p$ and $B\in\mathfrak{U}_{k(\mathfrak{p})}$, an ordinary framed lifting of $(\rho(\mathfrak{p})|_{G_{F_v}},L_v)$ to $B$ is a pair $(\rho_B,L_{v,B})$, where $\rho_B\in D_v^{\square,\chi}(B)$ and $L_{v,B}\subset B^2$ is a $G_{F_v}$-stable rank one direct summand such that $L_{v,B}\otimes_B k(\mathfrak{p})=L_v$. The resulting functor is pro-represented by a CNL $k(\mathfrak{p})$-algebra, which we denote by $R_{v,k(\mathfrak{p})}^{\triangle,\chi}$. See \cite[Definitions 2.1--2.3 and Section 2.2]{Calegari_2008}. Since the two characters on $L_v$ and $k(\mathfrak{p})^2/L_v$ are distinct, a stable line lifting the prescribed line $L_v$ is uniquely determined whenever it exists; we will therefore usually suppress the line from the notation.
	
	The following result is standard using local characteristic formula Theorem \ref{local char for}.
	
	\begin{prop} \label{loc dim}
		Let $v$ be a finite place in the set $T$.
		
		1) If $v \nmid p$, then we have $\operatorname{dim} R_{v, k(\mathfrak{p})}^{\square, \chi} \ge 3$.
		
		2) If $v| p$, then we have $\operatorname{dim} R_{v, k(\mathfrak{p})}^{\triangle, \chi} \ge 5$.
	\end{prop}
	
	\begin{proof}
		The proof is almost the same as in the finite-residue-field case. In that case, one can see \cite[page 64, Theorem 3.3.1]{Berger_2013} for $v \nmid p$ and \cite[Lemma 2.2]{Skinner_1999} for $v \mid p$.
	\end{proof}
	
	Let $S \subset T$ be a finite set of primes of $F$. Define the functor $D_{k(\mathfrak{p})}^{\square_S}$ from $ \mathfrak{U}_{k(\mathfrak{p})}$ to the category of sets sending $R$ to the set of tuples $(\rho_R,\{L_{v,R}\}_{v|p},(\alpha_v)_{v \in S})$ modulo the equivalence relation $\sim_S$ where
	\begin{itemize}
		\item $\rho_R:G_{F, T}\to \GL_2(R)$ is a lifting of $\rho(\mathfrak{p})$ to $R$ with determinant $\chi$.
		\item For every $v|p$, $L_{v,R}\subset R^2$ is a $G_{F_v}$-stable rank one direct summand lifting $L_v$.
		\item $\alpha_v\in 1+M_2(\km_R),v\in S$. Here $\km_R$ is the maximal ideal of $R$.
		\item $(\rho_R,\{L_{v,R}\}_{v|p},\{\alpha_v\}_{v\in S})\sim_S (\rho'_R,\{L'_{v,R}\}_{v|p},\{\alpha'_v\}_{v\in S})$ if there exists an element $\beta\in 1+M_2(\km_R)$ with $\rho'_R=\beta\rho_R\beta^{-1}$, $L'_{v,R}=\beta L_{v,R}$ for every $v|p$, and $\alpha'_v=\beta\alpha_v$ for any $v\in S$.
	\end{itemize}
	Then the functor $D_{k(\mathfrak{p})}^{\square_S}$ is pro-representable by a similar argument of \cite[page 102, Proposition 5.1.1]{Berger_2013}. Note that in the case $S=\varnothing$, we need to use the fact that $\rho(\mathfrak{p})$ is absolutely irreducible. We write $R_{k(\mathfrak{p})}^{\square_S, \chi}$ for the universal $S$-framed deformation ring and $R_{k(\mathfrak{p})}^{\operatorname{univ}}$ in the case $S=\varnothing$. As the lines are uniquely determined by the corresponding liftings, we will omit them from the notation. Write $(\ad^0 \rho)^*=  \Hom(\ad^0 \rho, k(\kp))$. Then the following result is also standard.
	
	\begin{prop} Write $R_{\operatorname{loc}} = (\hat{\bigotimes}_{v | p} R_{v, k(\mathfrak{p})}^{\triangle, \chi}) \hat{\otimes}_{k(\mathfrak{p})} (\hat{\bigotimes}_{v \nmid p, v \in T} R_{v, k(\mathfrak{p})}^{\square, \chi})$. 
		
		1) $ R_{k(\mathfrak{p})}^{\square_T, \chi} \cong R_{k(\mathfrak{p})}^{\operatorname{univ}}[[x_1, ..., x_{4|T|-1}]]$.
		
		2) There exists an isomorphism $ R_{\operatorname{loc}} [[y_1, ..., y_r]] /(f_1, ..., f_t) \cong R_{k(\mathfrak{p})}^{\square_T, \chi}$, where 
		$$r= \operatorname{dim}_{k(\mathfrak{p})} H^1_{T}(G, \ad^0 \rho \otimes k(\mathfrak{p})),~~~~ t \le \operatorname{dim}_{k(\mathfrak{p})} H^1(G, (\ad^0 \rho)^*(1)).$$
	\end{prop}
	
	\begin{proof}
		The proof is also the same as in the finite-residue-field case. For 1), one can see \cite[page 103, Corollary 5.2.1]{Berger_2013}. For 2), one can see \cite[Lemma 4.4 \& 4.6]{Khare_2009}.
	\end{proof}
	
	\begin{cor} \label{completion}
		Assume $ H^0(G_{F, T}, \ad^0 \rho(\mathfrak{p})(1))=0$.
		We have $\operatorname{dim} R_{k(\mathfrak{p})}^{\operatorname{univ}} \ge [F : \mathbb{Q}]$.
	\end{cor}
	
	\begin{proof}
		At first, by Proposition \ref{loc dim} and Krull's principal ideal theorem, we have $\dim R_\loc \ge 3|T| +2 [F: \mathbb{Q}]$. Then combining the second part of the previous proposition and Proposition \ref{Greenberg-Wiles}, we have $$\dim R_{k(\mathfrak{p})}^{\square_T, \chi} \ge 3|T| +2 [F: \mathbb{Q}]+ (-1 +|T| - \sum_{v \mid \infty} \operatorname{dim}_K H^0(G_{F_v}, \ad^0 \rho \otimes K)) = 4|T|-1+[F:\mathbb{Q}].$$
		Here the last equality is because $\bar{\chi}$ is totally odd. Using the first part of the previous proposition, we have $\operatorname{dim} R_{k(\mathfrak{p})}^{\operatorname{univ}} \ge [F : \mathbb{Q}]$.
	\end{proof}
	
	For every $v|p$, let $L_v^\circ=L_v\cap A^2$. Since $A$ is a discrete valuation ring, $L_v^\circ$ is a saturated $G_{F_v}$-stable rank one submodule of $A^2$, and hence a direct summand. Let $\bar L_v=L_v^\circ\otimes_A\mathbb{F}\subset\mathbb{F}^2$. We can similarly define the universal deformation ring $R_b^{\operatorname{univ}} $ parametrizing all ordinary liftings of $(\bar{\rho}_b,\{\bar L_v\}_{v|p})$ with determinant $\chi$. Recall that $A$ is the normalization of $R^\ps/\kp$, and hence its fraction field is $k(\kp)$. By the universal property, we get a homomorphism $ R_b^{\operatorname{univ}} \to A$. Let $\mathfrak{q}$ be the kernel of this map, and we assume the residue field of $ \widehat{(R_b^{\operatorname{univ}})_{\mathfrak{q}}}$ is $K'$, which is contained in $k(\kp)$. Let $\rho_{k(\mathfrak{p})}^{\operatorname{univ}}$ (resp. $\rho_b^{\operatorname{univ}}$) be the universal lifting. Then $\hat{\rho}=\rho_b^{\operatorname{univ}} \otimes _{R_b^{\operatorname{univ}}} (\widehat{(R_b^{\operatorname{univ}})_{\mathfrak{q}}} \otimes_{K'}  k(\mathfrak{p}))$, together with the induced ordinary lines, defines an ordinary lifting of $(\rho(\mathfrak{p}),\{L_v\}_{v|p})$. By the universal property, we get a natural map $R_{k(\mathfrak{p})}^{\operatorname{univ}} \to \widehat{(R_b^{\operatorname{univ}})_{\mathfrak{q}}} \otimes_{K'} k(\mathfrak{p})$.
	
	\begin{prop} \label{Kisin}
		The map $f: R_{k(\mathfrak{p})}^{\operatorname{univ}} \to \widehat{(R_b^{\operatorname{univ}})_{\mathfrak{q}}} \otimes_{K'} k(\mathfrak{p})$ is actually an isomorphism.
	\end{prop}
	
	\begin{proof}
		Using \cite[Proposition 3.1.5]{Bockle_2023}, we just need to show that $f$ is formally \'etale. The proof is the nearly ordinary variant of \cite[Proposition 9.5]{kisin2003overconvergent}; see also \cite[Section 2.2, Remark on page 112]{Calegari_2008} and \cite[Theorem 3.3.1]{Bockle_2023}. We briefly sketch it here, keeping track of the ordinary lines. To simplify our notation, we may suppose $K=K'=k(\mathfrak{p})$ (using \cite[Lemma 3.2.6]{Bockle_2023}) and write $\widehat{R}=\widehat{(R_b^{\operatorname{univ}})_{\mathfrak{q}}} \otimes_{K'}  k(\mathfrak{p})$.
		
		We consider a commutative diagram with $C$ an Artinian local $K$-algebra and $I$ a square zero ideal. We need to show that there exists a unique dashed arrow $g$ such that the two triangular subdiagrams commute.
		
		$$\begin{tikzcd} 
			R_{K}^{\operatorname{univ}} \arrow[r] \arrow[d, "f"] & C \arrow[d, two heads]\\
			\widehat{R} \arrow[ru, dashed ,"g"] \arrow[r] & C/I\\
		\end{tikzcd}$$
		
		By possibly conjugating $\hat{\rho}$ and the associated ordinary lines by some matrix $X \in \GL_2(\widehat{R})$ satisfying $X \equiv \operatorname{id} \mod \mathfrak{m}_{\widehat{R}}$, we can suppose that $\rho_{K}^{\operatorname{univ}} \otimes_{R_{K}^{\operatorname{univ}}} \widehat{R} =\hat{\rho}$ as ordinary deformations.
		
		Choose a finitely generated free $\mathcal{O}$-subalgebra $C^\circ \subset C$ surjecting onto $\mathcal{O}$ under the natural surjection $c: C \twoheadrightarrow K$. Set $\mathfrak{n} = \operatorname{ker} c$ and $\mathfrak{n}^\circ = \mathfrak{n} \cap C^\circ$. For any $n \ge 0$, define $C_n^\circ = \sum_{j >0} \varpi^{-nj}(\mathfrak{n}^\circ)^j + C^\circ \subset C$. Note that the sum is finite since $ \mathfrak{n}^\circ$ is nilpotent. Then we have $\cup_{n >0} C_n^\circ = c^{-1}(\mathcal{O})$. We write $(C/I)_n^\circ= C_n^\circ/(I \cap C_n^\circ)$. As $G_{F, T}$ is compact, we know that $\rho_{K}^{\operatorname{univ}} \otimes_{R_{K}^{\operatorname{univ}}} C$ factors through $\operatorname{GL}_2(C_n^\circ)$ for $n$ sufficiently large.
		
		After increasing $n$, we may also assume that the ordinary lines descend to $C_n^\circ$. Indeed, let $L_{v,C}\subset C^2$ be the ordinary line associated with the deformation defined by the upper horizontal map. Choose an $\mathcal{O}$-basis of $\mathcal{O}^2$ adapted to $L_v^\circ$. In this basis, we may write $L_{v,C}=C(e_1+a_ve_2)$ for some $a_v\in\mathfrak{n}$. Since $\cup_{n>0}C_n^\circ=c^{-1}(\mathcal{O})$, all the finitely many elements $a_v$ belong to $C_n^\circ$ for $n$ sufficiently large. Thus $L_{v,n}^\circ:=C_n^\circ(e_1+a_ve_2)$ is a rank one direct summand of $(C_n^\circ)^2$ whose base change to $C$ is $L_{v,C}$. Since the representation takes values in $\GL_2(C_n^\circ)$ and $L_{v,C}$ is $G_{F_v}$-stable, $L_{v,n}^\circ$ is also $G_{F_v}$-stable.
		
		Note that $R_b^{\operatorname{univ}}$ is topologically finitely generated, the composite $R_b^{\operatorname{univ}} \to \widehat{R} \to C/I$ factors through $ (C/I)_n^\circ$ for $n$ sufficiently large. 
		
		From our discussions above, we know that for sufficiently large $n$, $\rho_{K}^{\operatorname{univ}} \otimes_{R_{K}^{\operatorname{univ}}} C$, together with the lines $\{L_{v,n}^\circ\}_{v|p}$, defines an ordinary deformation over $C_n^\circ$, and it reduces to the ordinary deformation $\rho_b^{\operatorname{univ}} \otimes_{R_b^{\operatorname{univ}}} (C/I)_n^\circ$ modulo $(I \cap C_n^\circ)$. Here the compatibility of the ordinary lines follows from the commutativity of the diagram and, equivalently, from the uniqueness of a stable line lifting the prescribed line under the distinguished hypothesis. By the universal property, there exists a unique map $g_n^\circ : R_b^{\operatorname{univ}} \to C_n^\circ$ such that the corresponding ordinary deformations are equivalent. Taking the completion, we obtain a map $g: \widehat{R} \to C$, and we can check that the diagram commutes with such $g$.  The uniqueness of $g$ follows from the uniqueness of $g_n^\circ$.
	\end{proof}
	
	Combining Corollary \ref{completion} and Proposition \ref{Kisin}, we obtain the following result.
	
	\begin{cor}\label{crucial}
		Assume $ H^0(G_{F, T}, \ad^0 \rho(\mathfrak{p})(1))=0$. Then there exists an irreducible component of $\Spec R_b^{\operatorname{univ}}$ containing $\mathfrak{q}$ of dimension at least $1+[F:\mathbb{Q}]$.
	\end{cor}
	
	\section{Local-global compatibility}\label{sec 3}
	We extend the local-global compatibility results of Pan \cite[Section 3]{pan2022fontaine} to the exceptional case $p=3$, using Pašk\=unas-Tung \cite{Pa_k_nas_2021}. After recalling completed cohomology and Hecke algebras, we obtain two inputs for the global argument: the module-theoretic compatibility needed for patching, and the classicality criterion which converts the relevant pro-modular points into automorphic points. The required results of Pašk\=unas theory are recalled in Appendix \ref{A}.

	Throughout this section, we fix a prime $p$ and a totally real field $F$ such that $p$ splits completely. We fix an isomorphism $\iota_p: \overbar{\Q_p} \cong \mathbb{C}$. We denote by $K$  a $p$-adic local field with its ring of integers $\cO$, a uniformizer $\varpi$ and its residue field $\F$.
	
	\subsection{Completed cohomology} In this subsection, we recall the definition of completed cohomology and the relation between it and classical automorphic forms. For more details, one can see \cite[Section 3.1 \& 3.2]{pan2022fontaine}.
	
	Let $D$ be a quaternion algebra with centre $F$ which is ramified at all infinite places of $F$ and unramified at all places above $p$, and we fix isomorphisms $D \otimes F_v \cong M_2(F_v)$ for any $v$ where $D$ is unramified. We view $K_p = \prod_{v \mid p} \operatorname{GL}_2(O_{F_v})$ and $D_p^\times =  \prod_{v \mid p} \operatorname{GL}_2(F_v)$ as subgroups of $(D \otimes_F \mathbb{A}_F)^\times$. We also write $N_{D/F}:(D\otimes_F\A_F)^\times\to \A_F^\times$ as the reduced norm.
	
	Let $A$ be a topological $\mathcal{O}$-algebra and $U = \prod_{v \nmid \infty} U_v$ be an open compact subgroup of $(D \otimes_F \mathbb{A}_F^\infty)^\times$ such that $U_v \subseteq \operatorname{GL}_2(O_{F_v})$ for $v \mid p$. We write $U^p= \prod_{v \nmid p} U_v$ (tame level) and $U_p =\prod_{v \mid p} U_v$. Let $\psi: (\mathbb{A}_F^\infty)^\times/F_+^\times \to A^\times$ be a continuous character, where $F_+$ is the set of totally positive elements in $F$. Let $\tau: \prod_{v \mid p} U_v \to \operatorname{Aut}(W_\tau)$ be a continuous representation on a finite $A$-module $W_\tau$, and we also view $\tau$ as a representation of $U$ by projecting to $\prod_{v \mid p} U_v$. 
	
	Let $S_{\tau,\psi}(U,A)$ be the space of continuous functions:
	\[f:D^{\times}\setminus \DAi\to W_{\tau}\]
	such that for any $g\in \DAi,u\in U,z\in \AFi$, we have
	\begin{itemize}
		\item $f(gu)=\tau(u^{-1})(f(g))$,
		\item $f(gz)=\psi(z)f(g)$.
	\end{itemize}
	Write $\DAi=\bigsqcup_{i\in I}D^\times t_iU\AFi$ for some finite set $I$ and $t_i\in \DAi$. If $\tau^{-1}|_{U\cap \AFi}$ acts as $\psi|_{U\cap \AFi}$, then there is an isomorphism:
	\begin{eqnarray*}
		S_{\tau,\psi}(U,A)\simeq \bigoplus_{i\in I}W_{\tau}^{(t_i^{-1}D^\times t_i\cap U\AFi)/F^{\times}},
	\end{eqnarray*}
	by sending $f$ to $(f(t_i))_i$. We say $U$ is \textit{sufficiently small} if $(t_i^{-1}D^\times t_i\cap U\AFi)/F^{\times}$ is trivial for all $i\in I$. This can be achieved by shrinking $U_v$ for some $v$.
	
	Suppose that $\xi: U^p \to A^\times$ is a continuous smooth character such that $\psi \mid_{\prod_{v \nmid p}(O_{F_v}^\times) \cap U_v}= \xi \mid_{\prod_{v \nmid p}(O_{F_v}^\times) \cap U_v}$. We define $S_{\tau, \psi, \xi}(U,A)$ as the space of continuous functions: $f: D^\times \setminus (D \otimes_F \mathbb{A}_F^\infty)^\times \to W_\tau$, such that for any $g \in (D \otimes_F \mathbb{A}_F^\infty)^\times$, $z \in (\mathbb{A}_F^\infty)^\times$, $u=u^pu_p \in U$, we have $$
	f(guz)= \psi(z)\xi(u^p)\tau(u_p^{-1})(f(g)).$$ 
	Similarly, if $\psi \mid_{U_p \cap \mathcal{O}_{F, p}^\times} = \tau^{-1} \mid_{U_p \cap \mathcal{O}_{F, p}^\times}$, then we also have $$
	S_{\tau, \psi, \xi}(U,A) \simeq \bigoplus_{i \in I} W_\tau ^{(t_i^{-1}D^{\times}t_i \cap U (\mathbb{A}_F^\infty)^\times)/F^\times},$$
	where $I=D^\times \setminus (D \otimes_F \mathbb{A}_F^\infty)^\times/ U(\mathbb{A}_F^\infty)^\times$ and $\{t_i\}_{i \in I}$ is a set of representatives.
	
	 We will ignore $\tau$ in the notation if it is the trivial action on $A$. 
	
	If $A$ is a torsion $\cO$-algebra with discrete topology, we write $S_{\psi}(U^p,A):=\varinjlim_{U_p} S_\psi(U^pU_p,A) $ with discrete topology, where $U_p=\prod_{v|p}U_v$ runs over all open compact subgroups $U_v$ of $\GL_2(F_v)$. 
	
	Note that the definition of $S_\psi(U,A)$ makes sense for any topological $\cO$-module $A$. We will also use $S_\psi(U,A)$ for any topological $\cO$-module $A$ by abuse of notation.
	
	\begin{defn} \label{Completed coh}
	 The \textit{completed cohomology} of tame level $U^p$ is defined to be
		\[S_\psi(U^p):=\Hom_\cO(K/\cO,S_{\psi}(U^p,K/\cO))\]
		equipped with $\varpi$-adic topology, and the \textit{completed homology} is defined to be
		\[M_\psi(U^p):=S_\psi(U^p,K/\cO)^\vee=\Hom_\cO(S_{\psi}(U^p,K/\cO),K/\cO)\]
		equipped with compact-open topology.
		
		We can also define $S_{\psi, \xi}(U^p)$ and $M_{\psi, \xi}(U^p) $ in a similar way.
	\end{defn}
	From the definition above, there is a natural action of $D_p^\times=\prod_{v|p}\GL_2(F_v)$ on all spaces defined above by right translation. It is almost by definition that for any open compact subgroup $U_p=\prod_{v|p}U_v\subseteq K_p=\prod_{v|p}\GL_2(O_{F_v})$,
	\[S_{\psi}(U^p,A)^{U_p}=S_{\psi}(U^pU_p,A).\]
	Hence $S_\psi(U^p,K/\cO)$ is a smooth admissible $\cO$-representation of $D_p^\times$. It is also clear that $S_\psi(U^p)$ is $\varpi$-torsion free and
	\begin{eqnarray*}
		S_\psi(U^p)\cong \varprojlim_n S_\psi(U^p,\cO/\varpi^n)\cong \Hom^{\mathrm{cont}}_{\cO}(M_\psi(U^p),\cO),~M_\psi(U^p)\cong \Hom_\cO(S_\psi(U^p),\cO).
	\end{eqnarray*}
	
	\begin{rem}\label{quaternionic and automorphic}
		 Suppose $A=K$ and $(\vec{k},\vec{w})\in \Z_{>1}^{\Hom(F,\overbar{\Q_p})}\times  \Z^{\Hom(F,\overbar{\Q_p})}$ such that $k_\sigma+2w_\sigma$ is independent of $\sigma:F\to \overbar{\Q_p}$. Write $w=k_{\sigma}+2w_{\sigma}-1$. We can define the following algebraic representation $\tau_{(\vec{k},\vec{w})}$ of $D_p^\times=(D\otimes \Q_p)^\times$ on
		\[W_{(\vec{k},\vec{w}),K}=\bigotimes_{\sigma:F\to K}(\Sym^{k_\sigma-2}(K^2)\otimes \det{}^{w_\sigma}),\]
		where $\Sym^{k_\sigma-2}$ denotes the space of homogeneous polynomials of degree $k_\sigma-2$ in two variables with an action of $\GL_2(F_{v(\sigma)})$ given by
		\[\begin{pmatrix} a&b\\c&d\end{pmatrix}(f)(X,Y)=f(\sigma(a)X+\sigma(c)Y,\sigma(b)X+\sigma(d)Y).\] 
		Here $v(\sigma)$ is the place above $p$ given by $\sigma$.  Let $\psi:(\A^\infty_F)^\times/F^\times_{+}\to K^\times$ be a continuous character such that $\tau_{(\vec{k},\vec{w})}^{-1}|_{U\cap \AFi}=\psi|_{U\cap \AFi}$ and write $\tau=\tau_{(\vec{k},\vec{w})}$. For simplicity, we will write $S_{(\vec{k},\vec{w}),\psi}(U,K)$ for $S_{\tau_{(\vec{k},\vec{w})},\psi}(U,K)$ from now on. Then there is an isomorphism:
		\begin{eqnarray*}
			S_{(\vec{k},\vec{w}),\psi}(U,K)\otimes_{K,\iota_p}\bC&\stackrel{\sim}{\longrightarrow} &\Hom_{D_\infty^\times}(W_{\iota_p(\vec{k},\vec{w}),\bC}^*,C^\infty(D^\times\setminus (D\otimes\A_F)^\times/U,\psi_\bC)),\\
			f\otimes 1&\longmapsto& ``w^*\mapsto (g\mapsto w^*(\tau_\bC(g_\infty)^{-1}\tau(g_p)f(g^\infty)))"
		\end{eqnarray*}
		where 
		\begin{itemize}
			\item $D_\infty^\times=(D\otimes_\Q \R)^\times$, 
			\item $W_{\iota_p(\vec{k},\vec{w}),\bC}=W_{(\vec{k},\vec{w}), K}\otimes_{K,\iota_p}\bC$ is viewed as an algebraic representation $\tau_\bC$ of $D_\infty^\times$ (induced by $\iota_p$), and $W_{\iota_p(\vec{k},\vec{w}),\bC}^*$ is its $\bC$-linear dual.
			\item $\psi_\bC:\A_F^\times/F^\times\to\bC^\times$ sends $g$ to $N_{F/\Q}(g_{\infty})^{1-w}\iota_p(N_{F/\Q}(g_p)^{w-1}\psi(g^\infty))$.
			\item $C^\infty(D^\times\setminus (D\otimes\A_F)^\times/U,\psi_\bC)$ is the space of smooth $\bC$-valued functions on $D^\times\setminus(D\otimes\A_F)^\times$, right invariant by $U$ and with central character $\psi_\bC$.
		\end{itemize}
		Note that the right hand side is a subspace of automorphic forms on $(D\otimes\A_F)^\times$.
		
		Furthermore, write $S_\psi(U^p)_K$ for $S_\psi(U^p)\otimes_\cO K$, which is a Banach space with unit ball $S_\psi(U^p)$. Then we have an isomorphism
		\[S_{(\vec{k},\vec{w}),\psi}(U^pU_p,K)\simeq \Hom_{K[U_p]}(W^{*}_{(\vec{k},\vec{w}),K},S_\psi(U^p)_K).\]
		See \cite[3.1.2 \& 3.2.4]{pan2022fontaine}.
	\end{rem}
	
	Let $v|p$ be a finite place of $F$. We assume $\psi|_{N_{D/F}(U^p)}$ is trivial and $\psi|_{O_{F_v}^\times}$ is an algebraic character. In particular, there exists an integer $m$ such that,
	$\psi(a_v)=\sigma_v(a_v)^m,~a_v\in O_{F_v}^\times$,
	where $\sigma_v:F\to K$ is the embedding induced by $v$. Let $U^v$ be an open compact subgroup of $\prod_{w\neq v,w|p}\GL_2(O_{F_w})$.
	
	\begin{defn}\label{locally algebraic vectors}
		The subspace of $\GL_2(O_{F_v})U^v$-\textit{algebraic vectors} of $S_\psi(U^p)_K$ is defined to be the image of the evaluation map:
		\[\bigoplus_{(\vec{k},\vec{w})}\Hom_{K[\GL_2(O_{F_v})U^v]}(W_{(\vec{k},\vec{w}),K},S_\psi(U^p)_K)\otimes_K W_{(\vec{k},\vec{w}),K}\to S_\psi(U^p)_K,\]
		where the sum is taken over all $(\vec{k},\vec{w})$ with $k_\sigma+2w_\sigma=m+2$ for any $\sigma$. The subspace $S_\psi(U^p)_K^{v-\mathrm{a},v'-\mathrm{la}}$ of $\GL_2(O_{F_v})$-\textit{algebraic}, $\prod_{w\neq v,w|p}\GL_2(O_{F_w})$-\textit{locally algebraic vectors} of $S_\psi(U^p)_K$ is defined to be the union of all $\GL_2(O_{F_v})U^v$-algebraic vectors where $U^v$ runs through all open compact subgroups of $\prod_{w\neq v,w|p}\GL_2(O_{F_w})$.
	\end{defn}
	
	The following result is helpful for our local-global compatibility argument.
	
	\begin{prop}\label{density}
		The subspace $S_\psi(U^p)_K^{v-\mathrm{a},v'-\mathrm{la}}$ is dense in $S_\psi(U^p)_K$.
	\end{prop}
	
	\begin{proof}
		See \cite[Proposition 3.2.6]{pan2022fontaine}.
	\end{proof}
	
	\subsection{Hecke algebras} In this subsection, we recall the definition of the $p$-adic big Hecke algebra. Our reference is \cite[Section 3.3]{pan2022fontaine}.
	
	Let $A$ be a Noetherian topological $\Z_p$-algebra, and consider a level $U$, a character $\psi$ and a representation $\tau: U_p \to \Aut(W_\tau)$. Let $\Sigma$ be a finite set of primes of $F$ containing all places above $p$ and places $v$ where either $D$ is ramified or $U_v$ is not a maximal open subgroup. For any $v \notin \Sigma$, we define the Hecke operator $T_v \in \operatorname{End}(S_{\tau, \psi}(U,A))$ to be the double coset action $[U_v \begin{pmatrix}
		\varpi_v & ~\\
		~ & 1
	\end{pmatrix}U_v]$. Precisely, if we write $U_v \begin{pmatrix}
		\varpi_v & ~\\
		~ & 1
	\end{pmatrix}U_v  = \coprod_i \gamma_i U_v $, then $$
	(T_v \cdot f) (g) = \sum_i f(g\gamma_i), ~~f \in S_{\tau, \psi}(U,A). $$
	
	We define the Hecke algebra $\mathbb{T}_{\tau, \psi} ^\Sigma (U,A) \subseteq \operatorname{End}(S_{\tau, \psi}(U,A))$ to be the $A$-subalgebra generated  by all $T_v, v \notin \Sigma$. This is a finite commutative $A$-algebra. We will write $\mathbb{T}_{\psi}^\Sigma (U,A)$ for the Hecke algebra if $\tau $ is the trivial action on $A$. Suppose that $A= \mathcal{O}/ \varpi^n$, $ U$ is sufficiently small and $\psi \mid_{U \cap (\mathbb{A}_F^\infty)^\times} $ is trivial modulo $ \varpi^n$. Then we can check that the Hecke algebra $\mathbb{T}_{\psi}^\Sigma (U, \mathcal{O}/ \varpi^n)$ is independent of $\Sigma$, and we may drop $\Sigma$ in this situation.
	
	\begin{defn} (Big Hecke algebra) \label{big}
		Let $U^p$ be a tame level and $\psi: (\mathbb{A}_F^\infty)^\times/ (U^p \cap (\mathbb{A}_F^\infty)^\times)F_+^\times \to \mathcal{O}^\times$ be a continuous character. Then we define the \textit{big Hecke algebra} $$
		\mathbb{T}_{\psi}(U^p) = \varprojlim_{(n, U_p) \in \mathcal{I}} \mathbb{T}_{\psi}(U^pU_p, \mathcal{O}/\varpi^n), $$
		where $\mathcal{I}$ is the set of pairs $(n, U_p)$ with $U_p \subseteq K_p$ and $n$ a positive integer such that $ \psi \mid_{U_p \cap \mathcal{O}_{F,p}^\times} \equiv 1~ \operatorname{mod}~\varpi^n$.
		
	\end{defn}
	
	We conclude some basic properties of the big Hecke algebra.
	
	\begin{prop} 	Let $U^p$ be a tame level and $\psi: (\mathbb{A}_F^\infty)^\times/ (U^p \cap (\mathbb{A}_F^\infty)^\times)F_+^\times \to \mathcal{O}^\times$ be a continuous character.
		
		1) $\mathbb{T}_{\psi}(U^p) $ is semi-local. (Write $\km_1,\cdots,\km_r$ for all the maximal ideals of $\T_\psi(U^p)$.)
		
		2) $\T_\psi(U^p)\cong \T_\psi(U^p)_{\km_1}\times\cdots\times \T_\psi(U^p)_{\km_r}$ and each $\T_\psi(U^p)_{\km_i}$ is $\km_i$-adically complete and separated. 
		
		3) $S_{\psi}(U^p)\cong S_{\psi}(U^p)_{\km_1}\oplus\cdots\oplus S_{\psi}(U^p)_{\km_r}$ and each $S_\psi(U^p)_{\km_i}$ is $\km_i$-adically complete and separated. 
		
		4) $\T_\psi(U^p)$ commutes with base change. In other words, if $K'$ is a finite extension of $K$ with ring of integers $\cO'$ and replace $\cO$ by $\cO'$ in all the definitions, then we have $\T'_\psi(U^p)\cong \T_\psi(U^p)\otimes_\cO \cO'$.
		
	\end{prop}
	
	\begin{proof}
		See \cite[Proposition 3.3.2, Corollary 3.3.3, 3.3.5]{pan2022fontaine}.
	\end{proof}
	
	Suppose $\T_\psi(U^p)$ is non-zero and $\km\in\Spec \T_\psi(U^p)$ is a maximal ideal. Enlarging $\cO$ if necessary, we may assume $\km$ has residue field $\F$. Denote by $T_\km$ the two-dimensional pseudo-representation over $\F$ (see \cite[3.3.3 \& 3.3.4]{pan2022fontaine} for the existence):
	\[G_{F,\Sigma} \longrightarrow \T^\psi(U^p)\to \T^\psi(U^p)/\km=\F.\]
	For any $v|p$, let $R^{\ps, \psi\varepsilon^{-1}}_v$ be the universal deformation ring which parametrizes all two-dimensional pseudo-representations of $G_{F_v}$ lifting $T_\km|_{G_{F_v}}$ with determinant $\psi\varepsilon^{-1}|_{G_{F_v}}$. By the universal property, we have a natural map $R^{\ps,\psi\varepsilon^{-1}}_v\to\T_\psi(U^p)_\km$. Taking tensor products over all $v|p$, we get
	\[R^{\ps,\psi\varepsilon^{-1}}_p:=\widehat{\bigotimes}_{v|p}R^{\ps,\psi\varepsilon^{-1}}_v\to\T_\psi(U^p)_\km.\]
	Therefore, we have defined an action of $R^{\ps,\psi\varepsilon^{-1}}_p$ on the completed cohomology $S_\psi(U^p)_\km$ via the map above.
	
	Suppose that $\xi: U^p \to A^\times$ is a continuous smooth character such that $\psi \mid_{\prod_{v \nmid p}(O_{F_v}^\times) \cap U_v}= \xi \mid_{\prod_{v \nmid p}(O_{F_v}^\times) \cap U_v}$. Then we can also define the big Hecke algebra $\T_{\psi, \xi}(U^p)$ and obtain the properties in the same way.
	
	\subsection{Local-global compatibility}\label{section lgc} In this subsection, we generalize Pan's local-global compatibility result and give some applications. We will frequently use Pašk\=unas' results and keep the notations given in the appendix.
	
	Let $G=\prod_{i=1}^f \GL_2(\Q_p)$ and $Z(G)\simeq \prod_{i=1}^f\Q_p^\times$ be its centre. (In our case, we may take $f=[F: \Q]$.) Let $\zeta: Z(G) \to \cO^\times $ be a continuous character. We first recall Pašk\=unas theory for the $p$-adic analytic group $G$.
	
	Let $\Mod_{G, \zeta}^{\mathrm{l. adm}}(\cO)$ be the full subcategory of $ \Mod_{G, \zeta}^{\mathrm{sm}}(\cO)$ consisting of locally admissible representations. Here an object $V$ is \textit{locally admissible} if for every $v \in V$, the smallest $\cO[G]$-submodule of $V$ containing $v$ is admissible. Then by \cite[Lemma 3.4.1]{pan2022fontaine}, the categories $\Mod_{G, \zeta}^{\mathrm{l. adm}}(\cO)$ and 
	$ \Mod_{G, \zeta}^{\mathrm{l. fin}}(\cO)$ are actually the same. Note that the categories $ \Mod_{G, \zeta}^{\mathrm{sm}}(\cO)$ and $\mathrm{Mod}_{G, \zeta}^{\mathrm{pro\, aug}}(\cO)$ are anti-equivalent via the Pontryagin duality. We write $\kC_{G, \zeta}(\cO)$ as the full subcategory of $\mathrm{Mod}_{G, \zeta}^{\mathrm{pro\, aug}}(\cO)$ 
	anti-equivalent to $\operatorname{Mod}^{\mathrm{l\, fin}}_{G, \zeta}(\mathcal{O})$.
	
	Let $\mathrm{Irr}_{G,\zeta}$ be the set of irreducible representations in $\mathrm{Mod}_{G,\zeta}^{\mathrm{sm}}(\cO)$. Recall that a block is an equivalence class of the relation $\sim$, where $\pi \sim \pi'$ if there exists $\pi_1, \cdots, \pi_n \in \Irr_{G, \zeta}$ such that $\pi \cong \pi_1$, $\pi' \cong \pi_n$, and for $1 \leq i \leq n-1$, $\pi_i \cong \pi_{i+1}$ or $\Ext^1_{G}(\pi_i, \pi_{i+1}) \neq 0$ or $\Ext^1_{G}(\pi_{i+1}, \pi_i) \neq 0$. Then there exists a natural decomposition of $\mathrm{Mod}^{\mathrm{l\, adm}}_{G,\zeta}(\cO)$ with respect to the blocks:
	\begin{eqnarray*}
		\mathrm{Mod}^{\mathrm{l\, adm}}_{G,\zeta}(\cO)\cong \prod_{\mathfrak{B}\in\mathrm{Irr}_{G,\zeta}/\sim}\mathrm{Mod}^{\mathrm{l\, adm}}_{G,\zeta}(\cO)^\kB,
	\end{eqnarray*}
	where $\mathrm{Mod}^{\mathrm{l\, adm}}_{G,\zeta}(\cO)^\kB$ is the full subcategory of $\mathrm{Mod}^{\mathrm{l\, adm}}_{G,\zeta}(\cO)$ consisting of representations with all irreducible subquotients in $\kB$. Taking Pontryagin dual, this gives:
	\[\kC_{G,\zeta}(\cO)\cong\prod_{\mathfrak{B}\in\mathrm{Irr}_{G,\zeta}/\sim}\kC_{G,\zeta}(\cO)^\kB.\]
	
	For a block $\kB$, write $\pi_{\kB}=\bigoplus_{\pi\in\kB_i}\pi$, where $\kB_i$ is the set of isomorphism classes of elements of $\kB$. Suppose that it is a finite set. Let $\pi_\kB\hookrightarrow J_{\kB}$ be an injective envelope of $\pi_{\kB}$ in $\mathrm{Mod}^{\mathrm{l\, adm}}_{G,\zeta}(\cO)$. Its Pontryagin dual $P_{\kB}:=J_{\kB}^\vee$ is a projective envelope of $\pi^\vee\cong\bigoplus_{\pi\in\kB_i}\pi^\vee$ in $\kC_{G,\zeta}(\cO)$. Let
	$E_{\kB}:=\End_{\kC_{G,\zeta}(\cO)}(P_{\kB})\cong \End_G(J_{\kB}).$
	This is a pseudo-compact ring. Recall that the functor 
	$M\mapsto \Hom_{\kC_{G,\zeta}(\cO)}(P_\kB,M)$
	defines an equivalence of categories between $\kC_{G,\zeta}(\cO)^\kB$ and the category of right pseudo-compact $E_{\kB}$-modules. An inverse functor is given by 
	$\mathrm{m}\mapsto (\mm\hat{\otimes}_{E_{\kB}}P_{\kB}).$ 
	
	As the block of $\GL_2(\Q_p)$ has been well-understood (see Appendix \ref{A.2}), we can use the following results to study the blocks of $G=\prod_{i=1}^f \GL_2(\Q_p)$.
	
	\begin{lem}\label{block for G}
		1) For any absolutely irreducible representation $\pi$ in $\mathrm{Mod}_{G,\zeta}^{\mathrm{sm}}(\cO)$, there is a finite extension $\F'/\F$ such that $\pi\otimes_{\F}\F'$ is isomorphic to some $\bigotimes_{i=1}^f\pi_i$, where $\pi_i$ are absolutely irreducible $\GL_2(\Q_p)$ representations over $\F'$.
		
		 2) Keep $\pi_i$ in 1) and assume that $\F'=\F$. Let $P_r\to\pi_r^\vee$ be a projective envelope of $\pi_r^\vee$ for $r=1,\cdots,f$. Then $P_1\hat{\otimes}\cdots\hat{\otimes}P_f\to\widehat{\bigotimes}_{i=1}^f\pi_i^\vee$ is a projective envelope of $\widehat{\bigotimes}_{i=1}^f\pi_i^\vee\cong(\bigotimes_{i=1}^f\pi_i)^\vee$ in $\kC_{G,\zeta}(\cO)$. 
		
		3) Write $G=G_1\times G_2$, where each $G_r=\prod_{i=1}^{f_r}\GL_2(\Q_p)$ with centre $Z_r,~r=1,2$. Suppose $M_r,N_r\in\kC_{G_r,\zeta|_{Z_r}}(\cO),~r=1,2$. Assume further that each $M_r$ is a finite direct sum of
		projective envelopes of irreducible objects in
		$\mathcal C_{G_r,\zeta|_{Z_r}}(\mathcal O)$. Then there exists a natural isomorphism:
		\[\Hom_{\kC_{G,\zeta}(\cO)}(M_1\hat{\otimes}M_2,N_1\hat{\otimes} N_2)\cong\Hom_{\kC_{G_1,\zeta|_{Z_1}}(\cO)}(M_1,N_1)\hat{\otimes}\Hom_{\kC_{G_2,\zeta|_{Z_2}}(\cO)}(M_2,N_2).\]
		
		4) Given $f$ blocks $\kB_1,\cdots,\kB_f$ of $\GL_2(\Q_p)$ that contain an absolutely irreducible representation, then
		\[\kB_1\otimes\cdots\otimes\kB_f:=\{\pi_1\otimes\cdots\otimes\pi_f|\pi_r\in\kB_r,r=1,\cdots,f\}\]
		is a block of $G$. 
	\end{lem}
	
	\begin{proof}
		See \cite[Lemma 3.4.5, 3.4.6 \& 3.4.7]{pan2022fontaine}.
	\end{proof}
	
	Using the previous lemma and the explicit description of the blocks of $\GL_2(\Q_p)$ (see Appendix \ref{A.2}), we know that for each block $\kB$ of $G$, it consists of a finite number of isomorphism classes of elements of $\kB$.
	
	For the rest of this subsection, we suppose that $ p$ is an odd prime. 
	
	For each block $\kB_r$ of $\GL_2(\Q_p)$, we write $\bar{\rho}_{\kB_r}$ for the associated semi-simple $2$-dimensional representation of $G_{\Q_p}$ over $\F$, and $R^{\ps,\zeta\varepsilon}_{\kB_r}$ for the universal deformation ring which parametrizes all $2$-dimensional pseudo-representations of $G_{\Q_p}$ lifting $\tr\bar{\rho}_{\kB_r}$ with determinant $\zeta\varepsilon$. Using Theorem \ref{uni=cen}, we know that $R^{\ps,\zeta\varepsilon}_{\kB_r}$ is exactly the centre of $E_{\kB_r}$, and hence we obtain a natural injection $ f_r: R^{\ps,\zeta\varepsilon}_{\kB_r} \hookrightarrow E_{\kB_r}$.
	
	\begin{prop}\label{Paskunas theory for G}
		Let $\kB=\kB_1\otimes\cdots\otimes\kB_f$ be a block of $G$ such that each $\kB_i$ contains an absolutely irreducible representation. Then
		\begin{enumerate}
			\item $E_\kB\cong \widehat{\bigotimes}_{r=1}^f E_{\kB_r}$
			\item The natural inclusion $ f_r: R^{\ps,\zeta\varepsilon}_{\kB_r} \hookrightarrow E_{\kB_r}$ induces a natural finite map 
			\[\widehat{\bigotimes}_{r=1}^f R^{\ps,\zeta\varepsilon}_{\kB_r}\to E_{\kB},\] 
			which makes $E_{\kB}$ into a finitely generated module over $\widehat{\bigotimes}_{r=1}^f R^{\ps,\zeta\varepsilon}_{\kB_r}$.
			\item The centre of $E_\kB$ is Noetherian and $E_{\kB}$ is a finitely generated module over its centre.
		\end{enumerate}
	\end{prop}
	
	\begin{proof}
		The first assertion follows from 2) and 3) of Lemma \ref{block for G}. The second one follows from Theorem \ref{centre finite} and Theorem \ref{uni=cen}. The third one follows from the second one and the fact that the image of $\widehat{\bigotimes}_{r=1}^f R^{\ps,\zeta\varepsilon}_{\kB_r}$ is in the centre of $E_{\kB}$.
	\end{proof}
	
	Keep the level $U $, the set $\Sigma$ and the character $\psi$ as in the previous subsection. Let $\km$ be a maximal ideal of $\T_\psi(U^p)$. We get a two-dimensional pseudo-representation with determinant $\psi\varepsilon^{-1}$:
	\[T_\km:G_{F,\Sigma}\to \T_\psi(U^p)_\km.\]
	After restricting it to $G_{F_v},v|p$, we can attach a two-dimensional semi-simple representation $\bar{\rho}_{\km,v}$ of $G_{F_v}$ over $\F$. Enlarging $\cO$ if necessary, we assume that $\bar{\rho}_{\km,v}$ is a direct sum of absolutely irreducible representations. Let $\bar{\rho}'_{\km,v}=\bar{\rho}_{\km,v}\otimes \varepsilon$ (as we use the normalized Colmez functor $\mathbf{V} $ here). Based on the list given in Appendix \ref{A.2}, we can define a block $\kB_{\km,v}$ of $\GL_2(F_v)\cong\GL_2(\Q_p)$ from $\bar{\rho}'_{\km,v}$ (since $p$ splits completely in $F$). Let $\kB_\km=\otimes_{v|p}\kB_{\km,v}$ be the block of $D_p^\times=\prod_{v|p}\GL_2(F_v)$ (using Lemma \ref{block for G}). Note that it has central character $\psi$. We denote its projective generator by $P_{\kB_{\km}}$.
	
	Our main goal is to study the following object:

		$$\mm:=\Hom_{\kC_{D_p^\times,\psi}(\cO)}(P_{\kB_\km},M_\psi(U^p)_\km).$$
	This object will be the Hecke module for the patching arguments to prove some $R=\mathbb{T}$ theorems.
	
	There are two actions of $R^{\ps,\psi\varepsilon^{-1}}_p:=\widehat{\bigotimes}_{v|p}R^{\ps,\psi\varepsilon^{-1}}_v$ on (the Hecke module) $\mm$:
	\begin{enumerate}
		\item $\tau_{\Gal}$: induced by the action of $\T_\psi(U^p)_\km$ on $M_\psi(U^p)_\km$. See the discussions at the end of the previous subsection.
		\item $\tau_{\mathrm{Aut}}$: induced by the action of $\widehat{\bigotimes}_{v|p}R^{\ps,\psi\varepsilon}_v\cong\widehat{\bigotimes}_{v|p}R^{\ps,\psi\varepsilon^{-1}}_v$ on $P_{\kB_\km}$ via $E_{\kB_\km}$ due to the map in (2) of Proposition \ref{Paskunas theory for G}.  The natural isomorphisms between $R^{\ps,\psi\varepsilon}_v$ and $R^{\ps,\psi\varepsilon^{-1}}_v$ are obtained by twisting by the inverse cyclotomic character. We fix this isomorphism from now on.
	\end{enumerate}
	
	\begin{thm}[Local-global compatibility] \label{lgc}
		Keep notations and assumptions as above. 
		
		\begin{enumerate}
			\item $M_\psi(U^p)_\km\in\kC_{D_p^\times,\psi}(\cO)^{\kB_{\km}}$.
			\item Both actions $\tau_{\Gal},\tau_{\mathrm{Aut}}$ of $R^{\ps,\psi\varepsilon^{-1}}_p$ on $\mm$ are the same.
		\end{enumerate}
	\end{thm}
	
	\begin{proof}
		We adapt the proof of \cite[Theorem 3.5.3]{pan2022fontaine}. The density argument and the finite-level local-global compatibility used there remain unchanged. The block-theoretic inputs from \cite{Pa_k_nas_2013}, which do not cover the exceptional block, are replaced by the corresponding results of Pašk\=unas--Tung \cite{Pa_k_nas_2021}, recalled in Appendix~\ref{A.2}. We give the resulting argument, which establishes both assertions also in the exceptional case.
		
		Let $v$ be a place above $p$. It suffices to prove 
		\begin{enumerate} 
			\item $M_\psi(U^p)_\km\in\kC_{\GL_2(F_v),\psi|_{F_v^\times}}(\cO)^{\kB_{\km,v}}$,
			\item both actions $\tau_{\Gal},\tau_{\mathrm{Aut}}$ of $R^{\ps,\psi\varepsilon^{-1}}_v$ on $\Hom_{\kC_{\GL_2(F_v),\psi|_{F_v^\times}}(\cO)}(P_{\kB_{\km,v}},M_\psi(U^p)_\km)$ are the same.
		\end{enumerate}
		
		After twisting some character, we may assume $\psi$ is crystalline at $v$ of Hodge-Tate weight $w_\psi$.  We may also shrink $U^p$ so that $\psi|_{U^p\cap(\A_F^\infty)^\times}$ is trivial and $U^pK_p$ is sufficiently small. 
		
		Recall that for any $(\vec{k},\vec{w})\in \Z_{>1}^{\Hom(F,\overbar{\Q_p})}\times  \Z^{\Hom(F,\overbar{\Q_p})}$ such that $k_\sigma+2w_\sigma=w_\psi+2$ independent of $\sigma$ and $U_p=\GL_2(O_{F_v})U^v$ with $U^v$ an open subgroup of  $\prod_{w\neq v, w|p}\GL_2(O_{F_w})$, we have the following isomorphism (see Remark \ref{quaternionic and automorphic}):
		\begin{eqnarray*}
			S_{(\vec{k},\vec{w}),\psi}(U^pU_p,K)\cong \Hom_{K[U_p]}(W^{*}_{(\vec{k},\vec{w}),K},S_\psi(U^p)_K).
		\end{eqnarray*}
		Then we get a natural surjective map $t_\psi: \T_\psi(U^p)[\frac{1}{p}]\to\T_{(\vec{k},\vec{w}),\psi}(U^pU_p,K)$ sending $T_w$ to $T_w$ for $w\notin \Sigma$. Let $\kp$ be a prime ideal of $\T_{(\vec{k},\vec{w}),\psi}(U^pU_p,K)\otimes_K\overbar{\Q_p}$. It corresponds to an automorphic representation $\pi_\kp=\pi_\kp^\infty\otimes(\pi_\kp)_\infty$ on $\DAi$. Combining the discussions in Remark \ref{quaternionic and automorphic}, we get 
		\[\varinjlim_{U_v} (S_{(\vec{k},\vec{w}),\psi}(U^pU^vU_v,K)\otimes_{K,\iota_p} \bC)[\kp]\cong (\pi_\kp^\infty)^{U^pU^v}\cong (\pi_\kp)_v^{\oplus d(\kp)}, d(\kp)>0,\]
		where the direct limit ranges over all open compact subgroups $U_v$ of $\GL_2(O_{F_v})$, and $(\pi_\kp)_v$ is the local representation of $\pi_\kp$ at place $v$. Choose a model $\pi^{K(\kp)}_v$ of $(\pi_\kp)_v$ over a finite extension $K(\kp) $ of $K$. Therefore, we get a map 
		\[\Phi_{\kp}:W^{*}_{(\vec{k},\vec{w}),K}\otimes_K (\pi^{K(\kp)}_v)^{\oplus d(\kp)}\to (S_\psi(U^p)\otimes_{\cO} K(\kp))[\kp],\]
		where we also view $\kp$ as a maximal ideal of $\T_\psi(U^p)\otimes K(\kp)$ via the surjection $t_\psi$. Then the image of $\Phi_\kp$ in $S_{\psi}(U^p)\otimes_\cO K(\kp)$ contains the $\GL_2(O_{F_v})U^v$-algebraic vectors of $(S_\psi(U^p)_K\otimes_K K(\kp))[\kp]$ (see Definition \ref{locally algebraic vectors}), and we denote its closure in $S_{\psi}(U^p)\otimes_\cO K(\kp)$ by $\Pi(\kp)$.
		
		Let $\Pi_{\kB_{\km,v}}:=\Hom^{\cont}_\cO(P_{\kB_{\km,v}},K)$. This is a Banach space with unit ball $\Hom^{\cont}_\cO(P_{\kB_{\km,v}},\cO)$. The following lemma is an application of Proposition \ref{density}, and one can see \cite[Lemma 3.5.4]{pan2022fontaine} for the proof.
		
		\begin{lem} \label{3.5.4}
			The inclusion map $\Pi(\kp)\hookrightarrow S_{\psi}(U^p)\otimes_{\cO} K(\kp)$ induces a natural injective map:
			\[\Hom^{\cont}_{K[\GL_2(F_v)]}(S_{\psi}(U^p)_\km\otimes_\cO K,\Pi_{\kB_{\km,v}})\hookrightarrow\prod_{U^v}\prod_{(\vec{k},\vec{w})}\prod_\kp \Hom^{\cont}_{K(\kp)[\GL_2(F_v)]}(\Pi(\kp),\Pi_{\kB_{\km,v}}\otimes K(\kp)),\]
			where $U^v$ runs over all open subgroups of $\prod_{w\neq v,w|p} \GL_2(O_{F_w})$, the pair $(\vec{k},\vec{w})$ runs over all elements  in $\Z_{>1}^{\Hom(F,\overbar{\Q_p})}\times  \Z^{\Hom(F,\overbar{\Q_p})}$ such that $k_\sigma+2w_\sigma=w_\psi+2$  for any $\sigma$, and $\kp$ runs over all the maximal ideals of $\T_{(\vec{k},\vec{w}),\psi}(U^p\GL_2(O_{F_v})U^v)_\km\otimes \overbar{\Q_p}$.
		\end{lem}
		
		Note that we have a natural isomorphism 
		$$\Hom^{\cont}_{K[\GL_2(F_v)]}(S_\psi(U^p)_\km\otimes K,\Pi_{\kB_{\km,v}})\cong \Hom_{\kC_{\GL_2(F_v),\psi|_{F_v^\times}}(\cO)} (P_{\kB_{\km,v}},M_\psi(U^p)_\km) \otimes K.$$
		To prove the theorem, it remains to check the relation $\tau_{\Gal}|_{R^{\ps,\psi\varepsilon^{-1}}_v}=\tau_{\mathrm{Aut}}|_{R^{\ps,\psi\varepsilon^{-1}}_v}$ on the module $\Hom^{\cont}_{K(\kp)[\GL_2(F_v)]}(\Pi(\kp),\Pi_{\kB_{\km,v}}\otimes K(\kp)) $ using the previous lemma and the $p$-torsion freeness of $S_{\psi}(U^p)$ (see the discussion under Definition \ref{Completed coh}). Now we fix such a $\kp$ and  suppose that it comes from $\T_{(\vec{k},\vec{w}),\psi}(U^p\GL_2(O_{F_v})U^v)_\km \otimes \overbar{\Q_p}$. For simplicity, we assume $K(\kp)=K$ after enlarging $K$, which is harmless by base change.
		
		We divide the proof into two cases.
		
		Case 1: $\pi_\kp$ does not factor through the reduced norm $N_{D/F}$.
		
		The action $\tau_{\Gal}$ of $R^{\ps,\psi\varepsilon^{-1}}_v$ on $\Hom^{\cont}_{K[\GL_2(F_v)]}(\Pi(\kp),\Pi_{\kB_{\km,v}}\otimes K)$ is clear by the classical local-global compatibility (Hecke correspondence) at primes above $p$. 
		
		\begin{lem} \label{clgc}
			Let $\kp_v=R^{\ps,\psi\varepsilon^{-1}}_v[\frac{1}{p}]\cap\kp$ and $\rho(\kp)_v:G_{F_v}\to\GL_2(K)$ be the semi-simple representation given by $\kp_v$. Then
			\begin{enumerate}
				\item $\rho(\kp)_v$ is de Rham of Hodge-Tate weights $(w_{\sigma_v},w_{\sigma_v}+k_{\sigma_v}-1)$, where $\sigma_v:F\to K$ is the embedding induced by $v$. More precisely, 
				\[\gr^i(\rho(\kp)_v\otimes B_{\dR})^{G_{F_v}}=0\] 
				unless $i=w_{\sigma_v},w_{\sigma_v}+k_{\sigma_v}-1$.
				\item The semi-simple Weil-Deligne representation $\mathrm{WD}(\varepsilon\otimes\rho(\kp)_v)^{\mathrm{ss}}$ corresponds to $\pi_v^{K(\kp)}$ under the Hecke correspondence.
			\end{enumerate}
		\end{lem}
		
		\begin{proof}
			It is \cite[Lemma 3.5.5]{pan2022fontaine}. Also see \cite[Theorem 1.1]{2011prims}.
		\end{proof}
		
		For the automorphic side, we need to study the object $\Pi(\kp)$. Recall that this is the closure of 
		\[W^{*}_{(\vec{k},\vec{w}),K}\otimes (\pi^{K(\kp)}_v)^{\oplus d(\kp)}=[(\Sym^{k_{\sigma_v}-2}(K^2)\otimes \det{}^{w_{\sigma_v}})^*\otimes \pi^{K(\kp)}_v]^{\oplus d(\kp)'}\]
		in $S_\psi(U^p)_K$. Here $d(\kp)'$ is some multiple of $d(\kp)$.
		
		Let $\Pi_v$ be the universal unitary completion of $(\Sym^{k_{\sigma_v}-2}(K^2)\otimes \det{}^{w_{\sigma_v}})^*\otimes \pi^{K(\kp)}_v$ as a $K$-representation of $\GL_2(F_v)$. Note that $\pi^{K(\kp)}_v$ is an irreducible principal series as the level is $\GL_2(O_{F_v})$ (hence unramified). Otherwise it is one-dimensional and $\pi_\kp$ has to factor through the reduced norm map by the approximation Theorem (for example, see \cite[Fact 4.17]{Gee_2022}). Then $\Pi_v$ is a topologically irreducible admissible unitary $\GL_2(F_v)$-representation. Here we use \cite[Theorem 1.1]{Pa_k_nas_2009} and \cite[Th\'eor\`eme]{2010Breuil} for the non-ordinary case and \cite[Proposition 2.2.1]{2010Emerton} for the ordinary case.
		
		The following lemma describes $\Pi(\kp)$ by $\Pi_v$. See \cite[Lemma 3.5.7]{pan2022fontaine} for the proof.
		
		\begin{lem}
			$\Pi(\kp)$ is a quotient of $\Pi_v^{\oplus d(\kp)'}$.
		\end{lem}
		
		Using this lemma, we get an injective map 
		\[\Hom^{\cont}_{K[\GL_2(F_v)]}(\Pi(\kp),\Pi_{\kB_{\km,v}})\hookrightarrow \Hom^{\cont}_{K[\GL_2(F_v)]}(\Pi_v^{\oplus d(\kp)'},\Pi_{\kB_{\km,v}}).\]
		Let $\Pi_v^0$ be a $\GL_2(F_v)$-invariant bounded open ball of $\Pi_v$ and let $(\Pi_v^0)^d=\Hom_\cO(\Pi_v^0,\cO)$ be its Schikhof dual. Then we also have 
		\[\Hom^{\cont}_{K[\GL_2(F_v)]}(\Pi_v^{\oplus d(\kp)'},\Pi_{\kB_{\km,v}})\cong K^{\oplus d(\kp)'}\otimes \Hom_{\kC_{\GL_2(F_v),\psi|_{F_v^\times}}(\cO)}(P_{\kB_{\km,v}},(\Pi_v^0)^d).\]
		Note that $\Hom(P_{\kB_{\km,v}},(\Pi_v^0)^d)\otimes K$ is a right $E_{\mathfrak{B}_{\mathfrak{m},v}}$-module. Then we just need to prove that the action of the centre $R^{\ps,\psi\varepsilon^{-1}}_v$ of $E_{\mathfrak{B}_{\mathfrak{m},v}}$ on $\Hom(P_{\kB_{\km,v}},(\Pi_v^0)^d)\otimes K$ also factors through $R^{\ps,\psi\varepsilon^{-1}}_v[\frac{1}{p}]/\kp_v$. We also view $\kp_v $ as a prime ideal of $R^{\ps,\psi\varepsilon}_v$ (see the isomorphism in the definition of $ \tau_{\mathrm{Aut}}$), and it suffices to show that $(\Pi_v^0)^d $ appears as a subquotient of $P_{\kB_{\km,v}}/\kp_v P_{\kB_{\km,v}}$.
		
		If $\rho(\kp)_v$ is absolutely irreducible, then using Proposition \ref{4.37} and Proposition \ref{irreducible_irr}, we know that there is, up to isomorphism, a unique irreducible Banach representation \(\Pi'_v\) occurring as a subquotient of $\Hom^{\cont}_\cO(P_{\kB_{\km,v}}/\kp_v P_{\kB_{\km,v}},K)$, which is characterized by $\mathbf{V}(\Pi'_v)\cong\rho(\kp)_v\otimes\varepsilon$. By \cite[Theorem 1.3]{Colmez_2014} and Lemma \ref{clgc}, $\Pi_v'$ is a unitary completion of 
		\[\Sym^{k_{\sigma_v}-2}(K^2)\otimes \det{}^{-(w_{\sigma_v}+k_{\sigma_v}-2)}\otimes \pi_v^{K(\kp)}.\]
		However, such unitary completion is unique up to isomorphism (see \cite[Th\'eor\`eme 4.3.1]{2010Breuil}). We have $\Pi_v' \cong \Pi_v$ and hence $(\Pi_v^0)^d \otimes K$ appears as a subquotient of $P_{\kB_{\km,v}}/\kp_v P_{\kB_{\km,v}} \otimes K$.
		
		If $\rho(\kp)_v=\psi_1\oplus\psi_2$ is reducible, we assume the Hodge-Tate weight of $\psi_1$ (resp. $\psi_2$) is $w_{\sigma_v}$ (resp. $w_{\sigma_v}+k_{\sigma_v}-1$). Then we know that
		\[\pi^{K(\kp)}_v\cong(\Ind^{\GL_2(F_v)}_{B(F_v)}\psi_1\varepsilon^{w_{\sigma_v}}|\cdot|^{1-w_{\sigma_v}}\otimes\psi_2\varepsilon^{w_{\sigma_v}+k_{\sigma_v}-1}|\cdot|^{-w_{\sigma_v}-k_{\sigma_v}+1})_{\mathrm{sm}}\]
		is irreducible by our assumption, and we have $\psi_1/\psi_2\neq \varepsilon^{\pm1}$. Using \cite[Proposition 2.2.1]{2010Emerton}, $\Pi_v$ is the unitary parabolic induction $(\Ind^{\GL_2(F_v)}_{B(F_v)}\psi_2\varepsilon\otimes \psi_1)_{\cont}$. Combining Proposition \ref{4.37} and Proposition \ref{reducible_irr}, we also conclude that $(\Pi_v^0)^d$ appears in the subquotient of $P_{\kB_{\km,v}}/\kp_v P_{\kB_{\km,v}}$.
		
		Case 2: $\pi_\kp$ factors through the reduced norm $N_{D/F}$.
		
		In this case, $\Pi(\kp)$ has the form $\eta\circ\det$ for some continuous character $\eta:F_v^\times\to\cO^\times$ and the corresponding pseudo-character of $G_{F_v}$ is $\eta+\eta\varepsilon^{-1}$. Our claims follow from Proposition \ref{4.37} and Proposition \ref{reducible_irr}.
		
		We now finish the proof of the second assertion of this theorem.
		
		For the first part, we recall the injection given in Lemma \ref{3.5.4}. As we have shown that $\Pi(\kp)$ belongs to the block $\kB_{\km,v}$, for any block $ \kB'$ of $\GL_2(F_v)$, we have $\Hom^{\cont}_{K[\GL_2(F_v)]}(S_{\psi}(U^p)_\km\otimes K,\Pi_{\kB'})=0$ unless $\kB'=\kB_{\km,v}$. This proves the first claim.
	\end{proof}
	
	There are some applications of the local-global compatibility result. The first  one is the finiteness of the Hecke module $\mm$, which contributes to our patching argument.
	
	\begin{cor}\label{Hecke module}
		Keep notations as above.
		
		\begin{enumerate}
			\item $\mm$ is a faithful, finitely generated $\T_\psi(U^p)_\km$-module. 
			\item $\T_\psi(U^p)_\km$ is a finite $R^{\ps,\psi\varepsilon^{-1}}_p$-algebra.
			\item $S_\psi(U^p,\F)[\km]$ is a representation of $D_p^\times$ of finite length. 
		\end{enumerate}
	\end{cor}
	
	\begin{proof}
	    Since \(P_{\mathfrak B_{\mathfrak m}}\) is a projective generator of the block,
		the usual evaluation morphism
		\( \mm\widehat\otimes_{E_{\mathfrak B_{\mathfrak m}}}P_{\mathfrak B_{\mathfrak m}}
		\to M_\psi(U^p)_{\mathfrak m}
		\)
		is an isomorphism. Hence by Theorem \ref{lgc}, any element of \(\T_\psi(U^p)_{\mathfrak m}\) which annihilates
		\(\mm\) also annihilates \(M_\psi(U^p)_{\mathfrak m}\). Since
		\(\T_\psi(U^p)_{\mathfrak m}\) acts faithfully on \(M_\psi(U^p)_{\mathfrak m}\), it follows that \(\mm\) is a faithful \(\T_\psi(U^p)_{\mathfrak m}\)-module.
		
		By Proposition \ref{Paskunas theory for G} and Proposition \ref{4.17}, \(\mm\) is finitely generated over
		\(R^{\mathrm{ps},\psi\varepsilon^{-1}}_p\). Then the first and second claims follow from Theorem \ref{lgc} and the faithfulness of $\mm$.
		
		Finally, reducing the above description modulo \(\mathfrak m\), and using Proposition \ref{4.17} again, we see that
		\(S_\psi(U^p,\mathbb F)[\mathfrak m]\) has finite length as a representation of
		\(D_p^\times\).
	\end{proof}
	
	The following application describes modular primes of the big Hecke algebra.

		\begin{cor} \label{classicality} Keep notations as above.
			For a maximal ideal $\kp$ of $\T_{\psi}(U^p)_\km[\frac{1}{p}]$, we assume that for any $v|p$,
			\begin{itemize}
				\item $\rho(\kp)|_{G_{F_v}}$ is absolutely irreducible and de Rham with distinct Hodge-Tate weights,
			\end{itemize}
			where $\rho(\kp):G_{F,\Sigma}\to\GL_2(k(\kp))$ is the semi-simple representation associated to $\kp$. Then $\kp$ is a pull-back of a maximal ideal of $\T_{(\vec{k},\vec{w}),\psi}(U^pU_p)[\frac{1}{p}]$ for some weight $(\vec{k},\vec{w})\in \Z_{>1}^{\Hom(F,\overbar{\Q_p})}\times  \Z^{\Hom(F,\overbar{\Q_p})}$ and open compact subgroup $U_p\subseteq K_p$. By Jacquet-Langlands correspondence, it also arises from a regular algebraic cuspidal automorphic representation of $\GL_2(\A_F)$.
		\end{cor}

	\begin{proof}
		We adapt the proof of \cite[Corollary 3.5.10]{pan2022fontaine}, using \cite[Corollary 6.16]{Pa_k_nas_2021} for the Banach representations at absolutely irreducible Galois points, including in the exceptional block. Corollary~\ref{Hecke module}, deduced from Theorem~\ref{lgc}, provides the faithful finitely generated Hecke module required in the specialization argument.
		
		After a finite extension of the coefficient field and passage to a point above $\kp$, we may assume $k(\kp)=K$. Let $\Pi_v$ be the absolutely irreducible unitary Banach representation characterized by  
		\(
		\mathbf V(\Pi_v)\cong \rho(\mathfrak p)|_{G_{F_v}}\otimes\varepsilon,
		\) as in Proposition \ref{irreducible_irr}. Here we use the identification of the two local pseudo-deformation rings fixed in the definition of \(\tau_{\mathrm{Aut}}\). Set $E_v:=E_{\kB_{\km,v}}\otimes_{R^{\ps,\psi\varepsilon^{-1}}_v}K$ and $E_{\kp}:=E_{\kB_\km}\otimes_{R^{\ps,\psi\varepsilon^{-1}}_p}K$, where the maps to $K$ are induced by $\kp$. We use the following version of \cite[Lemma 3.5.11]{pan2022fontaine}.
		
		\begin{lem} Keep notations as above.
			\begin{enumerate}
				\item The unitary representation $\Pi:=\widehat{\bigotimes}_{v|p} \Pi_v$ of $D_p^\times$ is topologically irreducible.
				\item The algebra $E_{\kp}$ is central simple over $K$.
			\end{enumerate}
		\end{lem}
		
		\begin{proof}[Proof of lemma]
			By \cite[Corollary 6.16]{Pa_k_nas_2021}, every admissible Banach representation in the local block on which $R^{\ps,\psi\varepsilon^{-1}}_v$ acts through the given map to $K$ is a finite direct sum of copies of $\Pi_v$. The equivalence in Theorem \ref{4.34} therefore shows that $E_v$ is a split central simple $K$-algebra and $\mm(\Pi_v)$ is its unique simple module. By Lemma \ref{block for G}, $E_{\kp}\cong\bigotimes_{v|p}E_v$ and $\mm(\Pi)\cong\bigotimes_{v|p}\mm(\Pi_v)$, where the tensor products are over $K$. Thus $E_{\kp}$ is central simple and $\mm(\Pi)$ is a simple $E_{\kp}$-module. By Proposition \ref{6.7}, Proposition \ref{Paskunas theory for G} and Proposition \ref{finite dim}, Theorem \ref{4.34} also applies to $G=D_p^\times$, and gives the topological irreducibility of $\Pi$.
		\end{proof}
		
		Now we return to the proof of the corollary.
		
		Let $\tilde{\kp} \in \Spec \T_\psi(U^p)_\km$  be the pull-back of $\kp$. Note that we have isomorphisms
		\[\Hom^{\cont}_\cO(M_\psi(U^p)_\km/\tilde\kp M_\psi(U^p)_\km ,K)\cong S_\psi(U^p)_\km[\tilde\kp]\otimes K\cong (S_\psi(U^p)_\km\otimes K)[\kp].\]
		Using Corollary \ref{Hecke module} and the previous lemma, we obtain an isomorphism $(S_\psi(U^p)_\km\otimes K)[\kp]\cong(\widehat{\bigotimes}_{v|p} \Pi_v)^d$ for some positive integer $d$. By \cite[Theorem 1.3]{Colmez_2014}, for each $v|p$, $\Pi_v$ has non-zero locally algebraic vectors of $\GL_2(F_v)$. Thus $(S_\psi(U^p)_\km\otimes K)[\kp]$ contains non-zero locally algebraic vectors of $D_p^\times$. Combining the discussion in Remark \ref{quaternionic and automorphic}, we prove the result.
	\end{proof}
	
	Another application is the lower bound of the dimension of the big Hecke algebra.
	
	\begin{cor}\label{dim of hecke}
		Keep notations as above. Each irreducible component of $\T_\psi(U^p)_\km$ is of characteristic zero and of dimension at least $1+2[F:\Q]$. 
	\end{cor}
	
	\begin{proof}
		The proof of \cite[Theorem 3.6.1]{pan2022fontaine} applies with Corollary~\ref{Hecke module} in place of \cite[Corollary 3.5.8]{pan2022fontaine}; this is the only point at which local-global compatibility enters the argument. The subsequent Gelfand-Kirillov dimension calculation is unchanged and gives the asserted characteristic-zero and dimension statements.
	\end{proof}
	
	Keep the character $\xi$ as in the previous subsection. The same statements hold for $M_{\psi, \xi}(U^p) $ and $\T_{\psi,\xi}(U^p)$ and allow us to carry out Taylor's Ihara-avoidance argument as in \cite{taylor2008automorphy}.
	
	\section{Pro-modularity}\label{sec 4}
	
	This section supplies the automorphic inputs. Subsection \ref{sec 4.1} proves the patching theorem at nice primes; Subsection \ref{ord section} uses ordinary modularity and finiteness to find such primes on sufficiently large ordinary components. These results are used in Section \ref{section 5}. Subsection \ref{irr sec} gives the residually irreducible potential big $R=\T$ theorem used in Section \ref{section 6}.
	
	For notations, we suppose $p$ is an odd prime. We denote by $K$  a $p$-adic local field with its ring of integers $\cO$, a uniformizer $\varpi$ and its residue field $\F$.
	
	\subsection{The residually reducible case}\label{sec 4.1}
	
	Let $F$ be a totally real number field of even degree over $\mathbb{Q}$ in which $p$ splits completely. Write $\Sigma_p$ as the set of places of $F$ above $p$. Let $\Sigma$ be a finite set of finite places of $F$ containing $\Sigma_p$ such that for all $v \in \Sigma \setminus \Sigma_p$, we have $p \mid \operatorname{Nm}(v)-1$. Let $\xi_v: k(v)^\times \to \mathcal{O}^\times$ be characters of $p$-power order for $v \in \Sigma \setminus \Sigma_p$, and we can view them as characters of $I_v$ by local class field theory. Let $\chi: G_{F,\Sigma} \to \mathcal{O}^{\times}$ be a continuous totally odd character such that 
	
	1) $\chi$ is unramified outside $ \Sigma_p$,
	
	2) $\chi(\operatorname{Frob}_v) \equiv 1~\operatorname{mod}~\varpi$ for  
	$v \in \Sigma \setminus \Sigma_p$.
	
	Let $D$ be a quaternion algebra over $F$ ramified exactly at all infinite places, and we fix an isomorphism between $( D \otimes_F \mathbb{A}_F^\infty)^\times$ and $\operatorname{GL}_2(\mathbb{A}_F^\infty)$. Let $\psi=\chi \varepsilon$ be a character of $ (\mathbb{A}_F^\infty)^\times/ F_+^\times$ via global class field theory. Let $U^p=\prod_{v \nmid p} U_v$ be a tame level such that $U_v= \operatorname{GL}_2(O_{F_v})$ if $v \notin \Sigma$ and $U_v = \operatorname{Iw}_v := \{g \equiv \begin{pmatrix}
		* & *\\
		0 & *
	\end{pmatrix}
	~\operatorname{mod}~\varpi_v, g \in \operatorname{GL}_2(O_{F_v})\}$ otherwise. For any $v \in \Sigma \setminus \Sigma_p$, the map $\begin{pmatrix}
		a & b\\
		c & d
	\end{pmatrix} \to \xi_v (\frac{a}{d}~\operatorname{mod}~\varpi_v)$ defines a character of $U_v$ and the product of $\xi_v$ can be viewed as a character $\xi$ of $U^p$ by projecting to $\prod_{v \in  \Sigma \setminus \Sigma_p} U_v$. From our setting, we can define a Hecke algebra $\mathbb{T}_{\xi} := \mathbb{T}_{\psi, \xi}(U^p)$ as in the previous section. We also make the following assumption:
	
	\begin{assumption}\label{assumption}
		There exists a maximal ideal $\mathfrak{m}_{\xi}$ of $\mathbb{T}_{\xi}$ generated by $T_v-(1+\chi(\Frob_v)),v\notin \Sigma$ and $\varpi$.
	\end{assumption}
	 
	 Consider the universal deformation ring $R^{\ps}$ which pro-represents the functor from $\mathfrak{U}_{\mathcal{O}}$ to the category of sets sending $R$ to the set of two-dimensional pseudo-representations $T$ of $G_{F,\Sigma}$ over $R$  such that $T$ is a lifting of $1+\bar\chi$ with determinant $\chi$. Let $R^{\ps,\{\xi_v\}}$ be the natural quotient of $R^{\ps}$ further satisfying 
	 $T|_{I_{F_v}}=\xi_v+\xi_v^{-1} $ for any $v\in \Sigma\setminus \Sigma_p$. 
	 
	 \begin{rem}\label{belong}
	 	Using \cite[Proposition 3.1.5]{Zhang2024}, for any minimal prime $\kq$ of $R^{\ps}$, the natural surjection $R^{\ps} \twoheadrightarrow R^{\ps}/ \kq$ factors through $R^{\ps,\{\xi_v\}}$ for some character $\xi$.
	 \end{rem}
	 
	 Recall that there is a natural pseudo-deformation $T_{\mathfrak{m}_{\xi}} : G_{F, \Sigma} \to (\mathbb{T}_{\xi})_{\mathfrak{m}_{\xi}}$ with determinant $\chi$ sending $\operatorname{Frob}_v$ to the Hecke operator $T_v$ for $v \notin \Sigma$.  By the local-global compatibility at $v \in \Sigma \setminus \Sigma_p$, we get a natural surjection $R^{\textnormal{ps}, \{\xi_v\}} \twoheadrightarrow (\mathbb{T}_{\xi})_{\mathfrak{m}_{\xi}} $.
	 
	 For any $\mathfrak{p} \in \operatorname{Spec} (\mathbb{T}_{\xi})_{\mathfrak{m}_{\xi}} $, we can define a two-dimensional semi-simple representation $\rho(\mathfrak{p}): G_{F, \Sigma} \to \operatorname{GL}_2(k(\mathfrak{p}))$ with trace $T_{\mathfrak{m}_{\xi}}~\operatorname{mod}~\mathfrak{p}$ in the sense of Definition \ref{construction}.
	
	\begin{defn}\label{nice}
		(1) We say that a prime  of $ R^{\textnormal{ps}, \{\xi_v\}}$ is \textit{pro-modular} if it comes from a prime of $(\mathbb{T}_{\xi})_{\mathfrak{m}_{\xi}} $. We say that a prime of $ R^\textnormal{ps}$ is \textit{pro-modular} if it comes from a pro-modular prime of some $ R^{\textnormal{ps}, \{\xi_v\}}$.
		
		(2) Let $\mathfrak{q} $ be a prime of $ (\mathbb{T}_{\xi})_{\mathfrak{m}_{\xi}} $ and $A$ be the normalization of $(\mathbb{T}_{\xi})_{\mathfrak{m}_{\xi}} / \mathfrak{q} $ in $k(\mathfrak{q})$. We say that $\mathfrak{q} $ is a \textit{nice} prime if $\mathfrak{q} $ contains $p$ and $\operatorname{dim} (\mathbb{T}_{\xi})_{\mathfrak{m}_{\xi}} / \mathfrak{q}=1$ and there exists a two-dimensional representation $\rho(\mathfrak{q})^o : G_{F, \Sigma} \to \operatorname{GL}_2(A)$ satisfying:
		
		\quad 1) $\rho(\mathfrak{q})^o \otimes k(\mathfrak{q}) \cong \rho(\mathfrak{q})$ is irreducible.
		
		\quad 2) The $\operatorname{mod}~\mathfrak{m}_A$ reduction $\bar{\rho}_b$ of $\rho(\mathfrak{q})^o$ is a non-split extension and has the form $ \bar{\rho}_b(g)=\begin{pmatrix}
			* & *\\
			0 & *
		\end{pmatrix}$, $g \in G_{F, \Sigma}$. Here $\mathfrak{m}_A $ is the maximal ideal of $A$.
		
		\quad 3) If $\rho(\mathfrak{q}) $ is dihedral, namely isomorphic to $\operatorname{Ind}_{G_L}^{G_F} \theta$ for some quadratic extension $L$ of $F$ and continuous character $\theta: G_L \to k(\mathfrak{q})^\times$, then $L \cap F(\mu_p) = F$,
		where $\mu_p $ is a primitive $p$-th root of unity.
		
		\quad 4) $\rho(\mathfrak{q})^o \mid_{G_{F_v}} = \bar{\rho}_b \mid_{G_{F_v}} $ for any $v \in \Sigma \setminus \Sigma_p$, i.e. $\rho(\kq)^o|_{G_{F_v}}$ is the trivial lifting.
		
		(3) We say that a prime of $ R^{\textnormal{ps}, \{\xi_v\}}$ is \textit{nice} if it comes from a nice prime of $(\mathbb{T}_{\xi})_{\mathfrak{m}_{\xi}} $.
		
	\end{defn}
	
	Now we can state our $R=\mathbb{T}$ theorem.
	
	\begin{thm} \label{R=T}
		Keep notations and assumptions as above. Let $\kq\in \Spec  (\mathbb{T}_{\xi})_{\mathfrak{m}_{\xi}} $ be a nice prime and $\kq^{\ps}=\kq\cap R^{\ps,\{\xi_v\}}$. Then the natural surjective map 
		\[(R^{\ps,\{\xi_v\}})_{\kq^{\ps}}\to(\mathbb{T}_{\xi})_{\mathfrak{q}}\]
		has nilpotent kernel.
	\end{thm}
	
	\begin{rem}
		Compared to \cite[Theorem 4.1.4]{pan2022fontaine}, our theorem removes the condition $\bar{\chi}|_{G_{F_v}}\neq \omega^{\pm 1}$ for any $v|p$ if $p=3$. The proof follows Pan's patching architecture, with two essential inputs replaced: Theorem~\ref{lgc} supplies local--global compatibility for the exceptional block, while Proposition~\ref{dim of loc framed der ring} supplies the required normality of the local framed deformation rings. We record the argument below and indicate precisely where these new inputs enter.
	\end{rem}
	
	\begin{thm}\label{cor of R=T}
		Keep notations and assumptions as in the previous theorem. Let $\kp$ be a maximal ideal of $R^{\ps,\{\xi_v\}}[\frac{1}{p}]$. Assume that 
		\begin{itemize}
			\item For any $v|p$, $\rho(\kp)|_{G_{F_v}}$ is irreducible and de Rham with distinct Hodge-Tate weights.
			\item There exists an irreducible component of $R^{\ps,\{\xi_v\}}$ containing both $\kp$ and a nice prime $\kq$.
		\end{itemize}
		Then $\rho(\kp)$ arises from a regular algebraic cuspidal automorphic representation of $\GL_2(\A_F)$.
	\end{thm}
	
	\begin{proof}
		This is clear from the previous theorem and Corollary \ref{classicality}.
	\end{proof}
	
     Using these two theorems above, to prove the Fontaine-Mazur conjecture in our case, we just need to find enough nice primes in $\Spec R^{\ps,\{\xi_v\}}$. 
     
     The following lemma explains the third condition of the definition of a nice prime.
     
     \begin{lem} \label{nirred}
     	Let $\kq$ be a prime ideal of $\Spec R^{\ps,\{\xi_v\}}$ containing $p$ such that $R^{\ps,\{\xi_v\}}/\kq$ is one-dimensional.  Suppose $\rho(\kq)$ is irreducible. Then the third condition in (2) of Definition \ref{nice} holds for $\kq$ if one of the following conditions holds:
     	\begin{enumerate}
     		\item $\bar{\chi}$ is not quadratic. 
     		\item $\bar{\chi}|_{G_{F_v}}=\mathbf{1}$ for some $v|p$.
     		\item There exists a place $v|p$ such that $\bar{\chi}|_{G_{F_v}}\neq\mathbf{1}$ and $\rho(\kq)^o|_{G_{F_v}}\cong \begin{pmatrix}\chi_{v,1}& *\\ 0 & \chi_{v,2}\end{pmatrix}$ is reducible. Moreover $\chi_{v,1}/\chi_{v,2}$ is of infinite order, which is equivalent to saying $\chi_{v,1}$ is of infinite order as $\chi_{v,1}\chi_{v,2}=\bar{\chi}$ is of finite order.
     	\end{enumerate}
     \end{lem}
	
	\begin{proof}
		This is \cite[Lemma 4.1.6]{pan2022fontaine}.
	\end{proof}
	
    For the rest of this subsection, we outline the main steps of the proof of Theorem \ref{R=T}. Here we use the notation of \cite[Section 4]{pan2022fontaine}, and one can find more details there.
    
    Enlarging $\cO$ if necessary, we may assume $A \cong \F[[T]]$. Let $\mathfrak{q}$ be a nice prime of $(\mathbb{T}_{\xi})_{\mathfrak{m}_{\xi}}$. Let $B$ be the topological closure of the $\mathbb{F}$-algebra generated by all the entries of $\rho(\mathfrak{q})^o(G_{F,\Sigma})$ as a subring of $A$. By Chebotarev’s density theorem, we may find a finite set of primes $T'$ disjoint from $\Sigma$ such
    that the entries of $\rho(\mathfrak{q})^o(\operatorname{Frob}_v), v \in T'$, topologically generate $B$. Let $P= T' \cup \Sigma$.
	
	\textbf{Step 1}. Describe universal deformation rings. 
	
	Write $$ R_{\textnormal{loc}}^{\{\xi_v\}}:=(\widehat{\bigotimes}_{v\in \Sigma_p}R^{\square}_v)\widehat{\otimes}_{\cO}(\widehat{\bigotimes}_{v\in \Sigma\setminus \Sigma_p}R^{\square,\xi_v}_v)\widehat{\otimes}_{\cO}(\widehat{\bigotimes}_{v\in T'}R^{\square,ur}_v)
	$$ for the completed tensor product of local framed deformation rings determined by primes in $P$, where the local deformation problem (in the sense of \cite[Definition 2.2.2]{Clozel_2008}) for $v \in \Sigma_p$ (resp. $v \in \Sigma\setminus \Sigma_p$, $v \in T'$) is the unrestricted (resp. level raising, unramified) one with determinant $\chi$. By the universal property, the Galois representation $\rho(\mathfrak{q})^o$ gives rise to a one-dimensional prime $\kq^{\{\xi_v\}}_{\loc}$ of $R^{\{\xi_v\}}_{\loc}$. 
	
	\begin{prop}{\cite[Proposition 4.2.3]{pan2022fontaine}}\label{4.2.3}
		The $\kq^{\{\xi_v\}}_{\loc}$-adic completion $\widehat{(R^{\{\xi_v\}}_{\loc})_{\kq^{\{\xi_v\}}_{\loc}}}$ of $(R^{\{\xi_v\}}_{\loc})_{\kq^{\{\xi_v\}}_{\loc}}$ is equidimensional of dimension $3[F:\Q]+3|P|$. The generic point of each irreducible component has characteristic zero. Moreover,
		\begin{enumerate}
			\item If all $\xi_v$ are non-trivial, then $\widehat{(R^{\{\xi_v\}}_{\loc})_{\kq^{\{\xi_v\}}_{\loc}}}$ is integral.
			\item In general, each minimal prime of $\widehat{(R^{\{\xi_v\}}_{\loc})_{\kq^{\{\xi_v\}}_{\loc}}}/(\varpi)$  contains a unique minimal prime of $\widehat{(R^{\{\xi_v\}}_{\loc})_{\kq^{\{\xi_v\}}_{\loc}}}$.
		\end{enumerate}
	\end{prop}
	
	\begin{rem}\label{4.1.9}
		1) Here we need to use the last condition of the definition of a nice prime.
		
		2) To prove this proposition, we first need to know some properties of the local framed deformation rings. More precisely, Pan shows that $R_v^\square$ is a normal domain of dimension $7$ for any $v|p$ in \cite[Lemma 4.2.7]{pan2022fontaine}. However, Pan uses \cite[Corollary B.5]{Pa_k_nas_2013} and \cite[Corollary 3.6 \& 3.7]{hutan15} to depict the versal deformation ring in the case $(\bar{\rho}_b|_{G_{F_v}})^{\operatorname{ss}} \cong \eta \oplus \eta\omega$, and both results require the assumption $p>3$. To cover the case $p=3$, we need to use Proposition \ref{dim of loc framed der ring} instead.
	\end{rem}
	
	\textbf{Step 2}. Find Taylor-Wiles primes.
	
	For a finite set $Q$ of finite places of $F$, we define $R^{\square_P,\{\xi_v\}}_{\bar{\rho}_b,Q}$ to be the universal object pro-representing the functor $\mathrm{Def}^{\square_P,\{\xi_v\}}_{\bar{\rho}_b,Q}$ from $\mathfrak{U}_{\mathcal{O}}$ to the category of sets sending $R$ to the set of tuples $(\rho_R;\alpha_v)_{v\in P}$ modulo the equivalence relation $\sim_P$ where
	\begin{itemize}
		\item $\rho_R:G_{F,\Sigma\cup Q}\to\GL_2(R)$ is a lifting of $\bar{\rho}_b$ to $R$ with determinant $\chi$ such that $\tr (\rho_R)|_{I_{F_v}}=\xi_v+\xi_v^{-1}$ for any $v\in \Sigma \setminus \Sigma_p$.
		\item $\alpha_v\in 1+M_2(\km_R),v\in P$. Here $\km_R$ is the maximal ideal of $R$.
		\item $(\rho_R;\alpha_v)_{v\in P}\sim_P (\rho'_R;\alpha'_v)_{v\in P}$ if there exists an element $\beta\in 1+M_2(\km_R)$ with $\rho'_R=\beta\rho_R\beta^{-1},\alpha'_v=\beta\alpha_v$ for any $v\in P$.
	\end{itemize}
	By the universal property, there is a natural map $ R^{\{\xi_v\}}_{\loc}\to R^{\square_P,\{\xi_v\}}_{\bar{\rho}_b,Q}$.
	
	\begin{prop} \label{exttwp}{\cite[Proposition 4.3.1]{pan2022fontaine}}
		Let $r=\dim_{k(\kq)}H^1(G_{F,P},\ad^0\rho(\kq)(1))$. Then there exists an integer $C$ such that for any positive integer $N$, we can find a finite set of primes $Q_N$ disjoint with $P$ such that
		\begin{enumerate}
			\item $|Q_N|=r$.
			\item $N(v)\equiv 1\mod p^N$ for any $v\in Q_N$.
			\item $\rho(\kq)(\Frob_v)$ has distinct eigenvalues $\alpha_v,\beta_v$ with $\ell(A/(\alpha_v-\beta_v)^2)<C$ for any $v\in Q_N$.
			\item There exists an $A$-module $M_N$ with $\ell(M_N)<C$ such that 
			\[\kq_{b,Q_N}/(\kq_{b,Q_N}^2,\kq^{\{\xi_v\}}_{\loc})\otimes_B A \cong A^{\oplus g} \oplus M_N\] 
			as $A$-modules, where $g=r+|P|-[F:\Q]-1$.
			\item There exists a map $R_{\infty}^{\{\xi_v\}}:=R^{\{\xi_v\}}_{\loc}[[x_1,\cdots,x_g]]\to R^{\square_P,\{\xi_v\}}_{\bar{\rho}_b,Q_N}$ such that the images of $x_i$ are in $\kq_{b,Q_N}$ and $\kq_{b,Q_N}/(\kq_{b,Q_N}^2,\kq^{\{\xi_v\}}_{\loc},x_1,\cdots,x_g)$ is killed by some element $f\in B$ independent of $N$ with $\ell(A/(f))<C$. 
		\end{enumerate}
	\end{prop}
	
	\begin{rem}
		1) The proof of the proposition needs to use the third condition of the definition of a nice prime.
		
		2) With Taylor-Wiles primes, we can define the completed homology with auxiliary levels. See \cite[Section 4.4]{pan2022fontaine} for precise definitions.
		
	\end{rem}
	
	\textbf{Step 3}. Modify Hecke modules.
	
	Define:
	
	\begin{itemize}
		\item $\Delta_\infty:=\Z_p^{\oplus r}$, where $r$ is defined in Proposition \ref{exttwp}.
		\item $\cO_\infty=\cO[[y_1,\cdots,y_{4|P|-1}]]$ with maximal ideal $\kb$ and prime ideal $\kb_1=(y_1,\cdots,y_{4|P|-1})$.
		\item $S_\infty=\cO_\infty[[\Delta_\infty]]\cong \cO_{\infty}[[s_1,\cdots,s_r]]$. This is a local $\cO_\infty$-algebra with maximal ideal $\ka$. Let $\ka_0=\ker(S_\infty\to\cO_\infty)=(s_1,\cdots,s_r)$ be the augmentation ideal, and $\ka_1=(\ka_0,\kb_1)=(y_1,\cdots,y_{4|P|-1},s_1,\cdots,s_r)$.
		\item $S_\infty'\subseteq\cO_\infty[[\Delta_\infty]]$ is the closure (under the profinite topology) of the $\cO_\infty$-subalgebra generated by all elements of the form $g+g^{-1}$ with $g=(0,\cdots,0,a,0,\cdots,0)\in\Delta_\infty$ for some $a\in\Z_p$. This is a regular local $\cO_\infty$-algebra and $S_\infty$ is a finite free $S_\infty'$-algebra. Write $\ka_0'=\ka_0\cap S_\infty',\ka_1'=\ka_1\cap S_\infty'$. We may find $r$ elements $s_1',\cdots,s_r'$ that generate $\ka_0'$ and $S_\infty'\cong\cO_\infty[[s_1',\cdots,s_r']]$.
		\item $\mm^{\{\xi_v\}}:=\Hom_{\kC_{D_p^\times,\psi}(\cO)}(P_{\kB_\km},M_{\psi, \xi}(U^p)_\km)$, and $\mm_0^{\{\xi_v\}}:=$ the $\kq$-adic completion of $(\mm^{\{\xi_v\}})_\kq$ as a $\T_{\psi,\xi}(U^p)_\kq$-module.
	\end{itemize} 
	
	In general, the (patched) completed homology is not a finitely generated module over the big Hecke algebra as there is an action of $\GL_2(\Q_p)$. Thus, it is necessary to modify it (using \cite[Proposition 4.4.4]{pan2022fontaine}). After taking the inverse limit, we denote the modified Hecke module by $\mm_{\infty}^{\{\xi_v\}}$ (see \cite[Definition 4.7.4]{pan2022fontaine} for the precise definition).
	
	\begin{prop} {\cite[Proposition 4.7.5]{pan2022fontaine}} 
		Let $\kq^{\{\xi_v\}}_{\infty}$ be the prime ideal of $R_{\infty}^{\{\xi_v\}}$ generated by $\kq^{\{\xi_v\}}_{\loc},x_1,\cdots,x_g$. 
		\begin{enumerate}
			\item $\mm^{\{\xi_v\}}_\infty$ is a finitely generated $\widehat{(R^{\{\xi_v\}}_\infty)_{\kq^{\{\xi_v\}}_{\infty}}}$-module.
			\item $\mm^{\{\xi_v\}}_\infty$ is a flat $S_\infty$-module. 
			\item We have $\mm^{\{\xi_v\}}_\infty/\ka_1\mm^{\{\xi_v\}}_\infty\cong (\mm_0^{\{\xi_v\}})^{\oplus 2^r}.$
		\end{enumerate}
	\end{prop}
	
	\begin{rem}\label{4.1.13}
		We can verify that Pan's proof applies after replacing its local-global compatibility inputs by Theorem~\ref{lgc} and Corollary~\ref{Hecke module}; all other ingredients are unchanged.
	\end{rem}
	
	\textbf{Step 4}. Use Taylor's trick.
	
	The strategy of the proof of Theorem \ref{R=T} follows from \cite[Theorem 4.1]{taylor2008automorphy}.
	
	First, using Corollary \ref{Hecke module}, we know that $\textnormal{m}_0^{\{\xi_v\}}$ is a finitely generated faithful $\widehat{(\mathbb{T}_{\xi})_{\mathfrak{q}}}$-module. Then by Corollary \ref{dim of hecke}, we have $\operatorname{dim}_{\widehat{(\mathbb{T}_{\xi})_{\mathfrak{q}}}} (\textnormal{m}_0^{\{\xi_v\}})= \operatorname{dim} \widehat{(\mathbb{T}_{\xi})_{\mathfrak{q}}} \ge 2[F:\mathbb{Q}]$ since $\mathfrak{q}$ is of dimension $1$.
	
	Let $(R^{\{\xi_v\}})'$ be $\widehat{(R_{\infty}^{\{\xi_v\}})_{\mathfrak{q}_{\infty}^{\{\xi_v\}}}} \otimes_{S_\infty'} S_\infty$.	We know that $\widehat{(\mathbb{T}_{\xi})_{\mathfrak{q}}} $ is a natural quotient of $(R^{\{\xi_v\}})'$ by mapping $ S_\infty$ to $S_\infty/ \mathfrak{a}_1=\mathcal{O}$. Note that $ y_1, ... , y_{4|P|-1}, s_1, ... , s_r \in \mathfrak{a}_1$ form a regular sequence of $\textnormal{m}_\infty^{\{\xi_v\}}$. Then we can conclude that $$\operatorname{dim}_{(R^{\{\xi_v\}})'}(\textnormal{m}_\infty^{\{\xi_v\}}) \ge 4|P|-1 +r +\operatorname{dim} \widehat{(\mathbb{T}_{\xi})_{\mathfrak{q}}} \ge 4|P|-1 +r+2[F:\mathbb{Q}].$$
	
	Using Proposition \ref{4.2.3}, we can conclude that $\widehat{(R_{\infty}^{\{\xi_v\}})_{\mathfrak{q}_{\infty}^{\{\xi_v\}}}}$ is equidimensional of dimension $ 4|P|-1 +r+2[F:\mathbb{Q}]$ (also see \cite[Section 4.2, Lemma 4.8.2]{pan2022fontaine}). As $(R^{\{\xi_v\}})'$ is finite free over $\widehat{(R_{\infty}^{\{\xi_v\}})_{\mathfrak{q}_{\infty}^{\{\xi_v\}}}}$, we have $$\operatorname{dim}_{(R^{\{\xi_v\}})'}(\textnormal{m}_\infty^{\{\xi_v\}})=\operatorname{dim}_{\widehat{(R_{\infty}^{\{\xi_v\}})_{\mathfrak{q}_{\infty}^{\{\xi_v\}}}}}(\textnormal{m}_\infty^{\{\xi_v\}})=4|P|-1 +r+2[F:\mathbb{Q}]= \operatorname{dim} \widehat{(R_{\infty}^{\{\xi_v\}})_{\mathfrak{q}_{\infty}^{\{\xi_v\}}}}.$$
	
	If the characters in $\{\xi_v\}$ are all non-trivial, using Proposition \ref{4.2.3} again, $\widehat{(R_{\infty}^{\{\xi_v\}})_{\mathfrak{q}_{\infty}^{\{\xi_v\}}}}$ is irreducible. From the previous equations, we know that $\textnormal{m}_\infty^{\{\xi_v\}} $ has full support on $\widehat{(R_{\infty}^{\{\xi_v\}})_{\mathfrak{q}_{\infty}^{\{\xi_v\}}}}$. By \cite[Lemma 4.8.4]{pan2022fontaine}, Theorem \ref{R=T} holds in this case.
	
	In general, we suppose that $\xi'_v:k(v)^\times\to \cO^\times$ are non-trivial characters of $p$-power order for $v\in \Sigma \setminus \Sigma_p$. We have natural isomorphisms $\widehat{(R^{\{\xi'_v\}}_\infty)_{\kq^{\{\xi'_v\}}_{\infty}}}/(\varpi)\cong \widehat{(R^{\{\xi_v\}}_\infty)_{\kq^{\{\xi_v\}}_{\infty}}}/(\varpi)$ and $\mm^{\{\xi'_v\}}_\infty/\varpi \mm^{\{\xi'_v\}}_\infty\cong \mm^{\{\xi_v\}}_\infty/\varpi \mm^{\{\xi_v\}}_\infty$. Then Theorem \ref{R=T} follows from Proposition \ref{4.2.3} and the previous case.
	
	\subsection{The ordinary case} \label{ord section}We recall the ordinary results of \cite[Section 5]{pan2022fontaine}. They control the dimension of the ordinary locus and provide nice primes on its components of dimension $[F:\Q]+1$, as needed for the special cases in Section \ref{section 5}.
	
	Let $F$ be an abelian totally real field in which $p$ is unramified. Let $\Sigma$ be a finite set of finite places containing all places above $p$. 
	
	\begin{defn} \label{psiord}
		Let $T:G_{F}\to R$ be a two-dimensional pseudo-representation over some ring $R$ such that for some place $v$, $T|_{G_{F_v}}=\psi_1+\psi_2$ is a sum of two characters. We say that $T$ is $\psi_1$-\textit{ordinary} if for any $\sigma,\tau\in G_{F_v},\eta\in G_F$,
		\[T(\sigma\tau\eta)-\psi_1(\sigma)T(\tau\eta)-\psi_2(\tau)T(\sigma\eta)+\psi_1(\sigma)\psi_2(\tau)T(\eta)=0.\] 
	\end{defn}
	
	\begin{lem} \label{ordpseudo}
		Let $\rho:G_F\to\GL_2(R)$ be a two-dimensional irreducible representation over some ring $R$ with trace $T$ such that $T|_{G_{F_v}}=\psi_1+\psi_2$, a sum of two characters. Suppose $\rho|_{G_{F_v}}$ has the form $\begin{pmatrix}\psi_1&*\\0& \psi_2\end{pmatrix}$. Then $T$ is $\psi_1$-ordinary. Conversely, if $R$ is a field and $T$ is $\psi_1$-ordinary, then after possibly enlarging $R$,
		\[\rho|_{G_{F_v}}\cong\begin{pmatrix}\psi_1&*\\0& \psi_2\end{pmatrix}.\]
	\end{lem}
	
	\begin{proof}
		See \cite[Lemma 5.3.2]{pan2022fontaine}.
	\end{proof}
	
	Let $\chi:G_{F,\Sigma}\to \cO^\times$ be a continuous totally odd character such that
	\begin{itemize}
		\item $\bar{\chi}$, the reduction of $\chi$ modulo $\varpi$, can be extended to a character of $G_\Q$.
		\item $\bar{\chi}|_{G_{F_v}}\neq \mathbf{1}$ for any $v|p$.
		\item $\chi|_{G_{F_v}}$ is de Rham for any $v|p$. 
	\end{itemize}
	
	Let $R^{\ps,\ord}$ be the universal pseudo-deformation ring which pro-represents the functor from $\mathfrak{U}_{\mathcal{O}}$ to the category of sets sending $R$ to the set of two-dimensional pseudo-representations $T$ of $G_{F,\Sigma}$ over $R$ such that $T$ is a lifting of $1+\bar\chi$ with determinant $\chi$ and $T|_{G_{F_v}}$ is reducible for any $v|p$. Thus $R^{\ps,\ord}$ is the quotient of $R^{\ps}$ by the local reducibility ideals at the places above $p$. Denote the universal pseudo-representation by $T^{univ}:G_{F,\Sigma}\to R^{\ps,\ord}$. 
	
	As $\bar{\chi}|_{G_{F_v}}$ is non-trivial, we may write $T^{univ}|_{G_{F_v}}=\psi^{univ}_{v,1}+\psi^{univ}_{v,2}$ for some characters $\psi^{univ}_{v,1},\psi^{univ}_{v,2}:G_{F_v}\to (R^{\ps,\ord})^\times$ which are liftings of $\mathbf{1},\bar{\chi}|_{G_{F_v}}$, respectively. By local class field theory, $\psi^{univ}_{v,1}|_{I_{F_v}}$ induces a homomorphism $\cO[[O_{F_v}^\times(p)]]\to R^{\ps,\ord}$ for any $v|p$. Here $O_{F_v}^\times(p)$ denotes the $p$-adic completion of $O_{F_v}^\times$. Taking the completed tensor product over $\cO$ for all $v|p$, we get a map from the Iwasawa algebra to the universal ordinary pseudo-deformation ring:
	\[\Lambda_F:=\widehat{\bigotimes}_{v|p}\cO[[O_{F_v}^\times(p)]]\to R^{\ps,\ord}.\]
	
	\begin{thm} \label{ordinary case}{\cite[Theorem 5.1.1]{pan2022fontaine}}
		Under the assumptions for $F,\chi$ as above, we have
		\begin{enumerate}
			\item $R^{\ps,\ord}$ is a finite $\Lambda_F$-algebra.
			\item For any maximal ideal $\kp$ of $R^{\ps,\ord}[\frac{1}{p}]$, we denote the associated semi-simple representation $G_{F,\Sigma}\to\GL_2(k(\kp))$ by $\rho(\kp)$ in the sense of Definition \ref{construction}. Assume that
			\begin{itemize}
				\item $\rho(\kp)$ is irreducible.
				\item For any $v|p$, $\rho(\kp)|_{G_{F_v}}\cong\begin{pmatrix}\psi_{v,1} & *\\ 0 & \psi_{v,2}\end{pmatrix}$ such that $\psi_{v,1}$ is de Rham and has strictly less Hodge-Tate number than $\psi_{v,2}$ for any embedding $F_v\hookrightarrow \overbar{\Q_p}$.
			\end{itemize}
			Then $\rho(\kp)$ comes from a twist of a Hilbert modular form.
		\end{enumerate}
	\end{thm}
	
	\begin{defn}
		Let $\kq$ be a prime ideal of $R^{\ps,\ord}$. If $ \kq$ is an inverse image of a prime $\kp$ of $R^{\ps,\ord}[\frac{1}{p}]$ satisfying the assumptions in (2) of the previous theorem via the natural map $R^{\ps,\ord} \to  R^{\ps,\ord}[\frac{1}{p}]$, then we call it a \textit{regular de Rham} prime.
	\end{defn}
	
	The following lemma is an essential application of Pan's finiteness result.
	
	\begin{lem}\label{find nice}
		Suppose that $p$ splits completely in $F$. Let $C$ be an irreducible component of $\Spec R^{\ps,\ord}$, and here we also view $C$ as a subset of $\Spec R^{\ps}$. Assume that $C$ is of dimension $[F:\mathbb{Q}]+1$. 
		
		1) Assume $F \ne \Q$. The subset $C^{\textnormal{reg}}$ consisting of regular de Rham primes of $C$ is dense in $C$. Consequently, $C$ is pro-modular, i.e., every prime in $C$ comes from some big Hecke algebra. In particular, $1+\bar{\chi}$ is modular.
		
		2) Suppose the following assumptions hold. 	\begin{itemize}
			\item $F$ is a totally real abelian extension of $\mathbb{Q}$ of even degree.
			\item  For any $v \in \Sigma \setminus \Sigma_p$, $p| \operatorname{Nm}(v)-1$.
			\item $[F:\mathbb{Q}]>7|\Sigma \setminus \Sigma_p|+2$.
			\item  $\chi$ is unramified outside $ \Sigma_p$, and $\chi(\operatorname{Frob}_v) \equiv 1~\operatorname{mod}~\varpi$ for any $v \in \Sigma \setminus \Sigma_p$.
		\end{itemize}
		Then there exists a nice prime in $C$. 
	\end{lem}
	
	\begin{proof}
		Under our assumptions, the Iwasawa algebra $\Lambda_F$ is a power series ring over $\cO$ of relative dimension $[F:\mathbb{Q}]$. Let $\kq$ be the minimal prime of $ R^{\ps,\ord}$ corresponding to $C$. By Theorem \ref{ordinary case}, the map $ \Lambda_F \to R^{\ps,\ord}/\kq$ is finite. It is also injective. Indeed, its kernel is a prime ideal since $R^{\ps,\ord}/\kq$ is a domain. If the kernel were non-zero, then $\dim R^{\ps,\ord}/\kq\le [F:\Q]$, contrary to the assumption that $\dim C=[F:\Q]+1$. Hence the morphism $C \to \Spec \Lambda_F$ is finite and surjective.
		
		\textbf{Proof of (1).} Using Lemma \ref{ordpseudo}, we know that for any $v|p$, there exists $n_v\in\{1,2\}$ such that for any irreducible prime $\kp \in C$,
		\[\rho(\kp)|_{G_{F_v}}\cong \begin{pmatrix} \psi^{univ}_{v,n_v}\modd\kp & *\\ 0 & \psi^{univ}_{v,3-n_v}\modd \kp\end{pmatrix}.\]
		Let $C^{\mathrm{red}}$ be the reducible locus of $C$. By \cite[Lemma 5.4.5]{pan2022fontaine} and Leopoldt's conjecture for the abelian field $F$, we have $\dim C^{\mathrm{red}}\le 2$. Since $F\ne\Q$, this is a proper closed subset of $C$. As $C\to\Spec\Lambda_F$ is finite, the image of $C^{\mathrm{red}}$ is also a proper closed subset. On the other hand, the regular algebraic primes of $\Lambda_F$ satisfying the strict Hodge--Tate inequalities prescribed by the collection $\{n_v\}_{v|p}$ are Zariski dense. Hence the image of $C^{\textnormal{reg}}$ in $\Spec \Lambda_F$ is dense, and the finiteness of $C\to\Spec\Lambda_F$ implies that $C^{\textnormal{reg}}$ is dense in $C$.
		
		By Theorem \ref{ordinary case}, every prime in $C^{\textnormal{reg}}$ comes from a twist of a Hilbert modular form. In particular, $1+\bar\chi$ is modular. To see the pro-modularity of the whole component, let $\widetilde{\kq}$ be the inverse image of $\kq$ in $R^{\ps}$ and choose a minimal prime $\mathfrak{P}$ of $R^{\ps}$ contained in $\widetilde{\kq}$. By Remark \ref{belong}, the quotient $R^{\ps}/\mathfrak{P}$, and hence also $R^{\ps}/\widetilde{\kq}$, factors through $R^{\ps,\{\xi_v\}}$ for some $\{\xi_v\}$. The closed image of the corresponding big Hecke algebra contains $C^{\textnormal{reg}}$, and hence, by density, contains $C$. Thus every prime in $C$ is pro-modular.
		
		\textbf{Proof of (2).} The following construction is adapted from \cite[Lemma 4.4.4]{Zhang2024}. We include the details since both the dimension estimate and the resulting local shapes will be used below.
		
		Write $s=|\Sigma\setminus\Sigma_p|$. For any $v \in \Sigma \setminus \Sigma_p$, we fix a lifting $\sigma_v $ of the generator of the $\mathbb{Z}_p$-quotient of $I_v$ and a lifting $\operatorname{Frob}_v$ of the Frobenius element. Write $S=\{\sigma_v, \operatorname{Frob}_v:~v \in \Sigma \setminus \Sigma_p\}$. Clearly, $|S|=2s$.
		
		Following notations in Definition \ref{pseudo}, we let $I$ be an ideal of $R^\textnormal{ps,ord}/\kq$ generated by $$\{\varpi, a(\sigma_{v})-1, a(\operatorname{Frob}_{v})-1, d(\operatorname{Frob}_{v})-1,~ v\ \in \Sigma \setminus \Sigma_p\}.$$ 
		We consider a minimal prime $\mathfrak{u}$ of $I$. By Krull's principal ideal theorem, $(R^\textnormal{ps,ord}/\kq)/\mathfrak{u}$ is of dimension at least $[F: \mathbb{Q}]-3|\Sigma \setminus \Sigma_p|$. For simplicity, we write such a domain as $R_0$. 
		
		Now we consider the following steps.\\
		
	    \textbf{Step 1}. We start from the CNL domain $R_0$.
		
		By Proposition \ref{innovation'}, we can find a partition of $S = S_1 \amalg S_2$, a positive integer $n>1$ prime to $p$ and a CNL domain $R_0'$ satisfying the following conditions.
		\begin{enumerate}
		\item  $R_0'$ is a quotient of $R_0$.
		\item  For any $\theta \in S_1$, we have $y(\theta, \alpha)=0$ for any $\alpha \in G_{F, \Sigma}$ in $R_0'$. For any $\theta', \theta'' \in S_2$, we have $y(\theta', \alpha)^n=y(\theta'', \alpha)^n$ for any $\alpha \in G_{F, \Sigma}$ in $R_0'$.
		\item  For any $\alpha \in G_{F, \Sigma}, \theta \in S_2$, either $ y(\theta, \beta)=0$ for any $\beta \in G_{F, \Sigma}$ or $u(\alpha, \theta)$ is well-defined and integral over $R_0'$.
		\item  We have $\operatorname{dim} R_0' \ge \operatorname{dim} R_0 - |S|$. If $S_2$ is not empty, then further $\operatorname{dim} R_0' \ge \operatorname{dim} R_0 - |S|+1$.
	\end{enumerate}
	
	   \textbf{Step 2}. We start from the CNL domain $R_0'$.
		
		Similar to \textbf{Step 1}, we can find a partition of $S = S_1' \amalg S_2'$, a positive integer $n'>1$ prime to $p$ and a CNL domain $R_0''$ satisfying the following conditions.
		
			\begin{enumerate}
		\item $R_0''$ is a quotient of $R_0'$.
		\item  For any $\theta \in S_1'$, we have $y(\alpha, \theta)=0$ for any $\alpha \in G_{F, \Sigma}$ in $R_0''$. For any $\theta', \theta'' \in S_2'$, we have $y(\alpha, \theta')^{n'}=y(\alpha, \theta'')^{n'}$ for any $\alpha \in G_{F, \Sigma}$ in $R_0''$.
		\item  For any $\alpha \in \operatorname{Gal}(F_\Sigma/F), \theta \in S_2'$, either $ y(\beta, \theta)=0$ for any $\beta \in G_{F, \Sigma}$ or $u'(\alpha, \theta)$ is well-defined and integral over $R_0''$.
		\item  We have $\operatorname{dim} R_0'' \ge \operatorname{dim} R_0' - |S|$. If $S_2'$ is not empty, then further $\operatorname{dim} R_0'' \ge \operatorname{dim} R_0' - |S|+1$.
	\end{enumerate}
	
		\textbf{Step 3}. We consider the domain $R_0''$.
		
		If both $S_2$ and $S_2'$ are not empty, then we fix an element $\theta_{i} \in S_2$ and an element $\theta_{j} \in S_2'$. Consider a prime $\mathfrak{p}_{ij}$ of $R_0''$ minimal over the ideal $(y(\theta_{i}, \theta_{j}))$ and let $R''=R_0''/\mathfrak{p}_{ij}$. By Proposition \ref{u}, we know that either $y(\theta_{i}, \alpha)=0$ for all $\alpha \in G_{F,\Sigma}$ or $y(\beta, \theta_{j})=0$ for all $\beta \in G_{F,\Sigma}$. By Krull's principal ideal theorem, we have $\operatorname{dim} R'' \ge \operatorname{dim} R_0'' -1$.
		
		If not, we take $R''=R_0''$.

		In conclusion, there are three possibilities.
		
		\begin{enumerate}
		 \item  $ y(\theta, \alpha) =y(\alpha, \theta)=0$ for all $\theta \in S$ and $\alpha \in G_{F, \Sigma}$.
		\item  There exists an element $\theta_{i} \in S$ such that either $y(\theta, \alpha)=0$ or $y(\theta, \alpha)^{n}=y(\theta_i, \alpha)^{n}$ for all $\theta \in S$ and $\alpha \in G_{F, \Sigma}$, and $y(\alpha, \theta) =0 $ for all $\theta \in S$ and $\alpha \in G_{F, \Sigma}$.
		\item  There exists an element $\theta_{j} \in S$ such that either $y(\alpha, \theta)=0$ or $y(\alpha, \theta)^{n'}=y(\alpha, \theta_j)^{n'}$ for all $\theta \in S$ and $\alpha \in G_{F, \Sigma}$, and $y(\theta, \alpha) = 0 $ for all $\theta \in S$ and $\alpha \in G_{F, \Sigma}$.
	\end{enumerate}
	
		To make the dimension estimate explicit, let $e=1$ if $S_2\ne\varnothing$ and $e=0$ otherwise, and define $e'$ similarly using $S_2'$. The first two steps give $\dim R_0'\ge \dim R_0-2s+e$ and $\dim R_0''\ge \dim R_0'-2s+e'$. The last step decreases the dimension by at most one, and it does so only if $e=e'=1$. Therefore, $\operatorname{dim} R'' \ge [F: \mathbb{Q}]-7s \ge 2$.
		
		Using \cite[Lemma 3.2.4]{Zhang2024}, we can choose an irreducible one-dimensional prime $\mathfrak{r}$ of $R''$. By the first claim of this lemma, we know that $\mathfrak{r}$ is pro-modular. Now we verify that $\mathfrak{r} $ is actually a nice prime in the sense of Definition \ref{nice}.
		
		Let $A$ be the normalization of $R''/\mathfrak{r}$. Since $R''/\mathfrak{r}$ is an excellent one-dimensional complete local domain, $A$ is a finite DVR. Enlarging its residue field if necessary, we write $A=\mathbb{F}'[[T]]$. By Definition \ref{construction}, we can construct a continuous irreducible representation $\rho(\mathfrak{r}) : G_{F, \Sigma} \to \operatorname{GL}_2(A)$ in the following form $$
		\rho(\mathfrak{r})(\sigma)=\begin{pmatrix}
			a(\sigma) & \frac{y(\sigma, \tau_{\mathfrak{r}})}{y(\sigma_{\mathfrak{r}}, \tau_{\mathfrak{r}})}\\
			y(\sigma_{\mathfrak{r}}, \sigma)& d(\sigma)
		\end{pmatrix}, ~y(\sigma_{\mathfrak{r}}, \tau_{\mathfrak{r}}) \ne 0.$$ 
		
		We first claim that the third condition holds. Otherwise, as $\rho(\mathfrak{r})$ is ordinary at each $v|p$, Lemma \ref{nirred} implies that both $\chi_{v,1}$ and $\chi_{v,2}$ are of finite order for each $v|p$. Since $A$ is a domain of characteristic $p$, it has no non-trivial $p$-power roots of unity. Thus, the restriction of these characters to $O_{F_v}^{\times}(p)$ is trivial, and the composite map $\Lambda_F \to R^{\ps, \ord}/\mathfrak{r} \hookrightarrow A$ factors through $\Lambda_F/\mathfrak{m}_{\Lambda_F}$. On the other hand, Theorem \ref{ordinary case} implies that $R''/\mathfrak{r}$ is finite over $\Lambda_F$, and its normalization $A$ is finite over $R''/\mathfrak{r}$ since the latter is Nagata (see \cite[Lemma 10.162.8]{stacks-project}). Hence $A$ would be finite over the finite field $\Lambda_F/\mathfrak{m}_{\Lambda_F}$, contrary to $\dim A=1$. Therefore, the third condition holds. It remains to check the second and the last conditions. Here we fix a complex conjugation $c \in G_{F, \Sigma}$.
		
		In case (1), our construction gives $\rho(\mathfrak{r})(\sigma_v)=\rho(\mathfrak{r})(\operatorname{Frob}_v)=1$ for every $v\in\Sigma\setminus\Sigma_p$. Since $\sigma_v$ and $\operatorname{Frob}_v$ generate the relevant maximal pro-$p$ quotient of $G_{F_v}$, we have $\rho(\mathfrak{r})|_{G_{F_v}}=\begin{pmatrix}
		1 & 0\\
		0 & 1
		\end{pmatrix}$ for any $v \in \Sigma\setminus\Sigma_p$. By \cite[Lemma 2.13]{Skinner_1999}, we may choose a stable lattice whose reduction is a non-split reducible representation; after interchanging the two basis vectors if necessary, its reduction is upper triangular. Since the local representations at all $v\in\Sigma\setminus\Sigma_p$ are trivial, changing the lattice does not affect the last condition in Definition \ref{nice}. Thus all the desired conditions hold in this case.
		
		In case (2),  if $y(\theta_{i}, \alpha)=0$ for all $\alpha \in G_{F, \Sigma}$, then it is the same as case (1). If not, by Proposition \ref{u}, we know that $y(\theta_{i}, \tau_{\mathfrak{r}}) \ne  0$ (otherwise $y(\sigma_{\mathfrak{r}}, \tau_{\mathfrak{r}}) = 0$). 
		
		We claim that $u(\sigma_{\mathfrak{r}}, \theta_{i})= \frac{y (\sigma_{\mathfrak{r}}, \tau_{\mathfrak{r}})}{y(\theta_{i}, \tau_{\mathfrak{r}})}$ is a unit in $A$. By the third part in \textbf{Step 1}, it is integral over $R_0'$, and hence belongs to $A$ since $A$ is normal. On the other hand, $u(\theta_i, \sigma_{\mathfrak{r}})=\frac{y(\theta_i,\tau_{\mathfrak r})}{y(\sigma_{\mathfrak r},\tau_{\mathfrak r})}$ is the upper-right entry of $\rho(\mathfrak r)(\theta_i)$, and hence also belongs to $A$. As $u(\theta_i, \sigma_{\mathfrak{r}})u(\sigma_{\mathfrak{r}}, \theta_{i})=1$, both elements are units in $A$. Then we can take $\rho(\mathfrak{r})^o=P\rho(\mathfrak{r})P^{-1}$, where $$P= \begin{pmatrix}
			u(\sigma_{\mathfrak{r}}, \theta_{i}) & 0\\
			0 & 1
		\end{pmatrix} \in \operatorname{GL}_2(A).$$
		
		From our construction, for any $\theta\in S$, either $y(\theta,\alpha)=0$ for all $\alpha\in G_{F,\Sigma}$, in which case the corresponding upper-right entry is zero, or $u(\theta,\theta_i)$ is well-defined and satisfies $u(\theta,\theta_i)^n=1$. As $(n,p)=1$, Hensel's lemma implies that every such $u(\theta,\theta_i)$ belongs to $\mathbb{F}' \hookrightarrow \mathbb{F}'[[T]]=A$. Hence, the last condition holds in this case. To check the second one, as $y(c,c)=0$, we may assume $
		\rho(\mathfrak{r})(c)^o=\begin{pmatrix}
			1 & 0\\
			*& -1
		\end{pmatrix}$ after conjugating the matrix $\begin{pmatrix}
		1 & \frac{u(c, \theta_{i})}{2}\\
		0 & 1
		\end{pmatrix} $ (note that $ \frac{u(c, \theta_{i})}{2}=\frac{1}{2}u(c, \sigma_{\mathfrak{r}})u(\sigma_{\mathfrak{r}}, \theta_{i})$ is an element of $A$ and
		 for any $\sigma \in G_{F_v}$ and any $v \in \Sigma\setminus\Sigma_p$, $
		 \rho(\mathfrak{r})(\sigma)$ is stable under this conjugation). As $u(\theta_{i}, \theta_{i})=1 \ne 0$, the reduction of $\rho(\mathfrak{r})^o$ satisfies $$
		 \overline{\rho(\mathfrak{r})^o}(c)= \begin{pmatrix}
		 	1 & 0\\
		 	0 & -1
		 \end{pmatrix}, ~~~ \overline{\rho(\mathfrak{r})^o}(\theta_{i})= \begin{pmatrix}
		 1 & 1\\
		 0 & 1
		 \end{pmatrix},$$
		and hence a non-split reducible representation.
		
			In case (3), we consider the continuous irreducible Galois representation $\rho(\mathfrak{r})' : G_{F, \Sigma} \to \operatorname{GL}_2(A)$ in the following form $$
		\rho(\mathfrak{r})'(\sigma)=\begin{pmatrix}
			a(\sigma) & y(\sigma, \tau_{\mathfrak{r}})\\
			\frac{y(\sigma_{\mathfrak{r}}, \sigma)}{y(\sigma_{\mathfrak{r}}, \tau_{\mathfrak{r}})}& d(\sigma)
		\end{pmatrix}, ~y(\sigma_{\mathfrak{r}}, \tau_{\mathfrak{r}}) \ne 0.$$
		If $y(\alpha, \theta_j)=0$ for all $\alpha \in G_{F, \Sigma}$, then we can take $$
		\rho(\mathfrak{r})^o=\begin{pmatrix}
			0 & 1\\
			1 & 0
		\end{pmatrix} \rho(\mathfrak{r})'\begin{pmatrix}
			0 & 1\\
			1 & 0
		\end{pmatrix}. $$
		Using a similar argument as in case (1), we can verify that $\mathfrak{r}$ is a nice prime.
		If not, the same argument as in case (2) shows that $u'(\tau_{\mathfrak r},\theta_j)$ is integral over $R_0''$ and hence belongs to $A$, while $u'(\theta_j,\tau_{\mathfrak r})$ is an entry of $\rho(\mathfrak r)'(\theta_j)$ and also belongs to $A$. Their product is $1$, so both are units. We can therefore take $$ \rho(\mathfrak{r})^o=\begin{pmatrix}
			0 & u'(\tau_{\mathfrak{r}}, \theta_{j})\\
			1 & 0
		\end{pmatrix} \rho(\mathfrak{r})'\begin{pmatrix}
			0 & 1\\
			u'(\theta_{j}, \tau_{\mathfrak{r}}) & 0
		\end{pmatrix}. $$
		Using a similar argument as in case (2), we can also verify that $\mathfrak{r}$ is a nice prime.
	\end{proof}
	
	By the previous lemma, to find a nice prime, it is crucial to find irreducible components in $\Spec R^{\ps,\ord}$ of large dimension. For this purpose, we recall a classical result in the rest of this subsection.
	
	Let $\Sigma^o$ be a subset of $\Sigma_p$. We define $\bar{\psi}_{v,1}:G_{F_v}\to \F^\times$ to be $\mathbf{1}$ if $v\in\Sigma^o$ and $\bar{\chi}|_{G_{F_v}} (\ne \mathbf{1})$ if $v\notin\Sigma^o$. 
	
	Consider the universal deformation ring $R^{\ps,\ord}_{\Sigma^o}$ which pro-represents the functor from $\mathfrak{U}_{\mathcal{O}}$ to the category of sets sending $R$ to the set of tuples $(T,\{\psi_{v,1}\}_{v\in\Sigma_p})$ where
	\begin{itemize}
		\item For any $v|p$, $\psi_{v,1}:G_{F_v}\to R^\times$ is a character which lifts $\bar{\psi}_{v,1}$ and satisfies
		\[T|_{G_{F_v}}=\psi_{v,1}+\psi_{v,1}^{-1}\chi.\]
		Moreover $T$ is $\psi_{v,1}$-ordinary.
	\end{itemize}
	From the definition, it is clear that $R^{\ps,\ord}_{\Sigma^o}$ is a quotient of $R^{\ps,\ord}$.
	
	Write \[H^1_{\Sigma^o}(F):=\ker(H^1(G_{F,\Sigma},\F(\bar{\chi}^{-1}))\stackrel{\mathrm{res}}{\longrightarrow}\bigoplus_{v\in \Sigma_p\setminus \Sigma^o}H^1(G_{F_v},\F(\bar{\chi}^{-1}))).\]
	Then it is easy to find that any cohomology class $b=[\phi_b]\in H^1_{\Sigma^o}(F)$ defines a two-dimensional representation
	\[\bar{\rho}_b:G_{F,\Sigma}\to \GL_2(\F),~g\mapsto \begin{pmatrix} 1 & \phi_b(g)\bar{\chi}(g)\\0& \bar{\chi}(g)\end{pmatrix},\]
	which can be viewed as an extension of $\bar\chi$ by $\mathbf{1}$. 
	
	For any non-zero $b\in H^1_{\Sigma^o}(F)$, we will write $R^{\ord}_{b}$ for the universal (unframed) deformation ring which pro-represents the functor from $\mathfrak{U}_{\mathcal{O}}$ to the category of sets sending $R$ to the set of tuples $(\rho_R,\{\psi_{v,1}\}_{\Sigma_p})$ where
	\begin{itemize}
		\item $\rho_R:G_{F,\Sigma}\to\GL_2(R)$ is a lifting of $\bar{\rho}_b$ with determinant $\chi$.
		\item For $v|p$, $\psi_{v,1}:G_{F_v}\to R^\times$ is a lifting of $\bar{\psi}_{v,1}$ such that $\rho_R$ has a (necessarily unique) $G_{F_v}$-stable rank one $R$-submodule, which is a direct summand of $\rho_R$ as an $R$-module, with $G_{F_v}$ acting via $\psi_{v,1}$.
	\end{itemize}
	By the universal property, there is a natural map $ R^{\ps,\ord}_{\Sigma^o} \to R^{\ord}_{b}$.
	
	\begin{prop}\label{dim for global deformation ring}
		Each irreducible component of $R^{\ord}_b$ has dimension at least $[F:\Q]+1$ if $H^0(G_{F,\Sigma},\ad^0\bar{\rho}_b(1))=0$ and $[F:\Q]$ in general.
	\end{prop}
	
	\begin{proof}
		Arguing as in \cite[Lemma 5.3.4]{pan2022fontaine}, we can show that the local framed ordinary deformation ring is of dimension at least $6$. Then our result follows from the Greenberg-Wiles formula (see \cite[Proposition 3.24]{Gee_2022} for example). 
	\end{proof}
	
	\begin{rem}\label{dim not enough}
		Let $\kq$ be a one-dimensional prime of $ R^{\ps,\ord}_{\Sigma^o}$ of characteristic $0$ such that the corresponding Galois representation $\rho(\kq): G_{F, \Sigma} \to \GL_2(k(\kq))$ is absolutely irreducible, and let $\bar{\rho_{b'}}$ be its non-split residual representation, where $b'$ is a non-zero element in $H^1_{\Sigma^o}(F) $. Let $R^{\ord}_{b'}$ be the universal ordinary deformation ring defined as above, and $\rho(\kq) $ defines a one-dimensional prime $\kq_{b'}$ of $R^{\ord}_{b'}$. Arguing as in the second part of \cite[Section 5.3.5]{pan2022fontaine}, with the prescribed ordinary characters, we have an isomorphism $\widehat{(R^{\ps,\ord}_{\Sigma^o})_{\kq}}\cong\widehat{(R^{\ord}_{b'})_{\kq_{b'}}}$. Then by the previous proposition, we have $\dim (R^{\ps,\ord}_{\Sigma^o})_{\kq} \ge [F: \mathbb{Q}]-1$. Here $\bar\rho_{b'}$ is an extension of $\bar\chi$ by $\mathbf{1}$, so if $\bar{\chi}\ne\omega$, then $H^0(G_{F,\Sigma},(\ad^0\bar\rho_{b'})(1))=0$. In this case, $\dim (R^{\ps,\ord}_{\Sigma^o})_{\kq}\ge[F:\Q]$, and we can use Lemma \ref{find nice} to find nice primes in some irreducible component containing $\kq$.
		
		If $p>3$ and $\bar{\chi}=\omega^{\pm1}$, after a finite twist and a suitable choice of lattice, we may assume $\bar{\chi}=\omega^{-1}\ne\omega$ with the same extension convention. Then we can use the argument in the previous paragraph to find nice primes. However, if $p=3$, the $\mod p$ cyclotomic character is quadratic, and the resulting lower bound for each irreducible component containing $\kq$ is only $[F:\Q]$. In this case, this bound does not suffice to apply Lemma \ref{find nice}, and we need further arguments.
	\end{rem}
	
	\subsection{The residually irreducible case} \label{irr sec} Let $F$ be a totally real field of even degree in which $p$ splits completely. Write $\Sigma_p$ as the set of places of $F$ above $p$. Let $\Sigma$ be a finite set of finite places of $F$ containing $\Sigma_p$ such that for all $v \in \Sigma \setminus \Sigma_p$, we have $p \mid \operatorname{Nm}(v)-1$.
	
	Let $\bar{\rho}: G_F \to \GL_2(\F)$ be a continuous, odd, absolutely irreducible Galois representation. Note that in this case, the deformation functor of \(\bar\rho\) is naturally identified with the pseudo-deformation functor of \(\operatorname{tr}\bar\rho\). We assume that $\bar{\rho}$ is unramified outside $p$ and for each $v \in \Sigma \setminus \Sigma_p$, $\bar{\rho}|_{G_{F_v}}$ is trivial. Further assume that $\bar\rho$ arises from a regular algebraic cuspidal automorphic representation $\pi_0$ of
	$\mathrm{GL}_2(\mathbb A_F)$, unramified outside $\Sigma$ and
	satisfying
	\(
	\pi_{0,v}^{\mathrm{Iw}_v}\ne 0,
	v\in\Sigma\setminus\Sigma_p
	\). Write $\det \bar{\rho} =\bar{\chi}$, and let $\chi : G_{F, \Sigma} \to \cO^\times$ be a de Rham lifting of the character $\bar{\chi}$ such that $\chi$ is unramified outside $\Sigma_p$.
	
    By the global class field theory, we may view $\psi=\chi\varepsilon$ as a character of $\AFi/F^\times_{+}$. Let $D$ be the quaternion algebra over $F$ which is ramified exactly at all infinite places. Fix an isomorphism between $\DAi$ and $\GL_2(\A_F^\infty)$. We define a tame level $U^p_0=\prod_{v\nmid p}U_{0,v}$ as follows: $U_{0,v}=\GL_2(O_{F_v})$ if $v\notin \Sigma$ and 
    \[U_{0,v}=\mathrm{Iw}_v:=\{g\in\GL_2(O_{F_v}),g\equiv \begin{pmatrix}*&*\\0&*\end{pmatrix}\mod \varpi_v\}\]
    otherwise. For $v\in\Sigma\setminus\Sigma_p$, let $\Delta_v$ be the maximal $p$-power quotient of $k(v)^\times$. Let $U^p_1$ be the kernel of the natural map $U^p_0\to\prod_{v\in\Sigma\setminus\Sigma_p}\Delta_v$ whose $v$-component is induced by $\begin{pmatrix}a&b\\c&d\end{pmatrix}\mapsto \frac{a}{d}\mod\varpi_v$.
    
    Let $\xi_v:k(v)^\times\to\mathcal{O}^\times$, $v\in\Sigma\setminus\Sigma_p$, be characters of $p$-power order, and let $\xi$ be the resulting character of $U^p_0$. Then $U^p_1\subseteq\ker(\xi)$, so $S_{\psi,\xi}(U^p_0)\hookrightarrow S_\psi(U^p_1)$ induces a surjection $\T:=\T_\psi(U^p_1)\twoheadrightarrow \T^{\{\xi_v\}}:=\T_{\psi,\xi}(U^p_0)$.
    As $\bar{\rho}$ is modular, we know that $T_v-\tr\bar{\rho}(\Frob_v),v\notin \Sigma$ and $\varpi$ generate a maximal ideal $\km$ (resp. $\km_{\xi}$) of $\T $ (resp. $\T^{\{\xi_v\}}$), and the preceding map induces a surjection $\T_\km\twoheadrightarrow\T^{\{\xi_v\}}_{\km_\xi}$. Since $\bar{\rho}$ is absolutely irreducible, there is a two-dimensional representation $\rho_{\km}:G_{F,\Sigma}\to\GL_2(\T_{\km})$ (resp. $\rho_{\km_{\xi}}:G_{F,\Sigma}\to\GL_2(\T^{\{\xi_v\}}_{\km_{\xi}})$) with determinant $\chi$ such that the trace of $\Frob_v$ is $T_v$ for $v\notin \Sigma$.
	
	Let $R_F^{\operatorname{univ}}$ be the universal deformation ring of $\bar\rho$ with determinant $\chi$. By the universal property, there is a natural surjection $R_F^{\operatorname{univ}}\twoheadrightarrow\T_\km$. Let $R^{\{\xi_v\}}$ be the quotient of $R_F^{\operatorname{univ}}$ parametrizing all deformations of $\bar\rho$ satisfying  $\operatorname{tr}(\rho)|_{I_{F_v}}=\xi_v+\xi_v^{-1}$ for all $v\in \Sigma\setminus \Sigma_p$. By the universal property,  there is a natural surjection $R^{\{\xi_v\}} \twoheadrightarrow \T^{\{\xi_v\}}_{\km_{\xi}}$, compatible with the two surjections above. Similar to Definition \ref{nice}, we can also define pro-modular and nice primes in this case.
	
	\begin{defn}
		We say a prime of $R^{\{\xi_v\}}$ is \textit{pro-modular} if it comes from a prime of $\T^{\{\xi_v\}}_{\km_{\xi}}$. A pro-modular prime $\kq$ of $R^{\{\xi_v\}}$ is \textit{nice} if 
		\begin{enumerate}
			\item $R^{\{\xi_v\}}/\kq$ is one-dimensional with characteristic $p$.
			\item $\rho(\kq)$ is absolutely irreducible and not induced from a quadratic extension of $F$. Here $\rho(\kq)$ denotes the push-forward of the universal deformation on $R^{\{\xi_v\}}$ to $k(\kq)$.
			\item $\rho(\kq)|_{G_{F_v}}$ is trivial for $v\in \Sigma\setminus\Sigma_p$.
		\end{enumerate}
	\end{defn}
	
	\begin{prop}
		Let $\kq\in\Spec R^{\{\xi_v\}}$ be a nice prime. Then the kernel of $(R^{\{\xi_v\}})_\kq\to \T^{\{\xi_v\}}_\kq$ is nilpotent, where we also view $\kq$ as a prime of $\T^{\{\xi_v\}}_{\km_{\xi}}$.
	\end{prop}
	
	\begin{proof}
		The proof is nearly the same as Theorem \ref{R=T}. One can see \cite[Section 8.2 \& 8.3]{pan2022fontaine} and \cite[Proposition 9.0.6]{pan2022fontaine} for details. As we discussed in Remark \ref{4.1.9} and Remark \ref{4.1.13}, when $p=3$, we need to use local-global compatibility results in Section \ref{section lgc} and the normality of the local framed deformation ring $R_v^{\square}$ (see Proposition \ref{dim of loc framed der ring}) instead.
	\end{proof}
	
	\begin{prop}\label{big R=T in irr}
		If $[F: \mathbb{Q}] \ge 4|\Sigma \setminus \Sigma_p|+2$, then the kernel of $R^{\{\xi_v\}}\to\T^{\{\xi_v\}}_{\km_{\xi}}$ is nilpotent. Furthermore, the kernel of $R_F^{\operatorname{univ}}\to\T_\km$ is nilpotent; in particular, $\Spec R_F^{\operatorname{univ}}$ and $ \Spec \T_\km$ have the same underlying topological space.
	\end{prop}
	
	\begin{proof}
		The proof of the first part is the same as \cite[Proposition 9.0.4]{pan2022fontaine} as $\xi_v \equiv 1 \modd \varpi$. 
		
		For the second part, let $\mathfrak a$ be the kernel of $R_F^{\operatorname{univ}}\twoheadrightarrow\T_\km$. It suffices to show that $\mathfrak a\subseteq\kp$ for every generic point $\kp$ of $\Spec R_F^{\operatorname{univ}}$. For $v \in \Sigma\setminus\Sigma_p$, fix a topological generator $\sigma_v$ of the pro-$p$ part of $I_v$, and let $\alpha_v, \alpha_v^{-1}$ be the eigenvalues of $\rho(\kp)(\sigma_v)$. By \mbox{Theorem~7.5.3 of \cite{Neukirch_2008}}, $\alpha_v^{\operatorname{Nm}(v)-1}=1$; since $\alpha_v\equiv1\modd\varpi$, it has $p$-power order. Hence $R_F^{\operatorname{univ}}\twoheadrightarrow R_F^{\operatorname{univ}}/\kp$ factors through $R^{\{\xi_v\}}$ for a suitable tuple $\{\xi_v\}$. By the first part, the induced prime of $R^{\{\xi_v\}}$ comes from $\T^{\{\xi_v\}}_{\km_\xi}$; pulling it back along $\T_\km\twoheadrightarrow\T^{\{\xi_v\}}_{\km_\xi}$ gives $\mathfrak a\subseteq\kp$. Thus $\mathfrak a$ is contained in the nilradical and hence is nilpotent.
	\end{proof}

	\section{The proof in the residually reducible case}\label{section 5} In this section, we prove the non-ordinary Fontaine-Mazur conjecture in the residually reducible case. 
	
	This section is organized as follows. Subsection \ref{sec propagation} fixes the auxiliary setting and establishes propagation; Subsections \ref{sec special generic} and \ref{sec special ordinary} give two ordinary special cases. In Subsection \ref{sec partial ordinary}, dimension estimates allow an induction through partial ordinary components, and propagation carries pro-modularity back to the original component. This proves the residually reducible case of the main theorem. Subsection \ref{sec potential promodularity} then deduces potential pro-modularity after one abelian base change.
	
	Let $p$ be an odd prime. Throughout this section, we denote by $K$  a $p$-adic local field with its ring of integers $\cO$, a uniformizer $\varpi$ and its residue field $\F$.
	
	\subsection{Setup and propagation of pro-modularity}\label{sec propagation}
	Our main goal is to prove the following theorem.
	
   \begin{thm} \label{mainthm}
	 Let $F$ be a totally real abelian extension of $\Q$ in which $p$ splits completely. Suppose 
	\[\rho:\Gal(\overbar{F}/F)\to \GL_2(\cO)\] 
	is a continuous irreducible representation with the following properties
	\begin{itemize}
		\item $\rho$ ramifies at only finitely many places.
		\item Let $\bar{\rho}$ be the reduction of $\rho$ modulo $\varpi$. We assume its semi-simplification has the form $\bar{\chi}_1\oplus\bar{\chi}_2$ and ${\bar{\chi}_1}/\bar{\chi}_2$ can be extended to a character of $G_\Q$.
		\item $\rho|_{G_{F_v}}$ is irreducible and de Rham of distinct Hodge-Tate weights for any $v|p$. 
		\item $({\bar{\chi}_1}/{\bar{\chi}_2})(c)=-1$ for any complex conjugation $c\in \Gal(\overbar{F}/F)$.
	\end{itemize}
	Then $\rho$ arises from a twist of a Hilbert modular form, i.e. a regular algebraic cuspidal automorphic representation of $\GL_2(\A_F)$.
\end{thm}

	After a finite twist, we may assume $\bar{\chi}_1$ is trivial. Let $\Sigma$ be a finite set of places containing all the ramified finite places of $\rho$ and let $\Sigma_p$ be the set of primes above $p$. Write $\det \rho =\chi$ and $\bar{\chi}$ for its reduction modulo $\varpi$.  Since \(F/\mathbb Q\) is abelian and totally real, Leopoldt's conjecture holds for \(F\).  Hence we may write $\chi= \psi \varepsilon^k$, where $\psi$ is a finite order character.  Write \(
	\chi_{\mathbb Q}:=\tilde{\bar{\psi}}\varepsilon^k
	\), where $\tilde{\bar{\psi}} $ is the Teichmüller lifting of $\bar{\psi}$. As $\bar{\chi}$ can be extended to a character of $G_\Q$, we can find that $\chi_{\mathbb Q} $ is also a character of $G_\Q$ and it ramifies at the same set of places away from $p$ as $\bar{\chi}$. Since \(p\) is odd, after enlarging \(\mathcal O\) if necessary, there exists a finite order character
	\(
	\varphi:G_F\to \mathcal O^\times
	\)
	lifting the trivial residual character such that \(\varphi^2=\chi_{\mathbb Q}\chi^{-1}\).
	Replacing \(\rho\) by \(\rho\otimes\varphi\), we may assume that
	\(
	\chi=\det\rho
	\)
	is the restriction of a \(p\)-adic character of \(G_{\mathbb Q}\).  This
	replacement does not affect the statement of Theorem \ref{mainthm}, since
	\(\varphi\) is a finite order character and hence comes from an algebraic
	Hecke character of \(F\).

	By Grunwald-Wang's theorem (see \cite[Theorem 5, page 80]{artin}) and abelian base change, we may suppose the number field $F$ satisfies the following assumptions (the simultaneous choice is explained in Subsection \ref{sec potential promodularity}):
	\begin{itemize}
		\item $F$ is a totally real abelian extension of $\mathbb{Q}$ of even degree.
		\item $p$ splits completely in $F$.
		\item  For any $v \in \Sigma \setminus \Sigma_p$, $p| \operatorname{Nm}(v)-1$.
		\item $[F:\mathbb{Q}]>7|\Sigma \setminus \Sigma_p|+6$.
		\item  $\chi$ is unramified outside $ \Sigma_p$, and $\chi(\operatorname{Frob}_v) \equiv 1~\operatorname{mod}~\varpi$ for any $v \in \Sigma \setminus \Sigma_p$.
	\end{itemize}
	
	Consider the functor from $\mathfrak{U}_{\mathcal{O}}$ to the category of sets sending $R$ to the set of $2$-dimensional pseudo-representations of $G_{F,\Sigma}$ which lift $1+\bar{\chi}$ with determinant $\chi$. This is pro-represented by a CNL ring $R^{\ps}$ and $\tr\rho$ gives rise to a prime ideal $\kp\in\Spec R^{\ps}$.  As $F$ is an abelian totally real field, Leopoldt's conjecture holds in this case, and hence the reducible locus of $R^{\ps}$ is of dimension at most $2$ (for example, see \cite[5.4.5]{pan2022fontaine}).

	We first fix a tame level for the pro-modularity statements below. Let $D$ and $\psi=\chi\varepsilon$ be as in Subsection \ref{sec 4.1}. As in Subsection \ref{irr sec}, put $U^p_0=\prod_{v\nmid p}U_{0,v}$, where $U_{0,v}=\GL_2(O_{F_v})$ if $v\notin\Sigma$ and $U_{0,v}=\mathrm{Iw}_v$ otherwise. For $v\in\Sigma\setminus\Sigma_p$, let $\Delta_v$ be the maximal $p$-power quotient of $k(v)^\times$, and let $U^p_1$ be the kernel of $U^p_0\to\prod_{v\in\Sigma\setminus\Sigma_p}\Delta_v$, with $v$-component induced by $\begin{pmatrix}a&b\\c&d\end{pmatrix}\mapsto a/d\mod\varpi_v$. Enlarge $\cO$ so that all characters of these finite groups take values in $\cO^\times$.
	
	For any tuple of $p$-power order characters $\xi_v:k(v)^\times\to\cO^\times$, let $\xi$ be the corresponding character of $U^p_0$. Since $U^p_1\subseteq\ker\xi$, the inclusion $S_{\psi,\xi}(U^p_0)\hookrightarrow S_\psi(U^p_1)$ induces a surjection $\T:=\T_\psi(U^p_1)\twoheadrightarrow\T_\xi:=\T_{\psi,\xi}(U^p_0)$. Under Assumption \ref{assumption}, let $\km_\xi$ be the residual maximal ideal of $\T_\xi$ and $\km$ its inverse image in $\T$. These ideals are determined by $1+\bar\chi$, and the map above induces $\T_\km\twoheadrightarrow(\T_\xi)_{\km_\xi}$. By \cite[Lemma 4.8.5]{pan2022fontaine}, occurrence of this residual system for one $\xi$ implies its occurrence for every $\xi$.
	
	The character decomposition for the finite abelian group $U^p_0/U^p_1$ after inverting $p$, together with the $\cO$-torsion-freeness of $S_\psi(U^p_1)$, gives an injection $\T_\km\hookrightarrow\prod_\xi(\T_\xi)_{\km_\xi}$. Thus $\Spec\T_\km$ is the union of the closed images of the spectra of these quotients. Consequently, the definition of pro-modularity in Definition \ref{nice}(1) is equivalent to coming from a prime of $\T_\km$ under the natural surjection $R^{\ps}\twoheadrightarrow\T_\km$. We use this fixed Hecke algebra throughout this section. An irreducible component of $\Spec R^{\ps}$ is called \textit{pro-modular} if its generic point comes from $\T_\km$; equivalently, every prime on it is pro-modular.
	
	Moreover, under the canonical identifications $R^{\ps,\xi}/\varpi\cong R^{\ps,\xi'}/\varpi$, the characteristic-$p$ pro-modular loci coincide for all $\xi$, by \cite[Section 4.8.4]{pan2022fontaine}. Thus a nice prime can be used with the tame characters attached to any ambient irreducible component containing it. We do not impose Assumption \ref{assumption} as an additional hypothesis in this section: the initial pro-modular points below are supplied by Lemma \ref{find nice}(1), using ordinary modularity. Whenever Lemma \ref{promodular to nice} is applied, the pro-modular point in its hypothesis already supplies the required residual maximal ideal.
	
	\begin{lem}\label{promodular to nice}
		Suppose that Assumption \ref{assumption} holds. Let $\mathfrak{r}$ be a one-dimensional irreducible pro-modular prime of $R^{\ps}$. Then every irreducible component of $R^{\ps}$ containing $\mathfrak{r}$ contains a nice prime, and hence is pro-modular.
	\end{lem}
	
	\begin{proof}
		%The proof is almost the same as \cite[Proposition 4.5.1]{Zhang2024}. The unique difference is that we need to use Corollary \ref{dim of hecke} instead of \cite[Theorem 3.6.1]{pan2022fontaine} in the proof of \cite[Proposition 4.5.1]{Zhang2024} to cover the case $p=3$.
		We first give a useful lemma on Jacobson spectra.
		
		\begin{lem}\label{Jacobon}
			Let $R$ be a noetherian local ring with maximal ideal $\mathfrak{m}_R$ satisfying $\dim R >0$. Suppose that $V \subsetneqq \Spec R$ is a closed subset. Then there exists a one-dimensional prime $\mathfrak{p} \notin V$.
		\end{lem}
		
		\begin{proof}[Proof of Lemma \ref{Jacobon}]
			Using \cite[Lemma 28.6.4]{stacks-project}, the scheme $\Spec R \setminus\{\mathfrak{m}_R\}$ is Jacobson, i.e., the underlying topological space is Jacobson. As $V \ne \Spec R$, there exists a closed point of $\Spec R \setminus\{\mathfrak{m}_R\}$ not contained in $V$. This shows the result.
		\end{proof}
		
		 Let $Z$ be the set of all irreducible components of $(R^{\ps})_{\mathfrak{r}} $. Write $Z= Z_1 \amalg Z_2$, where $Z_1$ is the subset of $Z$ consisting of all pro-modular irreducible components. We first show that if $Z_1$ is not empty, then $Z_2$ is empty.
		
		If not, by Proposition \ref{connect dim}, a component in $Z_1$ and a component $Y\in Z_2$ meet in dimension at least $2[F:\Q]-1$ in $\Spec (R^{\ps})_{\mathfrak r}$. Taking the closure of an irreducible component of their intersection in $\Spec R^{\ps}$ gives a pro-modular prime $\mathfrak r'\subseteq\mathfrak r$ with $\dim R^{\ps}/\mathfrak r'\ge2[F:\Q]-1$. We also write $Y$ for its closure in $\Spec R^{\ps}$. The following specializations are made in $R^{\ps}/\mathfrak r'$, not in $(R^{\ps})_{\mathfrak r}$; in particular, they remain on $Y$ even if they do not generalize $\mathfrak r$.
		
		The initial cut and \textbf{Steps 1--3} in the proof of Lemma \ref{find nice}(2) do not require ordinarity and decrease the dimension by at most $1+7|\Sigma\setminus\Sigma_p|$. Applying them to $R^{\ps}/\mathfrak r'$, we find a prime $\mathfrak{r}' \subseteq \mathfrak{r}''$ satisfying the following conditions:
		\begin{itemize}
			\item $\varpi \in \mathfrak{r}''$.
			\item  The dimension of $\mathfrak{r}'' $ is at least $ 2[F: \mathbb{Q}]-7|\Sigma\setminus\Sigma_p|-2 \ge [F: \mathbb{Q}]+2$.
			\item  For any irreducible one-dimensional prime $\mathfrak{t}$ containing $\mathfrak{r}''$, there exists a $2$-dimensional Galois representation $\rho(\mathfrak{t}) : G_{F, \Sigma} \to \GL_2(k(\mathfrak{t}))$ satisfying the conditions in (2) of Definition \ref{nice} except the third one.
		\end{itemize}
		
		If $\bar{\chi}$ is not quadratic, Lemma \ref{Jacobon} and the dimension bound on the reducible locus give an irreducible one-dimensional prime containing $\mathfrak r''$. It is nice by Lemma \ref{nirred}. Choose the tame characters $\xi$ for which $Y$ is a component of $\Spec R^{\ps,\xi}$, as in Remark \ref{belong}. The characteristic-$p$ compatibility just recalled allows us to apply Theorem \ref{R=T} with these characters. Thus $Y$ is pro-modular, contradicting the definition of $Z_2$.
		
		If $\bar{\chi}$ is quadratic, we know that the splitting field $F(\bar{\chi})$ is an abelian CM field. Then Leopoldt's conjecture is known in this case. Therefore, the dihedral locus in $R^{\ps}/(\varpi)$ is of dimension at most $ \frac{1}{2}[F(\bar{\chi}): \mathbb{Q}] +1=[F: \mathbb{Q}]+1$. (This is because if a representation $\rho$ is dihedral, i.e. $\rho \cong \Ind_{G_{F(\bar{\chi})}}^{G_F} \theta$, then $\theta$ is a lifting of the global character $\bar{\theta}: G_{F(\bar{\chi})} \to \F^\times$. In other words, to study the dimension of the dihedral locus, we just need to know the dimension of the $\modd p$ universal deformation ring parametrizing all deformations of the character $\bar{\theta}$, which is at most the $\mathbb{Z}_p$-rank of the maximal $\mathbb{Z}_p$-extension of $ F(\bar{\chi})$.) As the dimension of $\mathfrak{r}'' $ is at least $[F: \mathbb{Q}]+2$, using Lemma \ref{Jacobon}, we may find a prime containing $\mathfrak{r}''$, whose corresponding Galois representation is irreducible and not dihedral. From our choice of $\mathfrak{r}''$, we know that this is a nice prime. Using Theorem \ref{R=T} again, we get a contradiction.
		
		As $\mathfrak{r}$ is pro-modular, Corollary \ref{dim of hecke} gives a pro-modular prime $\mathfrak{r}_1\subseteq\mathfrak{r}$ with $\dim R^{\ps}/\mathfrak r_1\ge1+2[F:\Q]$. Arguing as above, we find a nice prime containing $\mathfrak r_1$. Any ambient irreducible component containing $V(\mathfrak r_1)$ contains both this nice prime and $\mathfrak r$, and is pro-modular by Theorem \ref{R=T}. Hence $Z_1$ is not empty.
	\end{proof}

	By Proposition \ref{connect dim}, we can choose an irreducible component $C$ of $\Spec R^\ps$ containing $\kp$ of dimension at least $2[F: \mathbb{Q}]+1$. By Corollary \ref{classicality}, it suffices to show that $C$ is pro-modular.
	
	\begin{rem}
		Actually, our argument proves the following stronger componentwise statement: each irreducible component of $R^\ps$ of dimension at least $2[F:\mathbb{Q}]+1$ is pro-modular. This generalizes the big $R=\mathbb{T}$ theorem studied in \cite{Zhang2024}.
	\end{rem}
	
	Write $d=[F:\Q]$. Restriction gives a natural map $R^{\ps}_p:=\widehat\bigotimes_{v|p}R^{\ps}_v\to R^{\ps}$, where $R^{\ps}_v$ parametrizes all two-dimensional pseudo-deformations of $(1+\bar\chi)|_{G_{F_v}}$, without fixing the determinant. Let $R^{\ps,\ord}_v$ be its reducible quotient, put $R^{\ps,\ord}_p:=\widehat\bigotimes_{v|p}R^{\ps,\ord}_v$, and write $R^{\ps,\ord}:=R^{\ps}\otimes_{R^{\ps}_p}R^{\ps,\ord}_p$ and $C^{\ord}:=C\cap\Spec R^{\ps,\ord}$.
	
	\begin{rem}\label{obstruction}
	Let us compare our strategy with \cite[Section 7.4]{pan2022fontaine}. 
	
	In Pan's setting, the local ordinary ideal is generated by two elements by \cite[Lemma 7.4.4]{pan2022fontaine}. Thus $\dim C\ge1+2d$ implies $\dim C^{\ord}\ge1$ by \cite[Corollary 7.4.5]{pan2022fontaine}. Starting with a one-dimensional prime of $C^{\ord}$, Pan treats the irreducible, generic reducible and non-generic reducible cases separately, and connects $C$ to automorphic points through potentially nice primes.
	
	If $p=3$, Proposition \ref{four elements} gives four local generators instead of two. The same argument only yields $\dim C^{\ord}\ge\dim C-4d$, which does not give $\dim C^{\ord}\ge1$ from our lower bound on $\dim C$. Thus the analogue of \cite[Corollary 7.4.5]{pan2022fontaine} does not follow from this argument, and we cannot begin by choosing an ordinary one-dimensional prime on every $C$.
	
	We instead use two special cases. First, a component of dimension at least $1+2d$ is pro-modular if it contains a one-dimensional ordinary prime for which the ratio of the two local characters is not $\varepsilon^{\pm1}$ at any $v|p$, by Proposition \ref{special case generic}; the associated global representation need not be irreducible. Second, a globally irreducible one-dimensional ordinary prime is pro-modular without this local genericity assumption, by Corollary \ref{special case ordinary}. By Lemma \ref{promodular to nice}, every irreducible component of $\Spec R^{\ps}$ containing such a prime is then pro-modular.
	
	The remaining step is an induction on the set $I\subseteq\Sigma_p$ of places where ordinarity has been imposed. At each stage we impose the four local reducibility equations at one further place and choose a globally irreducible one-dimensional prime of the resulting intersection. For the enlarged set $I$, we retain the ordinary characters already prescribed and choose the new one according to a stable line of the associated representation, as in \cite[Lemma 5.3.2]{pan2022fontaine}. The corresponding deformation problem is ordinary at these places and unrestricted at the other places above $p$. The Greenberg-Wiles dimension estimate, together with Proposition \ref{ps-def comparison}(2), gives a component of the corresponding partial ordinary pseudo-deformation ring through this point of dimension at least $2d-|I|$; see Lemma \ref{partial ordinary dimension}. This recovers enough dimension to repeat the construction: even after the last cut by four equations, the dimension is at least $d-3>2$, whereas the globally reducible locus has dimension at most $2$.
	
	The new component need not lie in the preceding one; they meet at the chosen globally irreducible one-dimensional prime. Iterating, we reach a globally irreducible one-dimensional prime which is ordinary at every place above $p$, so the second special case applies. Lemma \ref{promodular to nice} then propagates pro-modularity backwards along these intersections to the original component $C$. This induction avoids the need for a direct ordinary-intersection estimate on $C$, and the two-generator bound is no longer essential.
	\end{rem}
	
	Combining \cite[Theorem 7.1.1]{pan2022fontaine}, from now on, we assume $\bar{\chi}|_{G_{F_v}} = \omega^{- 1}$ for any $v|p$ and $p=3$. In fact, our method here works for all odd primes.
	
	In the following two subsections, we will first prove the pro-modularity of $C$ (hence the main theorem) under the assumption that $C$ contains a one-dimensional ordinary prime. For any such prime $\kq$, after enlarging $\cO$, we may assume that the normalization of $R^{\ps}/\kq$ is either $\cO$ or isomorphic to $\F[[T]]$.

	\subsection{Special case I: locally generic ordinary points}\label{sec special generic}
	Let $C$ be an irreducible component of $\Spec R^{\ps}$ of dimension at least $1+2d$. Suppose that $C$ contains a one-dimensional prime $\kq\in\Spec R^{\ps,\ord}$ such that, for every $v|p$, $(\tr\rho(\kq))|_{G_{F_v}}=\psi_{v,1}+\psi_{v,2}$ with $\psi_{v,1}/\psi_{v,2}\ne\varepsilon^{\pm1}$. We impose no condition on the reducibility of the global representation $\rho(\kq)$. We will show that $C$ is pro-modular.
	
	For any $v|p$, let $\kq_v^{\ord}\in\Spec R^{\ps,\ord}_v$ and $\kq_v\in\Spec R^{\ps}_v$ be the pull-backs of $\kq$. Both primes have dimension at most one, since the corresponding quotients embed continuously into the normalization of $R^{\ps}/\kq$. They are one-dimensional in characteristic zero. In characteristic $p$, if $\kq_v$ were maximal, the local pseudo-representation would be $1+\omega$, since $\bar\chi|_{G_{F_v}}=\omega$ and $p=3$. By multiplicity-freeness, the two characters would be $1$ and $\omega$, contrary to $\psi_{v,1}/\psi_{v,2}\ne\varepsilon^{\pm1}$. Thus both primes are one-dimensional in this case as well.
	
	\begin{lem}
		There is an irreducible component $C_0^{\ord}$ of $\Spec R^{\ps,\ord}$ of dimension $d+1$ such that $\kq\in C_0^{\ord}\subseteq C$. In particular, $\dim C^{\ord}\ge d+1$.
	\end{lem}
	
	\begin{proof}
		By Proposition \ref{principal}, for every $v|p$, the kernel of $(R^{\ps}_v)_{\kq_v}\to(R^{\ps,\ord}_v)_{\kq_v^{\ord}}$ is principal. Hence the kernel of $(R^{\ps})_\kq\to(R^{\ps,\ord})_\kq$ is generated by at most $d$ elements. Since $\dim C\ge1+2d$, catenarity and Krull's principal ideal theorem give $\dim(C^{\ord})_\kq\ge d$. Thus some irreducible component of $C^{\ord}$ through $\kq$ has dimension at least $d+1$. By Theorem \ref{ordinary case}, $\dim R^{\ps,\ord}\le d+1$, so this is an irreducible component of $\Spec R^{\ps,\ord}$ of dimension exactly $d+1$.
	\end{proof}
	
	\begin{prop}\label{special case generic}
		Let $C\subseteq\Spec R^{\ps}$ be an irreducible component of dimension at least $1+2d$. If $C$ contains a one-dimensional ordinary prime $\kq$ whose two local characters have ratio different from $\varepsilon^{\pm1}$ for every $v|p$, then $C$ is pro-modular.
	\end{prop}
	
	\begin{proof}
		By the previous lemma and Lemma \ref{find nice}, $C$ contains a nice prime. Choose tame characters for $C$ as in Remark \ref{belong}. Theorem \ref{R=T} then shows that $C$ is pro-modular, with respect to the common tame level fixed in Subsection 5.1.
	\end{proof}
	
	\begin{rem}\label{generic red rem}
		The proposition applies in both characteristics and does not require $\rho(\kq)$ to be reducible. Only ordinarity and local genericity at the places above $p$ are used. This formulation will be used in the next subsection.
	\end{rem}
	
	\subsection{Special case II: irreducible ordinary points}\label{sec special ordinary}
	In this subsection, let $\kq$ be any one-dimensional prime of $R^{\ps,\ord}$ such that $\rho(\kq)$ is irreducible. We will prove that $\kq$ is pro-modular, without imposing local genericity or restricting its residue characteristic. Lemma \ref{promodular to nice} then applies to every irreducible component of $\Spec R^{\ps}$ containing $\kq$. The result is summarized in Corollary \ref{special case ordinary}.
	
	Let $\psi^{univ}_{v,1},\psi^{univ}_{v,2}$ be the local characters over $R^{\ps,\ord}$ lifting $\mathbf{1}$ and $\bar\chi|_{G_{F_v}}$, respectively. Choose a $G_{F_v}$-stable line in $\rho(\kq)$ for every $v|p$, and let $\Sigma^o$ consist of the places where its character is $\psi^{univ}_{v,1}\mod\kq$. If the local representation is split, either stable line may be chosen. By Lemma \ref{ordpseudo}, $\kq$ is a prime of the prescribed ordinary quotient $R^{\ps,\ord}_{\Sigma^o}$ defined in Subsection \ref{ord section}.
	
	Choose a Galois-stable lattice $\rho(\kq)^o$ with non-split residual representation $\bar\rho_b$ which is an extension of $\bar\chi$ by $\mathbf{1}$. Intersecting the chosen ordinary lines with this lattice gives stable rank one direct summands, so $b\in H^1_{\Sigma^o}(F)$. Let $R_b^{\ord}$ be the corresponding deformation ring with determinant $\chi$ and these prescribed ordinary characters, and let $\kq_b$ be the prime defined by $\rho(\kq)^o$. By Proposition \ref{ps-def comparison} and the discussion in \cite[Section 5.3.5(2)]{pan2022fontaine}, as used in the proof of \cite[Lemma 7.4.7]{pan2022fontaine}, we have an isomorphism
	\[\widehat{(R^{\ps,\ord}_{\Sigma^o})_\kq}\cong\widehat{(R_b^{\ord})_{\kq_b}}.\]
	This comparison uses the prescribed ordinary characters and applies in both characteristics.
	
	\begin{lem}\label{ordinary large component}
		If $\dim(R^{\ps,\ord}_{\Sigma^o})_\kq\ge d$, then $\kq$ lies on an irreducible component of $\Spec R^{\ps,\ord}$ of dimension $d+1$, and is pro-modular.
	\end{lem}
	
	\begin{proof}
		The dimension bound gives an irreducible component of $\Spec R^{\ps,\ord}_{\Sigma^o}$ through $\kq$ of dimension at least $d+1$. By Theorem \ref{ordinary case}, this is also an irreducible component of $\Spec R^{\ps,\ord}$, of dimension exactly $d+1$. Lemma \ref{find nice}(1) gives its pro-modularity.
	\end{proof}
	
	We first treat the case $\bar\chi\ne\omega$. Since $p=3$, this is equivalent to $\bar\chi\ne\omega^{-1}$.
	
	\begin{prop}\label{ordinary nonexceptional}
		If $\bar\chi\ne\omega$, then $\kq$ is pro-modular.
	\end{prop}
	
	\begin{proof}
		For the chosen extension of $\bar\chi$ by $\mathbf{1}$, the condition $\bar\chi\ne\omega$ gives $H^0(G_{F,\Sigma},\ad^0\bar\rho_b(1))=0$, as in Remark \ref{dim not enough}. By Proposition \ref{dim for global deformation ring} and the completed local comparison above, $\dim(R^{\ps,\ord}_{\Sigma^o})_\kq\ge d$. The result follows from Lemma \ref{ordinary large component}.
	\end{proof}
	
	It remains to treat $\bar\chi=\omega=\omega^{-1}$. In this case the general lower bound in Proposition \ref{dim for global deformation ring} is only $d$, so the preceding argument does not immediately apply. See Remark \ref{dim not enough}.
	
	By the same comparison, $\dim(R^{\ps,\ord}_{\Sigma^o})_\kq\ge d-1$. We may assume that an irreducible component $C_1^{\ord}$ of $\Spec R^{\ps,\ord}_{\Sigma^o}$ through $\kq$ has dimension $d$; if one has dimension $d+1$, Lemma \ref{ordinary large component} applies. Let $\kq_1$ be its generic point and put $\mathfrak a=\kq_1\cap\Lambda_F$, using the map $\Lambda_F\to R^{\ps,\ord}\twoheadrightarrow R^{\ps,\ord}_{\Sigma^o}$. Theorem \ref{ordinary case} gives a finite injection $\Lambda_F/\mathfrak a\hookrightarrow R^{\ps,\ord}_{\Sigma^o}/\kq_1$. Hence $\mathfrak a$ has height one, and $C_1^{\ord}\to\Spec(\Lambda_F/\mathfrak a)$ is finite and surjective.
	
	Write $\Lambda_F=\cO[[X_1,\ldots,X_d]]$, where $1+X_i$ corresponds to a generator of $O_{F_{v_i}}^\times(p)$. If the ratio of the two local characters at $v_i$ is $\varepsilon^{\pm1}$, then $(\psi^{univ}_{v_i,1})^2=\chi\varepsilon^{\pm1}$ at that point. On the pro-$p$ inertia group these two equations determine a finite set $A_i=\{a_{i,j}\}_j\subset\mathfrak m_{\bar K}$ of possible values of $X_i$. Enlarge $\cO$ so that $A_i\subset\mathfrak m_K$, and put $\kp_{i,j}=(X_i-a_{i,j})\subset\Lambda_F$. These are height one primes, and the non-generic locus at $v_i$ maps into their union, also in characteristic $p$.
	
	\begin{prop}\label{generic ordinary specialization}
		If $\mathfrak a\ne\kp_{i,j}$ for every $i,j$, then $C_1^{\ord}$ contains a globally irreducible one-dimensional prime $\kq_1''$ which is locally generic at every $v|p$.
	\end{prop}
	
	\begin{proof}
		Since $\mathfrak a$ and $\kp_{i,j}$ are distinct height one primes, $V(\mathfrak a)$ is not contained in any $V(\kp_{i,j})$. By the finite surjectivity above, their inverse images are proper closed subsets of $C_1^{\ord}$. The globally reducible locus is also a proper closed subset, since it has dimension at most $2<d$. Lemma \ref{Jacobon} gives a one-dimensional prime outside their union. By construction, it is globally irreducible and locally generic at every $v|p$.
	\end{proof}
	
	If $\mathfrak a\ne\kp_{i,j}$ for all $i,j$, choose $\kq_1''$ as in the previous proposition. By Proposition \ref{connect dim}, some irreducible component $C'$ of $\Spec R^{\ps}$ through $\kq_1''$ has dimension at least $1+2d$. Proposition \ref{special case generic} shows that $C'$, and hence $\kq_1''$, is pro-modular. Lemma \ref{promodular to nice} then applies to every ambient irreducible component containing $\kq_1''$, including any one containing $C_1^{\ord}$. Thus $C_1^{\ord}$, and in particular the original prime $\kq$, is pro-modular.
	
	It remains to consider $\mathfrak a=(X_i-a_{i,j})$ for some $i,j$.
	
	Let $k$ be the common Hodge-Tate weight of $\chi$. Choose $\kq_2=(\mathfrak a,X_t-b_t\mid t\ne i)\subset\Lambda_F$, where each $b_t\in\mathfrak m_K$ corresponds to a locally algebraic character of weight $h_t\in\mathbb Z$ satisfying $2h_t-k\notin\{1,-1\}$. Then $\Lambda_F/\kq_2\cong\cO$, so $\kq_2$ is one-dimensional and $\varpi\notin\kq_2$. By the finite surjectivity above, choose a one-dimensional prime $\kq_2'$ of $C_1^{\ord}$ above $\kq_2$. If $\rho(\kq_2')$ were reducible, Leopoldt's conjecture would force the Sen weights of each global character to be parallel. This is impossible: the two local weights at $v_i$ differ by $1$, whereas their difference at $v_t$, for $t\ne i$, is $\pm(2h_t-k)\ne\pm1$. Thus $\rho(\kq_2')$ is irreducible. Since $\kq_2'$ has characteristic zero, $H^0(G_{F,\Sigma},\ad^0\rho(\kq_2')(1))=0$: a non-zero map $\rho(\kq_2')\to\rho(\kq_2')(1)$ would be an isomorphism and would force $\varepsilon^2=1$ by taking determinants. The point $\kq_2'$ need not be regular de Rham; below we only use its ordinarity and global irreducibility in characteristic zero.
	
	\begin{prop}\label{ordinary exceptional}
		If $\bar\chi=\omega=\omega^{-1}$, then $\kq$ is pro-modular.
	\end{prop}
	
	\begin{proof}
		By the preceding arguments, we may assume $\mathfrak a=(X_i-a_{i,j})$ and use the point $\kq_2'$ constructed above. Choose a lattice of $\rho(\kq_2')$ whose residual representation $\bar\rho_{b'}$ is a non-split extension of $\bar\chi$ by $\mathbf{1}$, and use the ordinary lines prescribed by $\Sigma^o$. Let $\kq_{b'}$ be the corresponding prime of $R_{b'}^{\ord}$. The same completed local comparison gives $\widehat{(R^{\ps,\ord}_{\Sigma^o})_{\kq_2'}}\cong\widehat{(R_{b'}^{\ord})_{\kq_{b'}}}$. By Corollary \ref{crucial} and the characteristic-zero vanishing just proved, the latter ring has dimension at least $d$. Lemma \ref{ordinary large component} shows that $\kq_2'$ is pro-modular. By Lemma \ref{promodular to nice}, every ambient irreducible component through $\kq_2'$ is pro-modular. Taking one containing $C_1^{\ord}$ proves the pro-modularity of $\kq_1$ and hence of the original prime $\kq$.
	\end{proof}
	
	\begin{cor}\label{special case ordinary}
		Let $\kq\in\Spec R^{\ps,\ord}$ be a one-dimensional prime such that $\rho(\kq)$ is irreducible. Then $\kq$ is pro-modular, and every irreducible component of $\Spec R^{\ps}$ containing $\kq$ is pro-modular. No restriction on the residue characteristic of $\kq$ or local genericity at $p$ is required.
	\end{cor}
	
	\begin{proof}
		The pro-modularity of $\kq$ follows from Propositions \ref{ordinary nonexceptional} and \ref{ordinary exceptional}. The assertion about irreducible components follows from Lemma \ref{promodular to nice}.
	\end{proof}
	
	\begin{rem}\label{ordinary R=T}
		Another possible approach to Corollary \ref{special case ordinary} is to use an ordinary big $R=\mathbb T$ theorem. We briefly explain the extra issues this would entail.
		
		Under the hypotheses of \cite[Corollary 5.5.2]{pan2022fontaine}, including a suitable lower bound for $[F:\Q]$ in terms of $\dim H^1_{\Sigma^o}(F)$ and the ramification data, every irreducible prime of the prescribed ordinary pseudo-deformation ring is pro-modular. This uses the ordinary $R=\mathbb T$ theorem of \cite[Proposition 5.3.3]{pan2022fontaine}. To apply this argument here, one would need to replace the local hypothesis $[F_v:\Q_p]\ge2$ used in \cite[Lemma 5.3.4]{pan2022fontaine} by the corresponding input from Proposition \ref{framed loc ord}.
		One would then need an abelian base change $F_1/F$ for which the degree is sufficiently large relative to the relevant Selmer dimensions, while controlling the class-group contributions and prescribing the behavior at each place in $\Sigma$.
		
		In the ordinary case, we can use \cite[Theorem]{Washington_1978} to find a suitable base change. However, in the non-ordinary case, we need the assumption that $p$ splits completely in $F$ (to use the $p$-adic Langlands correspondence), and it is difficult to use \cite[Theorem]{Washington_1978} in our case.
		
		In \cite[Theorem B]{Skinner_1999}, the authors give an assumption on the $p$-rank of the relative class group of $L/F$ instead of \cite{Washington_1978}, where $L$ is some CM extension of $F$. See \cite[Hypothesis H]{Skinner_1999}. This hypothesis has been studied for imaginary quadratic fields in \cite{wiles2015}. Analogous results for general abelian CM fields could help with this alternative approach.
		
		The proof above avoids these additional base-change and class-group requirements.
		\end{rem}

	\subsection{Induction on the ordinary places}\label{sec partial ordinary}\label{sec non-generic}
	We now complete the argument by imposing ordinarity one place at a time. After cutting by the four local reducibility equations at a new place, we choose a globally irreducible one-dimensional prime and prescribe an ordinary line there. A deformation-theoretic dimension estimate then gives a sufficiently large partial ordinary component through this prime, retaining all the ordinary characters already prescribed. The new component need not lie in the preceding one. Once all places above $p$ are ordinary, Corollary \ref{special case ordinary} applies, and Lemma \ref{promodular to nice} propagates pro-modularity backwards to the original component.
	
	For $v|p$, write $J_v=\ker(R_v^{\ps}\to R_v^{\ps,\ord})$, and also write $J_v$ for its image ideal in $R^{\ps}$. Given $I\subseteq\Sigma_p$, put $R_I^{\ps}=R^{\ps}/(J_v\mid v\in I)$. For each $v\in I$, the universal local pseudo-representation over $R_I^{\ps}$ is the sum of two characters $\psi_{v,1},\psi_{v,2}$ lifting $\mathbf{1}$ and $\bar\chi|_{G_{F_v}}$, respectively. These characters are uniquely determined since $\bar\chi|_{G_{F_v}}=\omega\ne\mathbf{1}$.
	
	For $I^o\subseteq I$, let $R_{I,I^o}^{\ps,\ord}$ be the quotient of $R_I^{\ps}$ imposing $\psi_{v,1}$-ordinarity for $v\in I^o$ and $\psi_{v,2}$-ordinarity for $v\in I\setminus I^o$, in the sense of Definition \ref{psiord}. Thus we prescribe the character on the ordinary stable line, not merely the reducibility of the local pseudo-representation; the lifting character is allowed to vary. No additional local condition is imposed at places of $\Sigma\setminus I$. For $I=\varnothing$ this is $R^{\ps}$, and for $I=\Sigma_p$ this is the prescribed ordinary ring $R^{\ps,\ord}_{I^o}$ of Subsection \ref{ord section}.
	
	\begin{lem}\label{partial ordinary dimension}
		Let $I^o\subseteq I\subseteq\Sigma_p$, and let $\kq$ be a one-dimensional prime of $R_{I,I^o}^{\ps,\ord}$ such that $\rho(\kq)$ is irreducible. Then $\kq$ lies on an irreducible component of $\Spec R_{I,I^o}^{\ps,\ord}$ of dimension at least $2d-|I|$. This holds in both residue characteristics.
	\end{lem}
	
	\begin{proof}
		After enlarging $\cO$, choose a Galois-stable lattice of $\rho(\kq)$ whose residual representation $\bar\rho_b$ is a non-split extension of $\bar\chi$ by $\mathbf{1}$. By Lemma \ref{ordpseudo}, $\rho(\kq)$ has the prescribed ordinary stable line at every $v\in I$. Intersecting these lines with the lattice gives stable rank one direct summands. Let $R_{b,I,I^o}^{\ord}$ be the universal deformation ring of $\bar\rho_b$ with determinant $\chi$, ordinary along these prescribed residual lines at $v\in I$ and unrestricted at the remaining places of $\Sigma$. The lattice defines a one-dimensional prime $\kq_b$ of this ring above $\kq$.
		
		Put $r=|I|$ and $h^0=\dim_\F H^0(G_{F,\Sigma},\ad^0\bar\rho_b(1))$. The local framed rings have dimensions $6$ at $v\in I$ and $7$ at $v\in\Sigma_p\setminus I$, by Proposition \ref{framed loc ord} and Proposition \ref{dim of loc framed der ring}. For the first assertion, if the prescribed residual line has character $\mathbf{1}$, we twist locally by the Teichm\"uller lift of $\omega$ and change basis; since $p=3$, this does not change the determinant. At $v\in\Sigma\setminus\Sigma_p$, the unrestricted local framed ring has dimension $4$ by \cite[Theorem 3.31]{Gee_2022}.
		
		Since $\bar\rho_b$ has only scalar endomorphisms, the Greenberg-Wiles calculation underlying \cite[Proposition 3.24]{Gee_2022} applies; compare \cite[Section 5.4.2]{pan2022fontaine}. More precisely, with framings at $\Sigma$, the global ring has a presentation over the completed tensor product of these local rings with the number of variables minus the number of relations at least $|\Sigma|-1-d-h^0$. Removing the $4|\Sigma|-1$ framing variables and using the local presentations gives, for every minimal prime $\mathfrak Q$ of $R_{b,I,I^o}^{\ord}$,
		\[\dim R_{b,I,I^o}^{\ord}/\mathfrak Q\ge 1+2r+3(d-r)-d-h^0=1+2d-r-h^0\ge2d-r.\]
		Here $h^0\le1$, since $\bar\rho_b$ is a non-split extension of two distinct characters. In particular, no vanishing of this group is assumed.
		
		Choose such a minimal prime $\mathfrak Q\subseteq\kq_b$. The representation at $\mathfrak Q$ is irreducible, since the reducible locus is closed and does not contain $\kq_b$. Apply Proposition \ref{ps-def comparison}(2) to the inverse image of $\mathfrak Q$ in the unrestricted deformation ring $R_b$. Writing $\mathfrak P$ for its pull-back to $R_{I,I^o}^{\ps,\ord}$, we obtain $\dim R_{I,I^o}^{\ps,\ord}/\mathfrak P\ge\dim R_{b,I,I^o}^{\ord}/\mathfrak Q\ge2d-r$. Since $\mathfrak P\subseteq\kq$, any irreducible component containing $V(\mathfrak P)$ gives the assertion.
	\end{proof}
	
	\begin{prop}\label{induction pro-modularity}
		Every irreducible component $C$ of $\Spec R^{\ps}$ of dimension at least $1+2d$ is pro-modular.
	\end{prop}
	
	\begin{proof}
		Order the places above $p$ as $v_1,\ldots,v_d$, and put $I_j=\{v_1,\ldots,v_j\}$. We construct prescribed subsets $I_j^o\subseteq I_j$, globally irreducible one-dimensional primes $\kq_j$, and, for $j<d$, irreducible components $C_j$ of $\Spec R_{I_j,I_j^o}^{\ps,\ord}$ through $\kq_j$ with $\dim C_j\ge2d-j$. Start with $C_0=C$ and $I_0=I_0^o=\varnothing$.
		
		Suppose $C_{j-1}$ has been chosen, where $1\le j\le d$. By Proposition \ref{four elements}, $J_{v_j}$ is generated by at most four elements. Its vanishing locus $Z_j$ in $C_{j-1}$ therefore satisfies $\dim Z_j\ge\dim C_{j-1}-4\ge2d-j-3\ge d-3>2$. Since the globally reducible locus has dimension at most $2$, Lemma \ref{Jacobon} gives a globally irreducible one-dimensional prime $\kq_j\in Z_j$.
		
		At the places in $I_{j-1}$, the representation $\rho(\kq_j)$ has the ordinary lines already prescribed. At $v_j$, its local pseudo-representation is reducible, so the local representation has a stable line. Choose one, and extend $I_{j-1}^o$ to $I_j^o$ according to whether its character lifts $\mathbf{1}$ or $\bar\chi|_{G_{F_{v_j}}}$. By Lemma \ref{ordpseudo}, $\kq_j$ is a prime of $R_{I_j,I_j^o}^{\ps,\ord}$. If $j<d$, Lemma \ref{partial ordinary dimension} supplies the required component $C_j$ through $\kq_j$. This completes the construction.
		
		The final prime $\kq_d$ is ordinary at every place above $p$, and hence is pro-modular by Corollary \ref{special case ordinary}. We now descend from $j=d$ to $j=1$. If $\kq_j$ is pro-modular, Lemma \ref{promodular to nice} shows that every irreducible component of $\Spec R^{\ps}$ through it is pro-modular. Choose one containing $C_{j-1}$, which contains $\kq_j$ by construction. Thus all primes of $C_{j-1}$ are pro-modular, including $\kq_{j-1}$ when $j>1$. For $j=1$ this proves that $C_0=C$ is pro-modular.
	\end{proof}
	
	Applying this proposition to the component chosen in Subsection 5.1 and using Corollary \ref{classicality} completes the proof of Theorem \ref{mainthm} in the exceptional case $p=3$.
	
	  \subsection{Potential pro-modularity}\label{sec potential promodularity}
	 
	 We now combine the componentwise pro-modularity results proved above with \cite{Zhang2024}. Abelian base change arranges the auxiliary hypotheses imposed at the beginning of this section. The previously treated local types are covered by \cite{Zhang2024}, after a finite twist if necessary, while the exceptional type at $p=3$ is covered by Subsection \ref{sec partial ordinary}.
	 
	 Let $F$ be a totally real abelian field in which $p$ splits completely. Let $\Sigma_p$ be the set of places of $F$ above $p$, and let $\Sigma$ be a finite set of finite places containing $\Sigma_p$. Let $\bar{\chi}:G_{F,\Sigma}\to\F^\times$ be a continuous totally odd character which extends to $G_\Q$, and let $\chi:G_{F,\Sigma}\to\cO^\times$ be a lift of $\bar\chi$ which is de Rham at every place above $p$.

	 \begin{thm}\label{potential R=T}
	 	Assume that $\bar{\chi}|_{G_{F_v}}\neq\mathbf{1}$ for every $v|p$. Then there exists a totally real abelian extension $F_1/\Q$ containing $F$, with $[F_1:F]$ even, in which $p$ splits completely, with the following property. Let $R^{\ps}_{F_1}$ be the universal pseudo-deformation ring parametrizing pseudo-deformations of $(1+\bar\chi)|_{G_{F_1}}$, unramified outside the places of $F_1$ above $\Sigma$, with fixed determinant $\chi|_{G_{F_1}}$. Then every irreducible component of $R^{\ps}_{F_1}$ of dimension at least $1+2[F_1:\Q]$ is pro-modular and has dimension $1+2[F_1:\Q]$.
	 \end{thm}
	 
	 \begin{remark}\label{rem:twisting_determinant_potential}
	 	For the general determinant in Theorem \ref{potential R=T}, Leopoldt's conjecture gives $\chi=\psi\varepsilon^k$ with $\psi$ of finite order, as in Subsection 5.1. Since $\bar\psi=\bar\chi\omega^{-k}$ extends to $G_\Q$, the character $\chi_0=\widetilde{\bar\psi}\varepsilon^k$ also extends to $G_\Q$. The quotient $\chi_0\chi^{-1}$ has finite $p$-power order, so, as $p$ is odd, it has a square root $\varphi$ of finite $p$-power order, unramified outside $\Sigma$. Twisting by $\varphi$ identifies the fixed-determinant pseudo-deformation rings for $\chi$ and $\chi_0$, preserving components and their dimensions. The automorphic twist preserves pro-modularity, after shrinking the tame level if necessary; see \cite[Section 3.2.2]{pan2022fontaine}. Thus it suffices to prove the theorem when $\chi$ extends to $G_\Q$.
	 \end{remark}
	 
	 %Indeed, twisting by a finite character $\varphi$ sends a pseudo-representation $T$ of determinant $\chi$ to $\varphi T$ of determinant $\varphi^2\chi$, and hence identifies the corresponding fixed-determinant pseudo-deformation rings.
	 
	 \begin{proof}
	 	By Remark \ref{rem:twisting_determinant_potential}, we may impose the additional assumption made there. The same application of the Grunwald--Wang theorem used to deduce the potential pro-modularity theorem in \cite[Theorem 1.0.3]{Zhang2024} allows us to choose $F_1$ as follows. Write $\Sigma_1$ for the set of places of $F_1$ above $\Sigma$, and $\Sigma_{1,p}$ for the set of places above $p$. In addition to the properties in the statement, we may require that
	 	\begin{itemize}
	 		\item $p\mid\operatorname{Nm}(w)-1$ for every $w\in\Sigma_1\setminus\Sigma_{1,p}$;
	 		\item $[F_1:\Q]>7|\Sigma_1\setminus\Sigma_{1,p}|+6$;
	 		\item $\chi|_{G_{F_1}}$ is unramified outside $\Sigma_{1,p}$ and $\chi(\operatorname{Frob}_w)\equiv1\pmod\varpi$ for every $w\in\Sigma_1\setminus\Sigma_{1,p}$.
	 	\end{itemize}
	 	For completeness, the degree bound can be arranged together with the local conditions. Let $m$ be the number of rational primes below $\Sigma\setminus\Sigma_p$. At each such prime $\ell$, prescribe a finite abelian local extension containing the completions of $F$, killing the ramification of $\chi$, and arranging the first and third conditions above. Require it also to contain an unramified extension of degree at least $M$, where $M>14m$. Apply \cite[Theorem 5, p.~80]{artin} to the corresponding local characters, taking them trivial at $p$ and at infinity, and then form the compositum with $F$. Additional totally real abelian extensions split at $p$ arrange $[F_1:F]$ even and $D:=[F_1:\Q]>12$. Since every local degree at the $m$ prescribed primes is at least $M$, we have $|\Sigma_1\setminus\Sigma_{1,p}|\le mD/M<D/14$, so $D>7|\Sigma_1\setminus\Sigma_{1,p}|+6$. If $m=0$, only the degree and splitting conditions need to be arranged.
	 	
	 	Since $p$ splits completely in $F_1$ and $\bar\chi$ extends to $G_\Q$, the local residual character $\bar\chi|_{G_{F_{1,w}}}$ has the same type for every $w|p$. If it is different from both $\mathbf{1}$ and $\omega^{-1}$, the result follows from \cite[Theorem 1.0.2(1)]{Zhang2024}. If it equals $\omega^{-1}$ and $p\ge5$, twist by the Teichm\"uller lift $\eta$ of $\bar\chi^{-1}|_{G_{F_1}}$. This replaces $1+\bar\chi$ by $1+\bar\chi^{-1}$ and $\chi$ by $\eta^2\chi$. The new local residual ratio is $\omega\ne\mathbf{1},\omega^{-1}$, and all the other hypotheses of \cite[Theorem 1.0.2(1)]{Zhang2024} are unchanged. Applying that theorem and twisting back proves the assertion. It remains to treat $p=3$ with local residual ratio $\omega^{-1}$; here Proposition \ref{induction pro-modularity} gives pro-modularity.
	 	
	 	It remains only to record the dimension assertion in the exceptional case. Let $C$ be one of the above pro-modular components. Since $F_1/\Q$ is abelian, Leopoldt's conjecture holds for $F_1$, and the reducible locus has dimension at most $2$. As $\dim C\geq1+2[F_1:\Q]$, Lemma \ref{Jacobon}, applied to $C$ and its reducible locus, gives a one-dimensional irreducible prime on $C$. This prime is pro-modular, so Lemma \ref{promodular to nice} shows that $C$ contains a nice prime. The dimension calculation in \textbf{Step 4} of the proof of Theorem \ref{R=T} then gives $\dim C=1+2[F_1:\Q]$.
	 \end{proof}
	 
	 Fix an extension $F_1/F$ as in Theorem \ref{potential R=T}. Use the common tame level of Subsection 5.1, shrinking it at places in $\Sigma_1\setminus\Sigma_{1,p}$ if necessary to include the finite twists above. Let $\T_\psi$ be the resulting big Hecke algebra without the auxiliary character $\xi$, where $\psi=\chi\varepsilon$, and let $\km$ be a maximal ideal corresponding to $1+\bar\chi|_{G_{F_1}}$. Thus the following assertion concerns the non-zero localization $\T_\km:=(\T_\psi)_\km$, with its natural surjection $R_{F_1}^{\ps}\twoheadrightarrow\T_\km$. Such a maximal ideal is supplied, in particular, whenever $R_{F_1}^{\ps}$ has a component of dimension at least $1+2[F_1:\Q]$, by the theorem.
	 
	 \begin{prop}\label{equi of T}
	 	The big Hecke algebra $\T_\km$ is equidimensional of dimension $1+2[F_1:\Q]$.
	 \end{prop}
	 
	 \begin{proof}
	 	By Corollary \ref{dim of hecke}, each irreducible component of $\Spec\T_\km$ has dimension at least $1+2[F_1:\Q]$. Theorem \ref{potential R=T} rules out any irreducible component of $R_{F_1}^{\ps}$ of larger dimension, so $\dim R_{F_1}^{\ps}\le1+2[F_1:\Q]$. Since $\T_\km$ is a quotient of $R_{F_1}^{\ps}$, every component of $\Spec\T_\km$ has exactly this dimension.
	 	
	 	%Let $\mathfrak P$ be a minimal prime of $\T_\km$, and let $\mathfrak p$ be its inverse image in $R_{F_1}^\ps$. Choose a minimal prime $\mathfrak q$ of $R_{F_1}^\ps$ contained in $\mathfrak p$. By Corollary \ref{dim of hecke}, $\T_\km/\mathfrak P$ is of characteristic zero and has dimension at least $1+2[F_1:\Q]$. Since it is a quotient of $R_{F_1}^\ps/\mathfrak q$, Theorem \ref{potential R=T} applies to the component defined by $\mathfrak q$ and gives $\dim R_{F_1}^\ps/\mathfrak q=1+2[F_1:\Q]$. Hence $\dim\T_\km/\mathfrak P=1+2[F_1:\Q]$, and the assertion follows.
	 \end{proof}

	\section{The proof in the residually irreducible case}\label{section 6} 
	In this section, we prove the non-ordinary Fontaine-Mazur conjecture in the residually irreducible case. 
	
	Let $p$ be an odd prime. Throughout this section, we denote by $K$  a $p$-adic local field with its ring of integers $\cO$, a uniformizer $\varpi$ and its residue field $\F$. 
	
	\subsection{The proof}
	
	We prove the following theorem.
	
	\begin{thm}\label{mainthm2}
		Let $F$ be a totally real extension of $\Q$ in which $p$ splits completely. Suppose 
		\[\rho:\Gal(\overbar{F}/F)\to \GL_2(\cO)\] 
		is a continuous irreducible representation with the following properties
		\begin{enumerate}
			\item $\rho$ ramifies at only finitely many places.
			\item $\bar\rho$ is absolutely irreducible, where $\bar{\rho}$ denotes the reduction of $\rho$ modulo $\varpi$. 
			\item For any $v|p$, $\rho|_{G_{F_v}}$ is absolutely irreducible and de Rham of distinct Hodge-Tate weights. 
			\item $\det \rho(c)=-1$ for any complex conjugation $c\in \Gal(\overbar{F}/F)$.
			\item $\bar\rho$ arises from a regular algebraic cuspidal automorphic representation $\pi_0$ of $\GL_2(\A_F)$. 
		\end{enumerate}
		Then $\rho$ arises from a regular algebraic cuspidal automorphic representation of $\GL_2(\A_F)$.
	\end{thm}
	
	Compared to \cite[Theorem 9.0.1]{pan2022fontaine}, our theorem removes the assumption that $\bar{\rho}|_{G_{F_v}}$ is not of the form $\begin{pmatrix} \eta & * \\ 0 & \eta\omega\end{pmatrix}$ or $\begin{pmatrix} \eta\omega & * \\ 0 & \eta\end{pmatrix}$ for any $v|p$ if $p=3$.
	
	\begin{rem}
		In the ordinary case, the result is known by \cite{skinnerwiles01}, \cite{kisin2009fontaine} and \cite{hutan15}. See \cite[Remark 9.0.2]{pan2022fontaine}.
	\end{rem}
	
	Let $\Sigma$ contain all places above $p$ and all finite places
	at which either $\rho$ or $\pi_0$ is ramified.
	Using \cite[Fact 4.23 \& 4.27]{Gee_2022}, we make a soluble totally
	real Galois base change, chosen to be linearly disjoint from
	the finite Galois extension cut out by $\bar\rho$.
	Thus $\bar\rho$ remains absolutely irreducible.
	Keeping the same notation for the base-changed data, we may
	arrange that:
	\begin{itemize}
		\item $p$ splits completely in $F$;
		\item $\rho|_{I_{F_v}}$ is unipotent for every $v\nmid p$;
		\item for every $v\in\Sigma\setminus\Sigma_p$,
		$\mathrm{Nm}(v)\equiv 1\pmod p$ and
		$\bar\rho|_{G_{F_v}}$ is trivial;
		\item $[F:\mathbb Q]$ is even and
		$[F:\mathbb Q]\ge 4|\Sigma\setminus\Sigma_p|+2$;
		\item $\pi_0$ is unramified outside $\Sigma$ and
		$\pi_{0,v}^{\mathrm{Iw}_v}\ne 0$ for every
		$v\in\Sigma\setminus\Sigma_p$.
	\end{itemize}
	%Let $\Sigma$ be a finite set of places containing all the ramified finite places of $\rho$ and let $\Sigma_p$ be the set of primes above $p$. 
	%By \cite[Fact 4.27]{Gee_2022} and soluble base change, we may assume the following conditions:
	%\begin{itemize}
		%\item $p$ splits completely in $F$.
		%\item $\rho|_{I_{F_v}}$ is unipotent for any $v\nmid p$. 
		%\item If $\rho|_{G_{F_v}}$ is ramified and $v\nmid p$, then $N(v)\equiv1\mod p$ and $\bar\rho|_{G_{F_v}}$ is trivial.
		%\item $[F:\Q]$ is even and $[F:\Q] \ge 4|\Sigma \setminus \Sigma_p|+2$.
		%\item $\pi_0$ is unramified everywhere except places above $p$.
	%\end{itemize}
	
	\begin{proof}[Proof of Theorem \ref{mainthm2}]
		We keep notations in Section \ref{irr sec}. Under our assumptions above, by Proposition \ref{big R=T in irr}, the surjection $R^{\{\mathbf{1}\}}\to\T^{\{\mathbf{1}\}}_{\km_{\mathbf{1}}}$ has nilpotent kernel. Note that $\rho$ defines a maximal ideal $\kp$ of $R^{\{\mathbf{1}\}}[\frac{1}{p}]$, and hence we can also view $\kp$ as a maximal ideal of $ \T^{\{\mathbf{1}\}}_{\km_{\mathbf{1}}} [\frac{1}{p}]$. Then the theorem follows from Corollary \ref{classicality}.
	\end{proof}

	\appendix
    \section{Pašk\=unas theory}\label{A}
    This appendix collects the results of Paškūnas and Paškūnas-Tung used in Section~\ref{section lgc}. We state them in the precise form required for the local-global compatibility argument; complete proofs may be found in \cite{Pa_k_nas_2013} and \cite{Pa_k_nas_2021}.
	
	\subsection{Banach space representations} \label{A.1}
	In this subsection, we recall some basic theory about Banach space representations, and one can find more details and proofs in \cite[Section 4]{Pa_k_nas_2013}. Throughout, we let $p$ be a prime and $G$ be a $p$-adic analytic group. Let $K$ be a $p$-adic local field with its ring of integers $\mathcal{O}$, a uniformizer $\varpi$ and the residue field $\mathbb{F}$. 
	
	A $K$-Banach space representation $\Pi$ of $G$ is a $K$-Banach space $\Pi$ together with a $G$-action by continuous linear automorphisms. A Banach space representation $\Pi$ is \textit{unitary}, if there exists a $G$-invariant norm defining the topology on $\Pi$. The existence of such norm is equivalent to the existence of 
	an open  bounded $G$-invariant $\mathcal{O}$-lattice $\Theta$ in $\Pi$. A unitary $K$-Banach space representation 
	is \textit{admissible} if the space of invariants  $(\Theta\otimes_{\mathcal{O}} \mathbb{F})^H$ is finite-dimensional for every open subgroup $H$ of $G$. We say that a $K$-Banach space representation $\Pi$ is \textit{irreducible}, if it does not contain 
	a proper closed $G$-invariant subspace.
	
	For a unitary admissible absolutely irreducible $K$-Banach space representation $\Pi$ of $G$, we say that $\Pi$ is \textit{ordinary} if it is a subquotient of a parabolic induction of a unitary character, so that
	$\Pi$ is either a unitary character $\Pi\cong \eta\circ \det$, 
	a twist of the universal unitary completion of the smooth Steinberg representation by a unitary character 
	$\Pi\cong \widehat{\Sp}\otimes \eta\circ \det$, , or $\Pi$ is a unitary parabolic induction of  a unitary character. Here $\Sp$ is defined by the exact sequence
	\[
	0 \rightarrow \mathbf{1} \rightarrow (\Ind_B^G \mathbf{1})_{\sm} \rightarrow \Sp \rightarrow 0,
	\]
	and $\widehat{\Sp} $ is its universal unitary completion.
	
    \begin{lem}\label{4.1}
    	Let $\Pi$ be an absolutely irreducible and admissible unitary $K$-Banach space representation of $G$.
    	
    	1) Let $\phi\in \End_{K[G]}^{\operatorname{cont}}(\Pi)$. 
    	If the algebra $K[\phi]$ is finite dimensional over $K$, then 
    	$\phi\in K$. 
    	
    	2) If $\End^{\operatorname{cont}}_{K[G]}(\Pi)=K$, then $\Pi$ is absolutely irreducible.
    \end{lem}
   
	\begin{proof}
		See \cite[Lemma 4.1 \& 4.2]{Pa_k_nas_2013}.
	\end{proof}
	
	 Let $\operatorname{Mod}^{\mathrm{sm}}_G(\mathcal{O})$ be the category of smooth 
	representations of $G$ on $\mathcal{O}$-torsion modules, and let $\operatorname{Mod}^{\mathrm{l\, fin}}_G(\mathcal{O})$ be the full subcategory 
	of $\operatorname{Mod}^{\mathrm{sm}}_G(\mathcal{O})$ consisting of locally finite representations, which means that for every $v\in \pi$, the smallest $\mathcal{O}[G]$-submodule of  $\pi$ containing  $v$ is of finite length. Let $H$ be a compact open subgroup of $G$ and  let 
	$\operatorname{Mod}^{\mathrm{pro \, aug}}_G(\mathcal{O})$ be the category of profinite $\mathcal{O}[[H]]$-modules with an action 
	of $\mathcal{O}[G]$ such that the two actions are the same when restricted to $\mathcal{O}[H]$. Note that it is actually independent of the choice of $H$. Taking Pontryagin duals (with the discrete topology on $K/\cO$):
	\[\pi \mapsto \pi^\vee:=\Hom^\mathrm{cont}_\cO(\pi,K/\cO)\mbox{ with the compact-open topology}\]
	induces an anti-equivalence of categories between $\mathrm{Mod}_G^{\mathrm{sm}}(\mathcal{O})$ and $\mathrm{Mod}_G^{\mathrm{pro\,aug}}(\mathcal{O})$. There is a natural isomorphism between $\pi^{\vee\vee}$ and $\pi$.
	
	Let $\Pi$ be a unitary $K$-Banach space representation of $G$ and $\Theta$ an open bounded $G$-invariant lattice in $\Pi$. We denote by $\Theta^d$ its Schikhof dual
	$ \Theta^d:=\Hom_{\cO}(\Theta, \cO)$
	equipped with the topology of pointwise convergence.
	
	\begin{lem}
		$\Theta^d$ is an object of $\operatorname{Mod}^{\mathrm{pro \, aug}}_G(\mathcal{O})$. 
	\end{lem}
	
	\begin{proof}
		See \cite[Lemma 4.4]{Pa_k_nas_2013}.
	\end{proof}
	
	Let $\mathrm{Mod}^?_{G}(\cO)$ be  a full subcategory of $\mathrm{Mod}^{\mathrm{l\, fin}}_G(\cO)$ closed under subquotients and 
	arbitrary direct sums in $\mathrm{Mod}^{\mathrm{l\, fin}}_G(\cO)$.  Let $\kC(\cO)$ be the full subcategory of $\mathrm{Mod}_G^{\mathrm{pro\, aug}}(\cO)$ 
	anti-equivalent to $\mathrm{Mod}^?_{G}(\cO)$ via Pontryagin duality. Then $\mathrm{Mod}^?_G(\cO)$ has injective envelopes and 
	so $\kC(\cO)$ has projective envelopes.
	
	\begin{lem} For an  admissible unitary $K$-Banach space representation $\Pi$ of $G$ the following are equivalent:
		\begin{itemize}
			\item[(i)] there exists an open bounded $G$-invariant lattice $\Theta$ in $\Pi$ such that 
			$\Theta^d$ is an object of $\kC(\cO)$;
			\item[(ii)] $\Theta^d$ is an object of $\kC(\cO)$ for every open bounded $G$-invariant lattice $\Theta$ in $\Pi$.
		\end{itemize}
	\end{lem}
	
	\begin{proof}
		See \cite[Lemma 4.6]{Pa_k_nas_2013}.
	\end{proof}
	
	Let $\mathrm{Ban}_G^{\mathrm{adm}}(K)$ be  the category of admissible 
	unitary $K$-Banach space representations of $G$ with morphisms continuous $G$-equivariant $K$-linear homomorphisms. 
	Let $\mathrm{Ban}^{\mathrm{adm}}_{\kC(\cO)}$ be the full subcategory of $\mathrm{Ban}_G^{\mathrm{adm}}(K)$  with objects which are admissible 
	unitary $K$-Banach space representations of $G$ satisfying the conditions of the previous lemma.
	
	Now we assume the following setup. Let $S_1, \ldots, S_n$ be irreducible pairwise non-isomorphic objects of $\kC(\cO)$ such that $\End_{\kC(\cO)}(S_i)$ is finite dimensional over $\mathbb{F}$ 
	for $1\le i \le n$. Let $\widetilde{P}$ be a projective envelope of $S:=\oplus_{i=1}^n S_i$ and let $\widetilde{E}=\End_{\kC(\cO)}(\widetilde{P})$. 
	Then $\widetilde{E}$ is a compact ring (maybe not commutative in general) and $\widetilde{E}/\operatorname{rad} \widetilde{E} \cong \prod_{i=1}^n \End_{\kC(\cO)}(S_i)$, 
	where $\operatorname{rad} \widetilde{E}$ is the Jacobson radical of $\widetilde{E}$. Moreover, we have $\widetilde{P}\cong \oplus_{i=1}^n \widetilde{P}_i$, where $\widetilde{P}_i$ is a projective envelope of $S_i$ in $\kC(\cO)$.  For $1\le i \le n$ let $\pi_i:=S^{\vee}_i$, so that $\pi_i$ is a smooth irreducible $\mathbb{F}$-representation of $G$ and $\pi:=\oplus_{i=1}^n \pi_i \cong S^{\vee}$.
	
	\begin{prop}\label{4.17}
		Let $\Pi$ be in $\mathrm{Ban}^{\mathrm{adm}}_{\kC(\cO)}$ and let $\Theta$ be an open bounded $G$-invariant lattice in $\Pi$. 
		
		1) $\Hom_{\kC(\cO)}(\widetilde{P}, \Theta^d)$ is a finitely generated module over $\widetilde{E}$.
		
		2) Suppose that $\Pi$ is irreducible and $\Theta\otimes_{\cO} \mathbb{F}$ contains $\pi_i$ as a subquotient 
		for some $i$. Let $\phi\in \Hom_{\kC(\cO)}(\widetilde{P}, \Theta^d)$ be non-zero 
		and  let
		$\mathfrak a:=\{a\in \widetilde{E}: \phi\circ a=0\}$. There exists an open bounded $G$-invariant 
		$\cO$-lattice $\Xi$ in $\Pi$ such that $\phi(\widetilde{P})= \Xi^d$.
		Moreover, 
		\begin{itemize}
			\item[(i)] $\Hom_{\kC(\cO)}(\widetilde{P}, \Xi^d)\cong\widetilde{E}/\mathfrak a$ as a right $\widetilde{E}$-module;
			\item[(ii)] $\Hom_{\kC(\cO)}(\widetilde{P}, \Xi^d)\otimes_{\cO} K$ is an irreducible right $\widetilde{E}\otimes_{\cO} K$-module;
			\item[(iii)] the natural map $\Hom_{\kC(\cO)}(\widetilde{P}, \Xi^d)\widehat{\otimes}_{\widetilde{E}} \widetilde{P}\rightarrow \Xi^d$ is surjective.
		\end{itemize}
		Furthermore, we have natural isomorphisms of rings:
		$$ \End_{\kC(\cO)}(\Xi^d)\cong \End_{\widetilde{E}}(\mm)\cong \End_{\kC(\cO)}(\mm\widehat{\otimes}_{\widetilde{E}} \widetilde{P}),$$
		where $\mm:=\Hom_{\kC(\cO)}(\widetilde{P}, \Xi^d)$.
	\end{prop}
	
	\begin{proof}
		See \cite[Proposition 4.17, 4.18 \& 4.19]{Pa_k_nas_2013}.
	\end{proof}
	
	\begin{prop} \label{finite dim}
		Let $\Pi \in \mathrm{Ban}^{\mathrm{adm}}_{\kC(\cO)}$ be irreducible and let $\Theta$ be an open bounded
		$G$-invariant lattice in $\Pi$. Suppose that $\Theta\otimes_{\cO} \mathbb{F}$ contains $\pi_i$ as a subquotient 
		for some $i$.  If the centre $\mathcal Z$ of $\widetilde{E}$ is noetherian and $\widetilde{E}$ is a finitely generated
		$\mathcal Z$-module, then $\Hom_{\kC(\cO)}(\widetilde{P}, \Theta^d)\otimes_{\cO} K$ is finite dimensional over $K$.
	\end{prop} 
	
	\begin{proof}
		See \cite[Proposition 4.20]{Pa_k_nas_2013}.
	\end{proof}
	
	Let $\Pi$ be an object in 
	$\mathrm{Ban}^{\mathrm{adm}}_{\kC(\cO)}$. Choose an open bounded $G$-invariant lattice $\Theta$ in $\Pi$ and 
	put $\mm(\Pi):=\Hom_{\kC(\cO)}(\widetilde{P}, \Theta^d) \otimes_{\cO} K$. Since any two open bounded lattices in $\Pi$ are commensurable (see \cite[Appendix A]{emerton11}), the definition of $\mm(\Pi)$ does not depend on the choice of $\Theta$. Then $\Pi \mapsto \mm(\Pi)$ defines 
	an exact functor from $\mathrm{Ban}^{\mathrm{adm}}_{\kC(\cO)}$ to the category of right $\widetilde{E}[1/p]$-modules (see \cite[Lemma 4.9]{Pa_k_nas_2013}).  Write $\operatorname{Mod}^{\mathrm{fg}}_{\widetilde{E}[1/p]}$ for the category 
	of finitely generated right $\widetilde{E}[1/p]$-modules. 
	
	Let $\mm$ be a compact right $\widetilde{E}$-module, free of finite rank  over $\cO$. Assume that $(\widetilde{E}/\operatorname{rad} \widetilde{E})\widehat{\otimes}_{\widetilde{E}} \widetilde{P}$ is of finite length
	in $\kC(\cO)$ and is a finitely generated $\cO[[H]]$-module, where $\operatorname{rad} \widetilde{E}$ is the Jacobson radical of $\widetilde{E}$. We associate an admissible unitary $K$-Banach space representation of $G$:
	$ \Pi(\mm):= \Hom^{\operatorname{cont}}_{\cO}((\mm\widehat{\otimes}_{\widetilde{E}}\widetilde{P})_{\mathrm{tf}}, K)$
	with the topology induced by the supremum norm. Here $M_\mathrm{tf}$ means the maximal $\cO$-torsion free quotient of the $\cO$-module $M$. From now on, we write $\mm_K := \mm \otimes_{\cO} K$.
	
	Let $\mathrm{Ban}^{\mathrm{adm. fl}}_{\kC(\cO)}$ be the full subcategory of $\mathrm{Ban}^{\mathrm{adm}}_{\kC(\cO)}$ consisting 
	of objects of finite length. Let $\Ker \mm$ be the full subcategory of $\mathrm{Ban}^{\mathrm{adm. fl}}_{\kC(\cO)}$
	consisting of those $\Pi$ such that $\mm(\Pi)=0$. Since $\mm$ is an exact functor, $\Ker \mm$ is a thick subcategory of 
	$\mathrm{Ban}^{\mathrm{adm. fl}}_{\kC(\cO)}$ and hence we can build a quotient category $\mathrm{Ban}^{\mathrm{adm. fl}}_{\kC(\cO)}/ \Ker \mm$.
	
	\begin{thm}  \label{4.34}
		Assume that 
		\begin{itemize} 
			\item[(i)] $(\widetilde{E}/\operatorname{rad} \widetilde{E})\widehat{\otimes}_{\widetilde{E}} \widetilde{P}$ is a finitely generated $\cO[[H]]$-module and is of finite length in $\kC(\cO)$;
			\item[(ii)] For every irreducible $\Pi$  in $\mathrm{Ban}^{\mathrm{adm}}_{\kC(\cO)}$, $\mm(\Pi)$ is finite dimensional.
		\end{itemize} 
		Then the functors $\mm_K \mapsto \Pi(\mm_K)$ and $\Pi\mapsto \mm(\Pi)$ induce an anti-equivalence of categories between 
		$\mathrm{Ban}^{\mathrm{adm. fl}}_{\kC(\cO)}/ \Ker \mm$ and the category of finite dimensional $K$-vector spaces with continuous right $\widetilde{E}$-action.
	\end{thm} 
	
	\begin{proof}
		See \cite[Theorem 4.34]{Pa_k_nas_2013}.
	\end{proof}
	
    Combining Proposition \ref{finite dim}, we also have the following theorem.
    
    \begin{thm} \label{decomposition}
    	Assume that 
    	\begin{itemize} 
    		\item[(i)] $(\widetilde{E}/\operatorname{rad} \widetilde{E})\widehat{\otimes}_{\widetilde{E}} \widetilde{P}$ is a finitely generated $\cO[[H]]$-module and is of finite length in $\kC(\cO)$;
    		\item[(ii)] the centre $\mathcal Z$ of $\widetilde{E}$ is noetherian and $\widetilde{E}$ is a finitely generated $\mathcal Z$-module.
    	\end{itemize} 
    	Then 
    	$$\mathrm{Ban}^{\mathrm{adm. fl}}_{\kC(\cO)}/\Ker \mm  \cong \bigoplus_{\mathfrak{n}\in \operatorname{MaxSpec} \mathcal Z[1/p]} (\mathrm{Ban}^{\mathrm{adm. fl}}_{\kC(\cO)}/\Ker \mm)_{\mathfrak{n}},$$
    	where the direct sum is taken over all the maximal ideals of $\mathcal Z[1/p]$, and  for  a maximal ideal $\mathfrak{n}$ of $\mathcal Z[1/p]$, $(\mathrm{Ban}^{\mathrm{adm. fl}}_{\kC(\cO)}/\Ker \mm)_{\mathfrak{n}}$  
    	is the full subcategory of $\mathrm{Ban}^{\mathrm{adm. fl}}_{\kC(\cO)}/\Ker \mm$, consisting of all Banach spaces which are killed by a power of $\mathfrak{n}$.
    	
    	Further, the functor $\mm \mapsto \Pi(\mm)$ induces an anti-equivalence of categories between the category of 
    	modules of finite length of the $\mathfrak{n}$-adic completion of $\widetilde{E}[1/p]$ and 
    	$(\mathrm{Ban}^{\mathrm{adm. fl}}_{\kC(\cO)}/\Ker \mm)_{\mathfrak{n}}$.
    \end{thm}
	
		\begin{proof}
			See \cite[Theorem 4.36]{Pa_k_nas_2013}.
		\end{proof}
		
		\begin{prop} \label{4.37}
			Keep the hypotheses of Theorem \ref{decomposition} and let $\mathfrak{n}$ be a maximal ideal of $\mathcal Z[1/p]$ and 
			$\mathfrak{n}_0:= \varphi^{-1}(\mathfrak{n})$, where $\varphi: \mathcal Z\rightarrow \mathcal Z[1/p]$. The irreducible objects of 
			$(\mathrm{Ban}^{\mathrm{adm. fl}}_{\kC(\cO)}/\Ker \mm)_{\mathfrak{n}}$ are precisely the irreducible Banach subrepresentations  
			of $\Hom_{\cO}^{\operatorname{cont}}((\widetilde{P}/\mathfrak{n}_0 \widetilde{P})_{\mathrm{tf}}, K)$.
		\end{prop}
		
		\begin{proof}
			See \cite[Proposition 4.37]{Pa_k_nas_2013}.
		\end{proof}
		
		Assume that $\operatorname{Mod}^?_{G}(\cO)$ has only finitely many irreducible objects $\pi_1, \ldots, \pi_n$, which are admissible. 
		Let $\widetilde{P}$ be a projective envelope of $\pi_1^{\vee}\oplus\ldots\oplus \pi_n^{\vee}$ in $\kC(\cO)$, and let $\widetilde{E}=\End_{\widetilde{\kC}(\cO)}(\widetilde{P})$. Then the functor $M\mapsto \Hom_{\kC(\cO)}(\widetilde{P}, M)$ induces an equivalence of categories between 
		$\kC(\cO)$ and the category of compact right $\widetilde{E}$-modules, with the inverse functor given by $\mm\mapsto \mm\widehat{\otimes}_{\widetilde{E}} \widetilde{P}$. This implies that 
		$(\widetilde{E}/\operatorname{rad} \widetilde{E}) \widehat{\otimes}_{\widetilde{E}} \widetilde{P} \cong \pi_1^{\vee}\oplus\ldots\oplus \pi_n^{\vee}$, which is a finitely generated $\cO[[H]]$-module.
		We further assume that  the centre $\mathcal Z$ of $\widetilde{E}$ is noetherian, and $\widetilde{E}$ is a finitely generated module over $\mathcal Z$. 
		
		\begin{prop} \label{faithfully flat}
			The functor $\mm: \operatorname{Ban}^{\mathrm{adm}}_{\kC(\cO)}\rightarrow \operatorname{Mod}^{\mathrm{fg}}_{\widetilde{E}[1/p]}$ is fully faithful.
		\end{prop} 
		
		\begin{proof}
			See \cite[Lemma 4.45]{Pa_k_nas_2013}.
		\end{proof}
		
		\subsection{The case \texorpdfstring{ $G=\GL_2(\mathbb{Q}_p)$}{The case G}}\label{A.2} In this subsection, we recall the case $G=\GL_2(\mathbb{Q}_p)$.
		Keep $p, K, \cO, \varpi, \mathbb{F}$ as in the previous subsection.  Let $Z$ be the centre of $G$ and $B$ be the upper triangular Borel subgroup of $G$. 
		
		Let $\zeta : Z \to \cO^\times$ be a continuous character. A subscript \(\zeta\) indicates the full subcategory of \(G\)-representations on which \(Z\) acts through \(\zeta\). Let $\kC(\cO)$ be the full subcategory of $\mathrm{Mod}_{G, \zeta}^{\mathrm{pro\, aug}}(\cO)$ anti-equivalent to $\operatorname{Mod}^{\mathrm{l\, fin}}_{G, \zeta}(\mathcal{O})$ via Pontryagin duality. 
		
		Let $\operatorname{Irr}_{G, \zeta}$ be the set of irreducible representations in $\operatorname{Mod}^{\operatorname{sm}}_{G, \zeta}(\F)$. We write $\pi \sim \pi'$ if there exists $\pi_1, \cdots, \pi_n \in \Irr_{G, \zeta}$ such that $\pi \cong \pi_1$, $\pi' \cong \pi_n$, and for $1 \leq i \leq n-1$, $\pi_i \cong \pi_{i+1}$ or $\Ext^1_{G}(\pi_i, \pi_{i+1}) \neq 0$ or $\Ext^1_{G}(\pi_{i+1}, \pi_i) \neq 0$. The relation $\sim$ is an equivalence relation on $\Irr_{G, \zeta}$. A block $\kB$ is an equivalence class of $\sim$.
		
		In this setting, the blocks containing an absolutely irreducible representation have been determined (for example, see \cite[page 1303]{Pa_k_nas_2016}). We list them here:
		
		\begin{enumerate}
			\item $\kB = \{ \pi \}$ with $\pi$ supersingular;
			\item $\kB = \{ (\Ind_B^G\chi_1 \otimes \chi_2 \omega^{-1})_{\sm},  (\Ind_B^G\chi_2 \otimes \chi_1 \omega^{-1})_{\sm}\}$ with $\chi_2 \chi_1^{-1} \neq \mathbf{1}, \omega^{\pm 1}$;
			\item $p>2$ and $\kB = \{ (\Ind_B^G \chi \otimes \chi \omega^{-1})_{\sm} \}$;
			\item $p \geq 5$ and $\kB = \{ \mathbf{1}, \Sp, (\Ind^G_B \omega \otimes \omega^{-1})_{\sm} \} \otimes \chi \circ \det$;
			\item $p=3$ and $\kB = \{ \mathbf{1}, \Sp, \omega \circ \det, \Sp \otimes \omega \circ \det \} \otimes \chi \circ \det$;
			\item $p=2$ and $\kB = \{ \mathbf{1}, \Sp \} \otimes \chi \circ \det$;
		\end{enumerate}
		where $\chi, \chi_1, \chi_2: \mathbb{Q}_p^{\times} \rightarrow \F^{\times}$ are smooth characters and $\omega: \mathbb{Q}_p^{\times} \rightarrow \F^{\times}$ is the character $\omega(x) = x |x| \pmod{\varpi}$.

		For each case, we can attach a semi-simple $2$-dimensional representation $\bar{\rho}_{\kB}$ of $G_{\mathbb{Q}_p}$ over $\F$ to each block using the normalized Colmez functor $\mathbf{V}$ (see \cite[Section 5.7]{Pa_k_nas_2013}). This is given by the following list:
		\begin{itemize}
			\item $\bar{\rho}_{\kB}= \mathbf{V}(\pi)$ if $\pi$ is supersingular. In this case, $\bar{\rho}_{\kB}$ is absolutely irreducible.
			\item  $\bar{\rho}_{\kB}=\chi_1 \oplus \chi_2 $ if $\chi_2 \chi_1^{-1} \neq \mathbf{1}, \omega^{\pm 1}$.
			\item $\bar{\rho}_{\kB}= \chi \oplus \chi$ in case (3) and (6).
			\item $\bar{\rho}_{\kB}= \chi \oplus \chi\omega$ in case (4) and (5).
		\end{itemize}
		Under this correspondence, the determinant of $\bar{\rho}_{\kB}$ is $\zeta\varepsilon \mod \varpi$, where $\varepsilon$ is the $p$-adic cyclotomic character.
		
		For a block $\kB$, write $\pi_{\kB}=\bigoplus_{\pi\in\kB_i}\pi$, where $\kB_i$ is the set of isomorphism classes of elements of $\kB$. Let $\pi_\kB\hookrightarrow J_{\kB}$ be an injective envelope of $\pi_{\kB}$ in $\mathrm{Mod}^{\mathrm{l\, adm}}_{G,\zeta}(\cO)$. Its Pontryagin dual $P_{\kB}:=J_{\kB}^\vee$ is a projective envelope of $\pi_{\kB}^\vee\cong\bigoplus_{\pi\in\kB_i}\pi^\vee$ in $\kC_{G,\zeta}(\cO)$, and $P_{\kB}$ is usually called a \textit{projective generator} of the block $\kB$. Let
		$E_{\kB}:=\End_{\kC_{G,\zeta}(\cO)}(P_{\kB})\cong \End_G(J_{\kB}),$
		which is a pseudo-compact ring. Let $Z_{\kB}$ be the centre of $ E_{\kB}$.
		
		By \cite[Corollary 5.35]{Pa_k_nas_2013}, the category $\kC(\cO)$ decomposes into a direct product of subcategories
		\begin{equation*}
			\kC(\cO) \cong \prod_{\kB \in \Irr_{G, \zeta} / \sim} \kC(\cO)_{\kB},
		\end{equation*}
		where the objects of $\kC(\cO)_{\kB}$ are those $M$ in $\kC(\cO)$ such that for every irreducible subquotient $S$ of $M$, $S^{\vee}$ lies in $\kB$. Moreover, the category $\kC(\cO)_{\kB}$ is equivalent to the category of compact right $E_{\kB}$-modules and the centre of $\kC(\cO)_{\kB}$ is isomorphic to $Z_{\kB}$ (see \cite[Proposition 5.45]{Pa_k_nas_2013}).
		
		Let $R^{\ps, \zeta\varepsilon}_{\tr\bar{\rho}_{\kB}}$ be the universal pseudo-deformation ring parametrizing all $2$-dimensional pseudo-representations of $G_{\mathbb{Q}_p}$ lifting $ \tr\bar{\rho}_{\kB}$ with determinant $\zeta\varepsilon$.
		
		\begin{thm}\label{centre finite}
		$E_{\kB}$  and $Z_{\kB}$ are finite over $R^{\ps, \zeta\varepsilon}_{\tr\bar{\rho}_{\kB}}$ and hence noetherian. 
	   \end{thm}
		
		\begin{proof}
			See \cite[Corollary 6.4]{Pa_k_nas_2021}.
		\end{proof}
		
		More precisely, the following result describes the relation between $R^{\ps, \zeta\varepsilon}_{\tr\bar{\rho}_{\kB}}$ and the centre $ Z_{\kB}$.
		
		\begin{thm}\label{uni=cen}
			We have $R^{\ps, \zeta\varepsilon}_{\tr\bar{\rho}_{\kB}} [1/p] \cong Z_{\kB}[1/p]$. If we are not in the case (6), then we have $ R^{\ps, \zeta\varepsilon}_{\tr\bar{\rho}_{\kB}}\cong Z_{\kB}$.
		\end{thm}
		
		\begin{proof}
			For the first part, see \cite[Theorem 1.4]{Pa_k_nas_2021}. For the second part, \cite[Theorem 1.5]{Pa_k_nas_2013} shows the result in the case $p \ge 5$ and the case (2) and (3) for $p \ge 3$. If $p=2$ or $p=3$ and we are in the case (1), it is proved in \cite[Theorem 1.3]{Pa_k_nas_2016}. If $p=3$ and we are in the case (5), \cite[Theorem 1.4]{Pa_k_nas_2021} shows the isomorphism $Z_{\kB} \cong (R^{\ps, \zeta\varepsilon}_{\tr\bar{\rho}_{\kB}})_{\mathrm{tf}} $. Since Proposition \ref{dim of ps} shows that the universal deformation ring $R^{\ps, \zeta\varepsilon}_{\tr\bar{\rho}_{\kB}} $ is $\cO$-torsion free, the result also holds in this case.
		\end{proof}
		
		Since $P_{\kB}$ is a projective generator for $\kC(\cO)_{\kB}$, the functor 
		$$N \mapsto \mm(N):= \Hom_{\kC(\cO)}(P_{\kB}, N)$$
		induces an equivalence of categories between $\kC(\cO)_{\kB}$ and the category of right pseudo-compact $E_{\kB}$-modules. Recall that its inverse functor is given by $\mm \mapsto \mm \widehat{\otimes}_{E_{\kB}} P_{\kB}$. 
		
		\begin{prop}\label{6.7}
			For $N$ in $\kC(\cO)_{\kB}$ the following assertions are equivalent: 
			\begin{enumerate}
				\item there is a surjection $P_\kB^{\oplus n} \twoheadrightarrow N$ for some $n\ge 1$;
				\item $\mm(N)$ is a finitely generated $E_{\kB}$-module;
				\item $\mm(N)$ is a finitely generated $R^{\ps, \zeta\varepsilon}_{\tr\bar{\rho}_{\kB}}$-module;
				\item $\F \widehat{\otimes}_{R^{\ps, \zeta\varepsilon}_{\tr\bar{\rho}_{\kB}}} N$ is of finite length in $\kC(\cO)$;
				\item the cosocle of $N$ in $\kC(\cO)$ 
				is of finite length.
			\end{enumerate} 
			The equivalent conditions  hold if $N$ is finitely generated over $\cO[[H]]$ for a compact open subgroup $H$ of $G$. 
		\end{prop}
		
		\begin{proof}
			See \cite[Corollary 6.7]{Pa_k_nas_2021}.
		\end{proof}
		
		Since every irreducible object in $\kB$ is admissible, its Pontryagin dual is finitely generated over
		$\cO[[H]]$ for any compact open subgroup $H$ of $G$. Using the results listed above, we know that  the assumptions made in Theorem \ref{4.34} and Theorem \ref{decomposition} are 
		satisfied for the category $\kC(\cO)_{\kB}$.
		
		Using \cite[Proposition 5.36]{Pa_k_nas_2013}, we have a decomposition $$
			\Ban_{G, \zeta}^{\adm} (K)\cong \bigoplus_{\kB \in \Irr_{G, \zeta} / \sim} \Ban_{G, \zeta}^{\adm}(K)_{\kB},
		$$
		where the objects of $\Ban_{G, \zeta}^{\adm}(K)_{\kB}$ are those $\Pi$ in $\Ban_{G, \zeta}^{\adm}(K)$ such that for every open bounded $G$-invariant lattice $\Theta$ in $\Pi$ the irreducible subquotients of $\Theta \otimes_{\cO} \F$ lie in $\kB$. It is equivalent to  require $\Theta^d$ to be an object of $\kC(\cO)_{\kB}$.
		
		For a block $\kB$ consisting of absolutely irreducible representations, we write $\Mod^{\mathrm{fl}}_{E_\kB[1/p]}$ for the category of finitely generated right $E_\kB[1/p]$-modules of finite length. By Theorem
		\ref{4.34} and Proposition \ref{faithfully flat}, the functor $\mm$ (defined in the previous subsection) induces an anti-equivalence of categories $$
		\mm:\Ban^{\operatorname{adm. fl}}_{G, \zeta}(K)_{\kB} \overset{\cong}{\longrightarrow} \Mod^{\mathrm{fl}}_{E_\kB[1/p]}.$$
		
		If $\mathfrak{m}$ is a maximal ideal of $R^{\ps, \zeta\varepsilon}_{\tr\bar{\rho}_{\kB}}[1/p]$, then we let $\Ban^{\mathrm{adm. fl}}_{G, \zeta}(K)_{\kB, \mathfrak{m}}$ be the full subcategory of 
		$\Ban^{\adm}_{G, \zeta}(K)$ consisting of finite length Banach space representations, 
		which are killed by some power of $\mathfrak{m}$. By Theorem \ref{centre finite} and Theorem \ref{decomposition}, 
		we have an equivalence of categories 
	$$	\Ban^{\mathrm{adm. fl}}_{G, \zeta}(K)_{\kB} \cong \bigoplus_{\mathfrak{m} \in \mathrm{MaxSpec} R^{\ps, \zeta\varepsilon}_{\tr\bar{\rho}_{\kB}}[1/p]} \Ban^{\mathrm{adm. fl}}_{G, \zeta}(K)_{\kB, \mathfrak{m}}.
	$$
		
		Let $\Irr(\mathfrak{m}, K')$ be the set of isomorphism classes  of irreducible objects in $\Ban^{\mathrm{adm. fl}}_{G, \zeta}(K')_{\kB, \mathfrak{m}}$, where $K'$ is a finite extension of $K$.
		
		\begin{prop}\label{reducible_irr} 
			 Let $x: R^{\ps, \zeta\varepsilon}_{\tr\bar{\rho}_{\kB}} \rightarrow K'$ be an $\cO$-algebra homomorphism. It defines a maximal ideal $ \mathfrak{m}_x$ of $R^{\ps, \zeta\varepsilon}_{\tr\bar{\rho}_{\kB}}[1/p]$. If $T_x= \psi_1+ \psi_2$ for characters $\psi_1, \psi_2: G_{\Q_p} \rightarrow (K')^{\times}$ then one of the following holds: 
			\begin{itemize} 
				\item if $\psi_1\psi_2^{-1}=\mathbf{1}$, then $\Irr(\mathfrak{m}_x, K')= \{ (\Ind_B^G \mathbf{1} \otimes  \varepsilon^{-1})_{\cont}\otimes \psi_1\circ \det\}$.
				\item if $\psi_1 \psi_2^{-1}= \varepsilon^{\pm 1}$, then $\Irr(\mathfrak{m}_x, K')= \{ \mathbf{1}, \widehat{\Sp}, (\Ind_B^G \varepsilon\otimes \varepsilon^{-1})_{\cont}\} \otimes \psi\circ \det$, where $\widehat{\Sp} $ is the universal unitary completion of the Steinberg representation.
				\item if $\psi_1 \psi_2^{-1}\neq \varepsilon^{\pm 1}, \mathbf{1}$, then  
				$$\Irr(\mathfrak{m}_x, K')= \{ (\Ind_B^G \psi_1\otimes \psi_2 \varepsilon^{-1})_{\cont}, (\Ind_B^G \psi_2\otimes \psi_1 \varepsilon^{-1})_{\cont}\}.$$
			\end{itemize}
			Here we consider $\psi_1$ and $\psi_2$ as unitary characters of $\Q_p^{\times}$ via class field theory, and $\psi$ in (2) is chosen so that $\{\psi_1,\psi_2\}=\{\psi,\psi\varepsilon\}$.
		\end{prop} 
		
		\begin{proof}
			See \cite[Corollary 6.10]{Pa_k_nas_2021}.
		\end{proof}

		\begin{prop}\label{irreducible_irr} 
			 Let $x: R^{\ps, \zeta\varepsilon}_{\tr\bar{\rho}_{\kB}} \rightarrow K'$ be an $\cO$-algebra homomorphism. It defines a maximal ideal $ \mathfrak{m}_x$ of $R^{\ps, \zeta\varepsilon}_{\tr\bar{\rho}_{\kB}}[1/p]$. 
			If $T_x =\tr \rho$, where $\rho:G_{\Q_p} \rightarrow \GL_2(K')$ is absolutely irreducible, 
			then $\Irr(\mathfrak{m}_x, K')=\{\Pi\}$ with $\Pi$ absolutely irreducible non-ordinary and $\mathbf{V}(\Pi)\cong \rho$, where $\mathbf{V} $ is the normalized Colmez functor.
		\end{prop}
		
		\begin{proof}
			See \cite[Proposition 6.11]{Pa_k_nas_2021}.
		\end{proof}

\bibliographystyle{alpha}
\bibliography{fontaine-mazur.bib}

\newcommand{\etalchar}[1]{$^{#1}$}
\begin{thebibliography}{BLGHT11}

\bibitem[AC14]{Allen_2014}
Patrick~B. Allen and Frank Calegari.
\newblock Finiteness of unramified deformation rings.
\newblock {\em Algebra \& Number Theory}, 8(9):2263--2272, December 2014.

\bibitem[AT68]{artin}
Emil Artin and John Tate.
\newblock {\em Class field theory}.
\newblock W. A. Benjamin, Inc., New York-Amsterdam, 1968.

\bibitem[BB10]{2010Breuil}
Laurent Berger and Christophe Breuil.
\newblock Sur quelques repr\'esentations potentiellement cristallines de
  $\operatorname{GL}_2(\mathbb{Q}_p)$.
\newblock {\em Ast\'erisque}, (330):155--211, 2010.

\bibitem[BBD{\etalchar{+}}13]{Berger_2013}
Laurent Berger, Gebhard B\"{o}ckle, Lassina Demb{\'e}l{\'e}, Mladen Dimitrov,
  Tim Dokchitser, and John Voight.
\newblock {\em Elliptic Curves, Hilbert Modular Forms and Galois Deformations}.
\newblock Springer Basel, 2013.

\bibitem[BE10]{2010Emerton}
Christophe Breuil and Matthew Emerton.
\newblock Repr\'esentations $p$-adiques ordinaires de
  $\operatorname{GL}_2(\mathbb{Q}_p)$ et compatibilit\'e local-global.
\newblock {\em Ast\'erisque}, (331):255--315, 2010.

\bibitem[Bel12]{bellaiche2012pseudodeformations}
Jo{\"e}l Bella{\"\i}che.
\newblock Pseudodeformations.
\newblock {\em Mathematische Zeitschrift}, 270(3):1163--1180, 2012.

\bibitem[BIP23]{B_ckle_2023}
Gebhard B\"{o}ckle, Ashwin Iyengar, and Vytautas Pa{\v s}k\=unas.
\newblock On local {G}alois deformation rings.
\newblock {\em Forum of Mathematics, Pi}, 11:Paper No. e30, 54, 2023.

\bibitem[BJ23]{Bockle_2023}
Gebhard B\"{o}ckle and Ann-Kristin Juschka.
\newblock Equidimensionality of universal pseudodeformation rings in
  characteristic $p$ for absolute {G}alois groups of $p$-adic fields.
\newblock {\em Forum of Mathematics, Sigma}, 11:Paper No. e102, 83, 2023.

\bibitem[BLGHT11]{2011prims}
Tom Barnet-Lamb, David Geraghty, Michael Harris, and Richard Taylor.
\newblock A family of {C}alabi--{Y}au varieties and potential automorphy ii.
\newblock {\em Publ. Res. Inst. Math. Sci.}, 47(1):29--98, 2011.

\bibitem[B{\"o}c10]{bockle2010deformation}
Gebhard B{\"o}ckle.
\newblock Deformation rings for some mod 3 {G}alois representations of the
  absolute {G}alois group of $\mathbb{Q}_3$.
\newblock {\em Ast{\'e}risque}, 330:529--542, 2010.

\bibitem[CDP14]{Colmez_2014}
Pierre Colmez, Gabriel Dospinescu, and Vytautas Pa{\v s}k\=unas.
\newblock The $p$-adic local {L}anglands correspondence for
  $\operatorname{GL}_2(\mathbb{Q}_p)$.
\newblock {\em Cambridge Journal of Mathematics}, 2(1):1--47, 2014.

\bibitem[Che11]{chenevier2009variete}
Ga\"etan Chenevier.
\newblock Sur la vari\'et\'e des caract\`eres $p$-adiques du groupe de {G}alois
  absolu de $\mathbb{Q}_p$.
\newblock {\em preprint},
  http://gaetan.chenevier.perso.math.cnrs.fr/articles/lieugalois.pdf, 2011.

\bibitem[CHT08]{Clozel_2008}
Laurent Clozel, Michael Harris, and Richard Taylor.
\newblock Automorphy for some $l$-adic lifts of automorphic mod $l$ {G}alois
  representations.
\newblock {\em Publ. Math. Inst. Hautes \'Etudes Sci.}, 108(1):1--181, November
  2008.

\bibitem[CM09]{Calegari_2008}
Frank Calegari and Barry Mazur.
\newblock Nearly ordinary {G}alois deformations over arbitrary number fields.
\newblock {\em Journal of the Institute of Mathematics of Jussieu},
  8(1):99--177, January 2009.

\bibitem[Eme11]{emerton11}
Matthew Emerton.
\newblock Local-global compatibility in the $p$-adic {L}anglands programme for
  $\operatorname{GL}_{2/\mathbb{Q}}$.
\newblock {\em preprint},
  \url{https://math.uchicago.edu/emerton/pdffiles/lg.pdf}, 2011.

\bibitem[Gee22]{Gee_2022}
Toby Gee.
\newblock Modularity lifting theorems.
\newblock {\em Essential Number Theory}, 1(1):73--126, October 2022.

\bibitem[HT15]{hutan15}
Yongquan Hu and Fucheng Tan.
\newblock The {B}reuil-{M}\'ezard conjecture for non-scalar split residual
  representations.
\newblock {\em Ann. Sci. \'Ec. Norm. Sup\'er. (4)}, 48(6):1383--1421, 2015.

\bibitem[Kis03]{kisin2003overconvergent}
Mark Kisin.
\newblock Overconvergent modular forms and the {F}ontaine-{M}azur conjecture.
\newblock {\em Inventiones mathematicae}, 153(2):373--454, 2003.

\bibitem[Kis09]{kisin2009fontaine}
Mark Kisin.
\newblock The {F}ontaine-{M}azur conjecture for $\operatorname{GL}_2$.
\newblock {\em Journal of the American Mathematical Society}, 22(3):641--690,
  2009.

\bibitem[KW09a]{Khare_2009a}
Chandrashekhar Khare and Jean-Pierre Wintenberger.
\newblock Serre's modularity conjecture (i).
\newblock {\em Inventiones mathematicae}, 178(3):485--504, July 2009.

\bibitem[KW09b]{Khare_2009}
Chandrashekhar Khare and Jean-Pierre Wintenberger.
\newblock Serre's modularity conjecture (ii).
\newblock {\em Inventiones mathematicae}, 178(3):505--586, July 2009.

\bibitem[NSW08]{Neukirch_2008}
J{\"u}rgen Neukirch, Alexander Schmidt, and Kay Wingberg.
\newblock {\em Cohomology of Number Fields}.
\newblock Springer Berlin Heidelberg, 2008.

\bibitem[Pan22a]{pan2022fontaine}
Lue Pan.
\newblock The {F}ontaine-{M}azur conjecture in the residually reducible case.
\newblock {\em Journal of the American Mathematical Society}, 35(4):1031--1169,
  2022.

\bibitem[Pan22b]{Pan_2022}
Lue Pan.
\newblock On locally analytic vectors of the completed cohomology of modular
  curves.
\newblock {\em Forum of Mathematics, Pi}, 10, 2022.

\bibitem[Pa{\v s}09]{Pa_k_nas_2009}
Vytautas Pa{\v s}k\=unas.
\newblock On some crystalline representations of
  $\operatorname{GL}_2(\mathbb{Q}_p)$.
\newblock {\em Algebra \& Number Theory}, 3(4):411--421, June 2009.

\bibitem[Pa{\v s}13]{Pa_k_nas_2013}
Vytautas Pa{\v s}k\=unas.
\newblock The image of {C}olmez's {M}ontreal functor.
\newblock {\em Publ. Math. Inst. Hautes \'Etudes Sci.}, 118(1):1--191, January
  2013.

\bibitem[Pa{\v s}16]{Pa_k_nas_2016}
Vytautas Pa{\v s}k\=unas.
\newblock On 2-dimensional 2-adic {G}alois representations of local and global
  fields.
\newblock {\em Algebra \& Number Theory}, 10(6):1301--1358, August 2016.

\bibitem[PT21]{Pa_k_nas_2021}
Vytautas Pa{\v s}k\=unas and Shen-Ning Tung.
\newblock Finiteness properties of the category of mod $p$ representations of
  $\operatorname{GL}_2 (\mathbb {Q}_{p})$.
\newblock {\em Forum of Mathematics, Sigma}, 9:Paper No. e80, 39, 2021.

\bibitem[Rub00]{rubin2000euler}
Karl Rubin.
\newblock {\em Euler systems}.
\newblock Number 147. Princeton University Press, 2000.

\bibitem[{Sta}24]{stacks-project}
The {Stacks Project Authors}.
\newblock \textit{Stacks Project}.
\newblock \url{https://stacks.math.columbia.edu}, 2024.

\bibitem[SW99]{Skinner_1999}
Christopher~M. Skinner and Andrew~J. Wiles.
\newblock Residually reducible representations and modular forms.
\newblock {\em Publ. Math. Inst. Hautes \'Etudes Sci.}, 89(1):6--126, December
  1999.

\bibitem[SW01]{skinnerwiles01}
Christopher~M. Skinner and Andrew~J. Wiles.
\newblock Nearly ordinary deformations of irreducible residual representations.
\newblock {\em Ann. Fac. Sci. Toulouse Math. (6)}, 10(1):185--215, 2001.

\bibitem[Tay08]{taylor2008automorphy}
Richard Taylor.
\newblock Automorphy for some $l$-adic lifts of automorphic mod $l$ {G}alois
  representations. {II}.
\newblock {\em Publ. Math. Inst. Hautes \'Etudes Sci.}, 108(1):183--239, 2008.

\bibitem[Tun21a]{Tung_2021}
Shen-Ning Tung.
\newblock On the automorphy of 2-dimensional potentially semistable deformation
  rings of ${G}_{{\mathbb{Q}}_p}$.
\newblock {\em Algebra \& Number Theory}, 15(9):2173--2194, December 2021.

\bibitem[Tun21b]{Tung_2020}
Shen-Ning Tung.
\newblock On the modularity of 2-adic potentially semi-stable deformation
  rings.
\newblock {\em Mathematische Zeitschrift}, 298:107--159, August 2021.

\bibitem[Was78]{Washington_1978}
Lawrence~C. Washington.
\newblock The non-$p$-part of the class number in a cyclotomic
  $\mathbb{Z}_p$-extension.
\newblock {\em Inventiones Mathematicae}, 49(1):87--97, March 1978.

\bibitem[WE17]{WangErickson2018}
Carl Wang-Erickson.
\newblock Algebraic families of {G}alois representations and potentially
  semi-stable pseudodeformation rings.
\newblock {\em Mathematische Annalen}, 371:1615--1681, July 2017.

\bibitem[Wil15]{wiles2015}
Andrew~J. Wiles.
\newblock On class groups of imaginary quadratic fields.
\newblock {\em J. Lond. Math. Soc.}, 92(2):411--426, 2015.

\bibitem[Zha24]{Zhang2024}
Xinyao Zhang.
\newblock On the pro-modularity in the residually reducible case for some
  totally real fields.
\newblock {\em preprint}, 2024.

\bibitem[Zha26]{ZhaZariskiDensity}
Xinyao Zhang.
\newblock Modular points and dimensions of {E}isenstein deformation spaces.
\newblock {\em preprint}, 2026.

\end{thebibliography}
%\enlargethispage{\baselineskip}
\end{document}